\documentclass[reqno]{amsart}
\usepackage{amssymb,amsmath,amsthm,mathtools}
\usepackage[T1]{fontenc}
\usepackage{mathrsfs}
\usepackage{lmodern}
\usepackage{anyfontsize}
\usepackage{amstext}
\usepackage{amscd}
\usepackage[mathscr]{euscript}
\usepackage{graphicx}
\usepackage{bbm}
\usepackage{bm}
\usepackage{bbold}
\usepackage{microtype}

\usepackage{hyperref}
\usepackage{dsfont}
\usepackage{tikz}
\usepackage{tikz-cd}
\usepackage{quiver}

\usepackage{xspace}
\usepackage{enumitem}

\newtheorem{thm}{Theorem}[section]
\newtheorem{cor}[thm]{Corollary}
\newtheorem{lem}[thm]{Lemma}
\newtheorem{prop}[thm]{Proposition}
\newtheorem{def-lem}[thm]{Definition-Lemma}

\theoremstyle{definition}
\newtheorem{defin}[thm]{Definition}

\newtheorem{warning}[thm]{Warning}
\newtheorem{rem}[thm]{Remark}

\newtheorem{example}[thm]{Example}

\newtheorem{notation}[thm]{Notation}
\newtheorem{recollection}[thm]{Recollection}

\newtheorem{assu}[thm]{Assumption}
\newtheorem{convention}[thm]{Convention}
\newtheorem{construction}[thm]{Construction}
\newtheorem*{acknowledgments}{Acknowledgments}

\numberwithin{equation}{section}
\newcommand{\category}{{$\infty$-cat\-e\-go\-ry}\xspace}
\newcommand{\categories}{$\infty$-cat\-e\-gories\xspace}

\newcommand{\operads}{{$\infty$-op\-er\-ads}\xspace}

\newcommand{\bbE}{\mathbb{E}}

\newcommand{\bbN}{\mathbb{N}}
\newcommand{\bbR}{\mathbb{R}}

\newcommand{\bbS}{\mathbb{S}}

\newcommand{\bbX}{\mathbb{X}}

\newcommand{\bbZ}{\mathbb{Z}}

\newcommand{\calE}{\mathcal{E}}

\newcommand{\calL}{\mathcal{L}}

\newcommand{\calP}{\mathcal{P}}

\newcommand{\calT}{\mathcal{T}}
\newcommand{\calU}{\mathcal{U}}
\newcommand{\calV}{\mathcal{V}}

\newcommand{\frakB}{\mathfrak{B}}

\newcommand{\frakP}{\mathfrak{P}}

\DeclareMathOperator{\Sing}{Sing}
\newcommand{\colim}{\operatorname*{colim}}
\newcommand{\Deltasm}{\Delta_{\rm sm}}
\newcommand{\PShv}{\mathrm{PShv}}
\newcommand{\Shv}{\mathrm{Shv}}
\newcommand{\PrL}{\mathrm{Pr^{L}}}

\newcommand{\PrSt}{\mathrm{Pr^{St}}}

\newcommand{\sh}{\mathrm{sh}}

\newcommand{\TrCor}{\mathrm{TrCor}}

\newcommand{\Tr}{\mathrm{Tr}}

\newcommand{\Cat}{\mathrm{Cat}}

\newcommand{\MU}{\mathrm{MU}}

\newcommand{\KO}{\mathrm{KO}}

\newcommand{\BO}{\mathrm{BO}}
\newcommand{\BOP}{\mathbb{Z}\times\mathrm{BO}}

\newcommand{\Mod}{\mathrm{Mod}}
\newcommand{\LMod}{\mathrm{LMod}}
\newcommand{\RMod}{\mathrm{RMod}}
\newcommand{\Alg}{\mathrm{Alg}}

\newcommand{\CAlg}{\mathrm{CAlg}}
\newcommand{\Fun}{\mathrm{Fun}}

\newcommand{\Map}{\mathrm{Map}}
\newcommand{\Spaces}{\mathscr{S}}

\newcommand{\rlz}[1]{\lvert #1 \rvert}
\newcommand{\Top}{\mathrm{Top}}
\newcommand{\Open}{\mathrm{Open}}
\newcommand{\Vect}{\mathrm{Vect}}
\newcommand{\Perf}{\mathrm{Perf}}

\newcommand{\cperf}{\mathrm{CPerf}}

\newcommand{\Ring}{\mathrm{Ring}}

\newcommand{\TKO}{\widetilde{\mathrm{KO}}}
\newcommand{\K}{\mathrm{K}}
\newcommand{\oko}{\Omega^{\infty}\mathrm{ko}}

\newcommand{\ko}{\mathrm{ko}}
\newcommand{\ku}{\mathrm{ku}}
\newcommand{\rmH}{\mathrm{H}}
\newcommand{\rmB}{\mathrm{B}}

\newcommand{\rmO}{\mathrm{O}}
\newcommand{\rmM}{\mathrm{M}}
\newcommand{\vir}{\mathrm{vir}}
\newcommand{\Free}{\mathrm{Free}}
\newcommand{\Gys}{\mathrm{Gys}}
\newcommand{\GysP}{\mathrm{GysP}}
\newcommand{\Cor}{\mathrm{Cor}}
\newcommand{\bCor}{\mathbf{Cor}}

\newcommand{\bbCat}{\widehat{\mathbf{Cat}}_{\infty}}

\newcommand{\Fin}{\mathrm{Fin}}
\newcommand{\Cov}{\mathrm{Cov}}

\newcommand{\sphere}{\mathbb{S}}
\newcommand{\Pic}{\mathrm{Pic}}
\newcommand{\GL}{\mathrm{GL}}
\newcommand{\BGL}{\mathrm{BGL}}
\newcommand{\LCH}{\mathrm{LCH}}
\newcommand{\PH}{\mathrm{PH}}
\newcommand{\CH}{\mathrm{CH}}
\newcommand{\Un}{\mathrm{Un}}
\newcommand{\CMon}{\mathrm{CMon}}
\newcommand{\CMongp}{\mathrm{CMon}^{\mathrm{gp}}}

\newcommand{\bigcat}{\widehat{\mathrm{Cat}}_{\infty}}
\newcommand{\const}{\mathrm{const}}
\newcommand{\id}{\mathrm{id}}
\newcommand{\ee}[1]{\mathbb{E}_{#1}}
\newcommand{\einf}{\mathbb{E}_{\infty}}
\newcommand{\ex}{\mathrm{ex}}
\newcommand{\hi}{\mathrm{hi}}

\newcommand{\quot}{/\!\!/}

\newcommand{\pr}{\mathrm{pr}}
\newcommand{\ev}{\mathrm{ev}}

\newcommand{\flow}{\mathrm{Flow}}
\newcommand{\fr}{\mathrm{fr}}

\newcommand{\Sp}{\mathrm{Sp}}

\newcommand{\Stab}{\mathrm{Stab}}
\newcommand{\Set}{\mathrm{Set}}

\newcommand{\Th}{\mathrm{Th}}

\newcommand{\unit}{\mathds{1}}
\newcommand{\m}[1]{\mathscr{#1}}

\newcommand{\op}{\mathrm{op}}

\newcommand{\opemb}{\mathrm{opemb}}

\newcommand{\Sm}{\mathrm{Sm}}
\newcommand{\DSm}{\mathrm{DSm}}

\newcommand{\cocone}{\vartriangleright}

\newcommand{\rk}{\operatorname{rk}}

\newcommand{\Hom}{\operatorname{Hom}}

\tikzset{
  trim node/.default=1cm,
  trim node/.style={
    overlay,
    append after command={
      ([xshift={+#1}]\tikzlastnode.north west)
      ([xshift={+-#1}]\tikzlastnode.south east)}},
  down and trim/.default=1cm,
  down and trim/.style={
    yshift=-(\pgfmatrixcurrentcolumn-1)*1.5\baselineskip,
    trim node={#1}},
  downup and trim/.default=1cm,
  downup and trim/.style={
    yshift=iseven(\pgfmatrixcurrentcolumn) ? -1.5\baselineskip : 0pt,
    trim node={#1}},
  -|/.style={to path={-|(\tikztotarget)\tikztonodes}},
  |-/.style={to path={|-(\tikztotarget)\tikztonodes}},
  -| sl/.style={-|, xslant=-1},
  |- sl/.style={|-, xslant= 1},
  center picture/.style={
    trim left=(current bounding box.center),
    trim right=(current bounding box.center)}} 

\title[A homotopy coherent Pontryagin-Thom isomorphism]{A homotopy coherent \\ Pontryagin-Thom isomorphism}
\author{Kenneth Blakey}
\address{Department of Mathematics, MIT, 182 Memorial Drive, Cambridge, MA 02139, U.S.A.} 
\email{kblakey@mit.edu}

\author{Liam Keenan}
\address{Department of Mathematics, Brown University, 151 Thayer St., Providence, RI 02912, U.S.A.} 
\email{liam{\_}keenan@brown.edu}

\usepackage[textwidth=25mm, textsize=tiny]{todonotes}

\begin{document}

\begin{abstract}
Classically, the Pontryagin-Thom isomorphism asserts that the multiplicative cohomology theory given by (structured) geometric cobordism is isomorphic to the cohomology theory determined by an associated Thom spectrum. 
We construct a presentably symmetric monoidal stable \category of homotopy invariant sheaves with transfers on smooth manifolds whose unit is precisely (structured) geometric cobordism.
We show the endomorphism ring of the unit can be canonically identified with the associated $\einf$-Thom ring spectrum, i.e., we provide an $\einf$-lift of the Pontryagin-Thom isomorphism. 
\end{abstract}

\maketitle
\tableofcontents

\section{Introduction} 
Ren\'{e} Thom's insight that the cobordism groups of smooth manifolds can be described in terms of the homotopy groups of what are now called Thom spectra created a bridge between geometry and homotopy theory which has profoundly shaped the past and present of algebraic topology \cite{Tho52,thomQuelquesProprietesGlobales1954}. 
For example, the pioneering work of Novikov \cite{novikov-Thom,novikov-cobordism} and Quillen \cite{quillenFormalGroupLaws1969} clarified the central role played by complex cobordism and the theory of formal groups in the study of stable homotopy theory. 
The ingenious vision of Morava, heavily inspired by both Novikov and Quillen, led to the development of modern chromatic homotopy theory.

Many of the Thom spectra encountered in practice from cobordisms of smooth manifolds (e.g., $\mathrm{MO}$, $\mathrm{MSO}$, $\MU$, etc.) can be regarded as $\einf$-ring spectra in a canonical fashion. 
As Thom spectra arise by forming Thom spaces of a sequence of bundles, these $\einf$-structures, in the view of the authors, are most simply produced by endowing the relevant sequence of bundles with a multiplicative structure.
Even though this construction is quite \textit{homotopical} in nature, the Pontryagin-Thom isomorphism permits a \textit{geometric} interpretation of the graded homotopy ring associated to such an $\einf$-Thom spectrum; the additive structure is given by the disjoint union of smooth manifolds and the multiplicative structure is given by the Cartesian product of smooth manifolds -- both up to the relation of structured cobordism.

It is a long-standing piece of folklore that the disjoint union and Cartesian product of smooth manifolds can directly give rise to the canonical $\einf$-ring structures mentioned above.
The starting point is a description of the infinite loop space associated to a Thom spectrum, following Quinn's thesis \cite{quinn-thesis}.
Roughly, these infinite loop spaces are weakly homotopy equivalent to the geometric realization of a simplicial set whose $n$-simplices are corner transverse maps $M \rightarrow \Delta^{n}$, where $M$ is an $n$-dimensional smooth manifold with corners equipped with a relevant tangential structure, and whose face and degeneracy maps are induced by pullback along $\Delta^{m} \rightarrow \Delta^{n}$.
This construction, in the setting of topological manifolds, was carefully developed by Quinn using the notion of an ``ad theory'' \cite{quinn}. 
Later, using Quinn's framework, Laures and McClure showed that the Quinn model for oriented topological cobordism admits the structure of an $\einf$-ring spectrum which is equivalent, as an $\einf$-ring, to the $\einf$-Thom spectrum $\mathrm{MSTop}$ \cite{laures-mcclure-mult,laures-mcclure-comm}. 
It is the understanding of the authors, based on discussions with Laures and McClure, that similar results have not yet been established for general tangential structures or the case of smooth manifolds.

Following Quillen's universal characterization of the cohomology theory $\MU^{\ast}$ \cite{quillenElementaryProofsResults1971}, we produce, for an arbitrary stable tangential structure, an $\einf$-ring spectrum in the spirit of Quinn above which is canonically equivalent to the associated $\einf$-Thom spectrum. We accomplish this by categorifying the problem, as we now describe.

\subsection{Main results}
Let $\Sp$ denote the presentable stable $\infty$-category of spectra endowed with its usual symmetric monoidal. Fix a map of grouplike $\bbE_\infty$-monoids $\varphi:B\to\BO$. We denote by $\rmM\varphi$ the associated Thom spectrum,
    \begin{equation}
    \rmM\varphi\equiv \colim \big(B \rightarrow \BO \rightarrow {\rm BGL}_1(\bbS) \rightarrow \Sp\big),
    \end{equation}
endowed with its canonical $\einf$-ring structure \cite{antolin-camarenaSimpleUniversalProperty2019}. As is now standard knowledge, there are two \emph{a priori} different multiplicative cohomology theories we may associate to $\varphi$. Let $X$ be a smooth manifold.
\begin{enumerate}
\item\emph{Geometric cobordism:} The cohomology theory, denoted $\Omega^*_\varphi(X)$, consisting of closed smooth manifolds $M$ equipped with proper $\varphi$-structured maps $M\to X$ modulo $\varphi$-structured cobordism. The additive structure is given by disjoint union and the multiplicative structure is given by Cartesian product.
\item\emph{Homotopical cobordism:} The cohomology theory 
    \begin{equation}
    \Omega^{*,{\rm ho}}_\varphi(X)\equiv[\Sigma^\infty_+X,\rmM\varphi]_{-\ast}
    \end{equation}
with the obvious additive and multiplicative structure induced from $\rmM\varphi$.
\end{enumerate}
The celebrated Pontryagin-Thom isomorphism constructs an equivalence of multiplicative cohomology theories, 
    \begin{equation}
    \Omega^*_\varphi(X)\xrightarrow{\sim}\Omega^{*,{\rm ho}}_\varphi(X),
    \end{equation}
via the Pontryagin-Thom construction.
However, these multiplicative cohomology theories possess a great deal more structure.
It is a deep insight of Quillen that $\varphi$-structured homotopical cobordism is initial among lax symmetric monoidal, homotopy invariant functors from smooth manifolds to graded abelian groups which additionally admit transfers (i.e. Gysin/Umkehr maps) along proper $\varphi$-structured maps and satisfy a Beck-Chevalley condition for transverse pullback squares \cite{quillenElementaryProofsResults1971} -- we refer to these as ``$\varphi$-structured Quillen cohomology theories''.
To organize this structure in a homotopy coherent fashion, we produce a symmetric monoidal \category of correspondences, $\TrCor_{\varphi}^{\ex}(\Sm)$, encoding the desired contravariant and covariant functoriality considered by Quillen; due to the absence of pullbacks in the category of smooth manifolds, our construction utilizes the theory of derived smooth manifolds, following Pardon \cite{pardon-DSm}. 

\begin{thm}\label{thm:premain}
    There exists a presentably symmetric monoidal stable \category
    \begin{equation}
    \Gys_{\varphi}^{\bbR}(\Sm;\Sp),
    \end{equation}
    consisting of homotopy invariant sheaves on $\TrCor_{\varphi}^{\ex}(\Sm)$, whose unit object we denote by $\unit^{\Sp}_{\varphi}$. 
    For any smooth manifold, $X$, there is an essentially tautological isomorphism of graded commutative rings 
    \begin{equation}
        \pi_{\ast}\unit^{\Sp}_{\varphi}(X) \cong \Omega^{-\ast}_{\varphi}(X)
    \end{equation}
    which extends to an isomorphism of $\varphi$-structured Quillen cohomology theories. 
\end{thm}

The $\einf$-ring $\rmM\varphi$ determines a contravariant functor from the category of smooth manifolds to $\Sp$ which is canonically lax symmetric monoidal:
\begin{equation}
    F\big(\Sigma_{+}^{\infty}(-),\rmM\varphi\big) : \Sm^{\op} \rightarrow \Sp.
\end{equation}
In light of Quillen's work, it is natural to wonder whether homotopical cobordism and the Pontryagin-Thom isomorphism admit $\einf$-refinements to $\Gys_{\varphi}^{\bbR}(\Sm;\Sp)$, i.e., whether there is a \emph{homotopy coherent Pontryagin-Thom isomorphism}. Our next result shows this is indeed the case.

\begin{thm}\label{thm:main}
    $\varphi$-structured homotopical cobordism refines to an $\einf$-algebra, $\Gamma_{\varphi}$ in $\Gys_{\varphi}^{\bbR}(\Sm;\Sp)$ and there is an essentially unique equivalence of $\einf$-algebras 
    \begin{equation}
        \mathrm{PT}_{\varphi} : \unit^{\Sp}_{\varphi} \xrightarrow{\sim} \Gamma_{\varphi}
    \end{equation}
    inducing an equivalence of $\einf$-ring spectra 
    \begin{equation}
        \mathrm{End}(\unit^{\Sp}_{\varphi}) \xrightarrow{\sim} \rmM\varphi
    \end{equation}
    which agrees with the Pontryagin-Thom isomorphism on homotopy groups.     
\end{thm}

Finally, we have the following recognition result which should be thought of as a cobordism-theoretic version of Dugger's result that a homotopy invariant differential cohomology theory is uniquely determined by a spectrum \cite{dugger}.

\begin{thm}\label{thm:main2}
The functor $\Hom(\unit_{\varphi}^{\Sp},-)$ induces an equivalence of presentably symmetric monoidal stable \categories 
    \begin{equation}
    \Gys^\bbR_\varphi(\Sm;\Sp)\simeq\Mod_{\operatorname{End}(\unit^\Sp_\varphi)}.
    \end{equation}
\end{thm}

Combined with Theorem \ref{thm:main}, this immediately yields the following corollary.

\begin{cor}\label{cor:main}
We have the following equivalence of presentably symmetric monoidal stable $\infty$-categories:
    \begin{equation}
    \Gys^\bbR_\varphi(\Sm;\Sp)\simeq\Mod_{\rmM\varphi}.
    \end{equation}
\end{cor}

\subsection{Other flavors of cobordism}
All of our main results are stated for $\varphi$-structured cobordism; however, there are slight variations one could take on the definition to obtain slight variations of the aforementioned results. The first one is a straightforward extension, and the last one would require considerable work.

\begin{rem}
Given $\varphi:B\to{\rm BO}$ as before, we may construct a new geometric cobordism theory, namely,  ``spherical'' $\varphi$-structured geometric cobordism, by considering $\rmM\varphi$-oriented maps $M\to X$ modulo $\rmM\varphi$-oriented cobordism. E.g., an example of an $\bbS$-oriented closed smooth manifold which is not stably frameable is a K3 surface. The initial draft of the present article contained two misconceptions regarding spherical $\varphi$-structured geometric cobordism, and we would like to thank Noah Porcelli for dispelling them. 
\begin{enumerate}
\item There is indeed a purely homotopical construction of spherical $\varphi$-structured geometric cobordism (more precisely, it is a special case of the geometric cobordism theories already considered in the present article, hence no distinction is necessary): it arises via
    \begin{equation}
    \rmM s\varphi\equiv\colim \Big(\operatorname{fib}\big(\BO\to{\rm BGL}_1(\rmM\varphi)\big)\to\BO\to\BGL_1(\sphere)\to\Sp\Big).
    \end{equation}
\item There is indeed a retraction of $\bbE_\infty$-rings,
    \begin{equation}
    \rmM s\varphi\to\rmM\varphi,
    \end{equation}
via \cite[Corollary 3.17]{antolin-camarenaSimpleUniversalProperty2019};
\end{enumerate}
\end{rem}

\subsubsection{Periodic}
In Section \ref{sec:perfectcomplexandvirtual}, we associate to any perfect complex on a derived smooth manifold a \emph{virtual} bundle; we then furthermore go to a \emph{stable} bundle. Meanwhile, since we consider a map of grouplike $\bbE_\infty$-monoids $\varphi:B\to\BO$ and ask for $\varphi$-structured maps $M\to X$, we are asking for the stable relative tangent bundle to lift to $B$. The point we are trying to make is that we could instead ask for the virtual relative tangent bundle to lift to $B$, where we now consider a map of grouplike $\bbE_\infty$-monoids $B\to\Omega^\infty{\rm ko}$, since our assignment naturally builds a virtual bundle first.\footnote{Here, ${\rm ko}$ denotes the real topological $K$-theory spectrum; in particular, $\Omega^\infty{\rm ko}\simeq\bbZ\times\BO$ is only an $\bbE_1$-splitting, cf. Recollection \ref{recollection:ko}.} In particular, the multiplicative cohomology theory $\Omega^*_\varphi$, $\varphi:B\to\Omega^\infty{\rm ko}$, could possibly be a ``periodic'' geometric cobordism theory.

\begin{example}
If we consider the natural map $\varphi: \Omega^{\infty}\ku \to\Omega^\infty{\rm ko}$, then $\Omega^*_\varphi$ is periodic complex cobordism. Moreover, 
    \begin{equation}
    {\rm MUP}\equiv\colim \big(\Omega^\infty\ku \rightarrow \Omega^\infty{\rm ko} \rightarrow \Pic(\bbS) \rightarrow \Sp\big)
    \end{equation}
is the periodic complex cobordism spectrum.
\end{example}

The upshot is that, essentially by repeating the present article using virtual bundles instead of stable bundles, one can obtain a version of Theorem \ref{thm:main} relating periodic geometric cobordism to periodic homotopical cobordism as $\bbE_\infty$-rings.

\subsubsection{(Derived) orbifold}
Periodic geometric cobordism is more or less a straightforward extension of our work here. However, a major technical undertaking would be to extend the ideas of the present article to accommodate replacing (derived) smooth manifolds with (derived) orbifolds, i.e., relating $\varphi$-structured geometric (derived) orbifold cobordism to $\varphi$-structured homotopical (derived) orbifold cobordism. 

\subsection{Relationship to Floer homotopy theory}
As already mentioned, although the roots of cobordism theory are inherently geometric, the homotopical counterpart has historically seen the most development. However, the recent resurgence of Floer homotopy theory, a field at the intersection of symplectic geometry and stable homotopy theory, has fostered interest in (1) developing the geometric story and (2) relating the geometric to the homotopical. Briefly, Floer homotopy theory is an idea, originally due to Cohen-Jones-Segal \cite{CJS95}, which lifts various Floer-type (co)homology theories to stable homotopy theory. The idea is to input a geometric setup in symplectic geometry (e.g., a symplectic manifold, Lagrangian submanifolds, etc.), together with a homotopical assumption (e.g., a stable $\bbR$-polarization, Maslov data, etc.), and output a spectrum,\footnote{In fact, one can refine the Fukaya category, an $A_\infty$-category linear over a discrete ring, to the spectral Fukaya category, an $A_\infty$-category linear over a ring spectrum. (In general, the $A_\infty$-structure is curved, but we digress.)} referred to as a ``Floer homotopy type'', which refines the original (co)homological invariant.\footnote{A Floer homotopy type always canonically exists over some version of ${\rm MU}$ (e.g., ${\rm MU}$ itself, or ${\rm MUP}$, or its derived orbifold variant). Trying to build a Floer homotopy type as a module over a different ring spectra is what requires the homotopical assumption. Moreover, we should note that the canonical ${\rm MU}$ Floer homotopy type may not be related to, say, an $\bbS$ Floer homotopy type, even when both exist, cf. \cite{BP26}.} 

Of course, one has to actually construct an apparatus which turns geometric data into algebraic data, as follows.

\subsubsection{Abouzaid-Blumberg framework}
The original apparatus is due to Cohen-Jones-Segal \cite{CJS95}, however, we will first explain the recent Abouzaid-Blumberg \cite{AB24} framework and the relationship of the present article to \emph{loc. cit.} 

The main gadget is a flow category $\bbX$, and flow categories may be $\varphi$-structured. \cite[Theorem 1.6]{AB24} shows $\varphi$-structured flow categories are the objects of a stable $\infty$-category, denoted $\flow^\varphi$, which is equivalent to the stable $\infty$-category of modules over endomorphisms of the unit $\varphi$-structured flow category $\unit^\varphi$. We will denote by $\flow^\varphi_\Sm$ the stable $\infty$-category of $\varphi$-structured flow categories whose $n$-simplices are enriched in smooth manifolds with corners. Analogously, we will denote by $\flow^\varphi_{\rm dOrb}$ the stable $\infty$-category of $\varphi$-structured flow categories whose $n$-simplces are enriched in (orbifold) global Kuranishi charts with corners. We drop the subscript when we refer to either.

So far, the following is in the literature. As already mentioned, \cite[Theorem 1.6]{AB24} shows $\operatorname{End}(\unit^\varphi)$ is an $\bbE_1$-ring spectrum and $\flow^\varphi$ is $\Mod_{\operatorname{End}(\unit^\varphi)}$. \cite[Proposition 1.10]{AB24} shows ${\rm End}(\unit^\fr_\Sm)$, where the superscript refers to framed flow categories, is equivalent to $\bbS$ as $\bbE_1$-rings. Moreover, recent work of Hedenlund-Oldervoll \cite[Theorem 0.0.6]{HO26}, building on Abouzaid-Blumberg's idea, shows that ${\rm End}(\unit^\varphi_\Sm)$ is equivalent to $\rmM\varphi$ as $\bbE_1$-rings.

Now, here is what is expected. $\flow^\varphi$ should be a presentably symmetric monoidal stable $\infty$-category; in particular, $\operatorname{End}(\unit^\varphi)$ has a natural $\bbE_\infty$-ring structure. Moreover, $\operatorname{End}(\unit^\varphi)$ should be equivalent to the appropriate Thom spectrum as $\bbE_\infty$-rings.

Finally, here are two ideas on how to bridge the known to the expected. It is essentially known how to endow $\flow^\varphi$ with the structure of a presentably symmetric monoidal stable $\infty$-category; so we will focus on the latter expectation with the former as a jumping off point. We will stick to the case of $\Sm$ since the extension of the present article to the (derived) orbifold case is conjectural; however, points (1) and (2) below should also adapt to the case of ${\rm dOrb}$.

\begin{enumerate}
\item First, one should be able to directly identify ${\rm End}(\unit^\varphi_\Sm)$ with ${\rm End}(\unit^\Sp_\varphi)$ as $\bbE_\infty$-rings for the following reason. As already mentioned, there is an explicit point-set model, originally due to Quinn \cite{quinn}, for an infinite loop space related to geometric cobordism obtained via taking the geometric realization of an explicit simplicial set. Now, $\Omega^\infty{\rm End}(\unit^\varphi_\Sm)$ is more or less tautologically the Quinn model. Meanwhile, $\Omega^\infty{\rm End}(\unit^\Sp_\varphi)$ is a slight perturbation of the Quinn model -- it is obtained via taking the geometric realization of a simplicial \emph{space}. In particular, the simplicial set whose geometric realization is $\Omega^\infty{\rm End}(\unit^\varphi_\Sm)$, when endowed with the levelwise discrete topology, is \emph{not} the simplicial space whose geometric realization is $\Omega^\infty{\rm End}(\unit^\Sp_\varphi)$; the latter simplicial space, levelwise, has a slightly more involved topology due to the fact that the category of $\varphi$-structured correspondences we consider is an $\infty$-category, hence, already spatially-enriched. Nevertheless, the simplicial spaces are similar enough, and the $\bbE_\infty$-ring structures both geometric enough, that a direct comparison seems feasible. 
\item Second, there should be an explicit equivalence of presentably symmetric monoidal stable $\infty$-categories
    \begin{equation}
    F:\flow^\varphi_\Sm\xrightarrow{\sim}\Gys^\bbR_\varphi(\Sm;\Sp).
    \end{equation}
Of course, this would imply the first point; however, constructing a symmetric monoidal functor of stable $\infty$-categories would allow one to avoid \emph{directly} comparing the two units. Instead, a symmetric monoidal functor would have to send the unit to the unit, and it would remain to check it is an equivalence. The functor should essentially be obtained by recognizing that any morphism space in a $\varphi$-structured flow $n$-simplex $\bbX$ is simply a correspondence between two Euclidean spaces where the right leg can be taken to be $\varphi$-structured. This would yield a diagram of $n$-simplices in $\Gys^\bbR_\varphi(\Sm;\Sp)$, and the colimit should be $F(\bbX)$.
\end{enumerate}

\subsubsection{Cohen-Jones-Segal framework}
As already mentioned, Cohen-Jones-Segal \cite{CJS95} constructed the original apparatus which yielded a spectrum from a flow category. However, the construction was a bit too clunky to warrant further development since it involved building a CW-spectrum cell-by-cell via repeated applications of the Pontryagin-Thom construction. I.e., the Cohen-Jones-Segal construction is ``married'' to the Pontraygin-Thom construction. However, one could reasonably interpret $\Gys^\bbR_\varphi(\Sm;\Sp)$ as the natural presentably symmetric monoidal stable $\infty$-category where a $\varphi$-structured flow category should live via interpreting a $\varphi$-structured flow category as the homotopy colimit of a diagram of homotopy invariant Yoneda sheaves. In particular, one should view the image of a $\varphi$-structured flow category under the equivalence of Corollary \ref{cor:main} as the Cohen-Jones-Segal construction.

\subsection{Relationship to other work}
The more geometrically-minded reader may wonder why we have chosen to prove a homotopy coherent Pontraygin-Thom isomorphism in this way. As already mentioned, it is long-standing folklore that disjoint union and Cartesian product can directly yield a canonical $\bbE_\infty$-ring lift of geometric cobordism. Therefore, one expects to be able to directly build an $\bbE_\infty$-ring structure using these operations by hand; moreover, one expects to be able to directly show the Pontryagin-Thom construction respects the coherence by hand. One can view this as a ``bottom up'' approach. In the present article, we take a ``top down'' approach by building a significant amount of structure, essentially categorifying \cite{quillenElementaryProofsResults1971}, in order to get both the $\bbE_\infty$-ring lift of geometric cobordism and the Pontraygin-Thom construction ``for free'', i.e., via universal properties. The present authors take the latter approach as they believed the former is too difficult to carry out -- meanwhile, the former approach is taken in upcoming work of Abouzaid--Bai. (In fact, \emph{loc. cit.} will additionally include a proof of the derived orbifold case; the present authors plan to investigate generalizing the techniques of the present article to obtain the derived orbifold case as well.) Of course, both approaches will yield the same $\bbE_\infty$-lifts by Pontryagin-Thom itself, but one can also show this directly by observing Abouzaid-Bai's lift is an $\bbE_\infty$-algebra in $\Gys^\bbR_{\varphi}(\Sm;\Sp)$. The present authors also believe the present article's approach, although requiring a bit more background to set up, has additional offerings as it provides more structure (through e.g. Theorems \ref{thm:premain} and \ref{thm:main2}) and a new perspective on Floer homotopy types (in particular, the Cohen-Jones-Segal construction).

Also, the cobordism categories featured in the Baez--Dolan cobordism hypothesis \cite{lurie2009classificationtopologicalfieldtheories} and the work of Galatius--Madsen--Tillmann--Weiss \cite{GMTW} seem closely related to the Quinn model for geometric cobordism. More specifically, if $\mathrm{Cob}_{\varphi}$ denotes an $(\infty,\infty)$-category of stable $\varphi$-structured cobordisms, there is an expected equivalence $\rmB\mathrm{Cob}_{\varphi} \simeq \Omega^{\infty}\rmM\varphi$. To produce such an equivalence, it seems simplest to directly compare $\rmB\mathrm{Cob}_{\varphi}$ with $\Omega^{\infty}\mathrm{End}(\unit^{\Sp}_{\varphi})$ (as they have seemingly identical ``formulae'') and then appeal to Theorem~\ref{thm:main}. The authors are grateful to Calle and Chan for discussions around this point.

\subsection{Outline}
As the present article is quite lengthy, we briefly outline its contents.

In Section \ref{sec:dsm}, we review the theory of derived smooth manifolds following Pardon's treatment \cite{pardon-DSm}. The notion of a topological site is recalled, as is the notion of a derived site -- both of which are key for defining the \category of derived smooth manifolds. 
Afterwards, the amplitude of a map between derived smooth manifolds is defined, and we introduce the notion of a standard presentation of a map of derived smooth manifolds. The section concludes by establishing some elementary properties of standard presentations.

In Section \ref{sec:perfectcomplexandvirtual}, we begin by introducing the notion of a perfect complex on a derived smooth manifold $X$, the chief example of which is the tangent complex; such perfect complexes can be organized into a stable \category $\Perf_{X}$. The remainder of the section is devoted to producing a natural map from algebraic K-theory of perfect complexes on $X$ to the real topological K-theory of the underlying space of $X$, for $X$ paracompact Hausdorff. To accomplish this, we introduce the notion of a continuous perfect complex on a topological space; these can also be organized into a stable \category $\cperf_X$. We use the algebraic K-theory of continuous perfect complexes to interpolate between $\K(\Perf_{X})$ and $\KO(X)$. 

We begin Section \ref{sec:structurefun} by reviewing the basics of correspondence \categories and abstract 6-functor formalisms. Following Dold's treatment of geometric cobordism, we introduce the notion of a structure functor $F$ on an \category $\m{C}$; this allows us to ``twist'' the morphisms in $\m{C}$ by the functor $F$. In the case where $F$ is related to an abstract 6-functor formalism $D$, we show that $D$ can also be twisted by $F$ to produce a new 6-functor formalism $D_{F}$. 

Section \ref{sec:correspondences} begins by reviewing the basics of Thom spectra and orientation theory following \cite{ABGHR}. 
The results of Section 4 are applied to construct an \category of tangentially structured correspondences between derived smooth manifolds, and an associated 6-functor formalism.
Using this 6-functor formalism, we construct the already mentioned \category $\TrCor^{\ex}_{\varphi}(\Sm)$ and establish its basic properties. 

Section \ref{sec:gysin} is concerned with $\Gys_{\varphi}^{\bbR}(\Sm;\Sp)$, the \category of homotopy invariant spectral sheaves on $\TrCor^{\ex}_{\varphi}(\Sm)$. We prove that $\Gys_{\varphi}^{\bbR}(\Sm;\Sp)$ is stable, presentably symmetric monoidal, and equivalent to modules over the endomorphism $\einf$-ring spectrum of $\unit^{\Sp}_{\varphi}$. We then prove that the unit has levelwise homotopy groups tautologically isomorphic to $\varphi$-structured geometric cobordism. 

In Section \ref{sec: coherent PT}, we prove that $\rmM\varphi$ admits a canonical lift to an $\einf$-algebra $\Gamma_{\varphi}$ in $\Gys_{\varphi}^{\bbR}(\Sm;\Sp)$, thereby receiving an essentially unique $\einf$-map from $\unit^{\Sp}_{\varphi}$. Using a straightforward generalization of \cite[Proposition 1.10]{quillenElementaryProofsResults1971} (which relies on the classical Pontryagin-Thom isomorphism), we show this essentially unique map is an equivalence of $\varphi$-Gysin sheaves and agrees with the Pontryagin-Thom isomorphism on homotopy groups. 

\subsection{Notation and terminology}
We freely use the language of \categories and higher algebra throughout this article, following \cite{lurieHigherToposTheory2009}, \cite{lurieHigherAlgebra2017}, and \cite{kerodon}. 
We consider \categories of varying sizes throughout this article, following the set-theoretic conventions in \cite{lurieHigherToposTheory2009} and \cite{kerodon}. 
Throughout the present article, we make repeated use the following notation. 
\begin{itemize}
    \item $\Cat_{\infty}$ denotes the large \category of small \categories. 
    \item $\Spaces$ denotes the large \category of small spaces/$\infty$-groupoids. 
    \item $\widehat{\Cat}_{\infty}$ denotes the very large \category of large \categories. 
    \item $\CAlg(\m{C})$ denotes the \category of $\einf$-algebras in a symmetric monoidal \category $\m{C}$.
    \item $\m{C}^{\simeq}$ denotes the groupoid core of an \category $\m{C}$. 
    \item $\Pic(\m{C}) \subset \m{C}^{\simeq}$ denotes the full subgroupoid spanned by invertible objects in a symmetric monoidal \category $\m{C}$.  
    \item $\PrL$ denotes the very large \category of presentable \categories and left adjoint functors; this is a closed symmetric monoidal $\infty$-category with respect to the Lurie tensor product. 
    \item $\PrSt$ denotes the full subcategory of $\PrL$ spanned by the presentable stable \categories; this is a closed symmetric monoidal subcategory of $\PrL$.
    \item $\Sp$ denotes the \category of spectra, and $F(X,Y)$ denotes the internal hom.
    \item $\Mod_{R}$ denotes the \category of $R$-module spectra for $R \in \CAlg(\Sp)$. 
    \item $\CMon$ denotes the \category of $\einf$-monoids, i.e., $\einf$-algebras in the \category of $\Sp$ with respect to the Cartesian symmetric monoidal structure. 
    \item $\CMongp \subset \CMon$ denotes the full subcategory spanned by those $\einf$-monoids which are group-complete (alias grouplike).
    \item $\int^{c}_{S^{\op}} F$ and $\int^{cc}_{S}F$ denote the Cartesian and coCartesian unstraightening of a $F : S \rightarrow \Cat_{\infty}$, respectively. 
    \item $\m{P}_{\Sigma}(\m{C})$ denotes the nonabelian derived \category of $\m{C}$, given by product-preserving presheaves on $\m{C}$. 
\end{itemize} 

\begin{acknowledgments}
The authors are especially grateful to John Pardon for patiently explaining his approach to derived smooth manifolds and for his interest in our work. The authors would like to thank Maxine Calle, Tyler Lawson, and John Pardon for feedback on an early draft. Finally, the present work also benefited from correspondence with Thomas Blom, David Chan, Tom Goodwillie, Rune Haugseng, Gerd Laures, Jim McClure, Haynes Miller, Noah Porcelli, and Marco Volpe. The first author was partially supported by an NSF Graduate Research Fellowship award during this work.
\end{acknowledgments}

\section{Derived smooth manifolds, after Pardon}\label{sec:dsm}
\subsection{(Derived) topological sites}
We begin by briefly introducing one of the main players: derived smooth manifolds. We denote by $\Top$ the $1$-category of topological spaces and continuous functions.

\begin{defin}
Let $\m{C}$ be an $\infty$-category equipped with a functor 
    \begin{equation}
    \rlz{-}:\m{C}\to\Top.
    \end{equation}
\begin{itemize}
\item An \emph{open embedding} in $\m{C}$ is a morphism $U\to X\in\m{C}$ which is Cartesian over an open embedding $\rlz{U}\to\rlz{X}\in\Top$. 
\item An \emph{open covering} of $X\in\m{C}$ is a collection of open embeddings $\{U_i\to X\}_{i\in I}$ such that $\big\{\rlz{U_i}\to\rlz{X}\big\}_{i\in I}$ is an open covering of $\rlz{X}$.
\end{itemize}
\end{defin}

\begin{defin}
A \emph{topological site} is a pair $\big(\m{C},\rlz{-}\big)$, where $\m{C}$ is an $\infty$-category and 
    \begin{equation}
    \rlz{-}:\m{C}\to\Top
    \end{equation}
is a functor, called the \emph{realization}, which has all open embeddings, i.e., for every $X\in\m{C}$, every open subset $\rlz{U}\subset\rlz{X}$ is realized by an open embedding $U\to X$.
\end{defin}

\begin{example}
Let $\Sm$ be the $1$-category whose objects are smooth manifolds (assumed second countable and Hausdorff) and whose morphisms are smooth maps; the forgetful functor $\Sm \rightarrow \Top$ determines a topological site.
\end{example}

\begin{rem}
It is straightforward to see that open embeddings endow $\m{C}$ with a Grothendieck topology in which a sieve $S \subset \m{C}_{/X}$ is covering if and only if it is generated by an open covering $\{U_{i} \rightarrow X\}_{i \in I}$.
We will implicitly use this Grothendieck topology to consider sheaves on a topological site. 
\end{rem}

So far, morphisms in topological sites are not of a ``local nature''. The following definition, which most topological sites of interest satisfy, remedies this.

\begin{defin} 
A topological site $\m{C}$ is \emph{subcanonical} if, for every $X \in \m{C}$, the presheaf $\Map_{\m{C}}(-,X) : \m{C}^{\op} \rightarrow \Spaces$ is a sheaf. 
\end{defin}

Let $\m{C}^\opemb\subset\m{C}$ be the wide subcategory where we only allow morphisms which are open embeddings; we analogously define $\Top^\opemb\subset\Top$.

\begin{defin}
Let $\m{C}$ be a subcanonical topological site. We say $\m{C}$ is \emph{perfect} if every lifting problem of the form
    \begin{equation}
    \begin{tikzcd}
    S\arrow[r]\arrow[d] & \m{C}^\opemb\arrow[d] \\
    S^{\cocone}\arrow[r]\arrow[ur,dashed] & \Top^\opemb
    \end{tikzcd}
    \end{equation}
has a solution whenever $S^{\cocone}\to\Top^\opemb$ is a covering sieve.
\end{defin}

In fact, given any topological site, it admits a ``perfection'' satisfying a suitable universal property. To precisely formulate this, we need to discuss functors between topological sites.  

\begin{defin}
Let $\m{C}$ and $\m{D}$ be topological sites. A \emph{topological functor} from $\m{C}$ to $\m{D}$ is a pair $(F,\pi)$, where 
    \begin{equation}
    F:\m{C}\to\m{D}
    \end{equation}
is a functor which preserves open embeddings and 
    \begin{equation}
    \pi:\rlz{F(-)}_{\m{D}}\to\rlz{-}_{\m{C}}
    \end{equation}
is a natural transformation sending open embeddings to pullbacks. We say $(F,\pi)$ is \emph{strict} if $\pi$ is a natural isomorphism.
\end{defin}

\begin{defin}
Let $(F,\pi):\m{C}\to\m{D}$ be a strict topological functor. Given any $X\in\m{C}$, we may consider the induced map of presheaves 
    \begin{equation}
    \Map_{\m{C}}(-,X)\to\Map_{\m{D}}\big(F(-),F(X)\big).
    \end{equation}
We say $(F,\pi)$ is \emph{topologically fully faithful} if, for every $X\in\m{C}$, the aforementioned induced map of presheaves is an isomorphism after sheafification.
\end{defin}

\begin{defin}
Let $(F,\pi):\m{C}\to\m{D}$ be a topological functor. We say $(F,\pi)$ is \emph{topologically essentially surjective} if every object in $\m{D}$ admits an open cover by objects in the image of $F$.
\end{defin}

\begin{prop}[Theorem 2.8.60 in \cite{pardon-DSm}]
Let $\m{C}$ be a topological site. There exists a topologically fully faithful and topologically essentially surjective strict topological functor 
    \begin{equation}
    (F,\pi):\m{C}\to\m{C}^\frakP,
    \end{equation}
where $\m{C}^\frakP$ is perfect, and satisfies the following universal property: for any perfect topological site $\m{D}$, precomposition by $(F,\pi)$ induces an equivalence of $\infty$-categories between topological functors from $\m{C}^\frakP$ to $\m{D}$ and topological functors from $\m{C}$ to $\m{D}$.
\end{prop}


We will need the following notion shortly.

\begin{defin}
Let $K$ be a simplicial set. We say $K$ is \emph{cosifted} if it is non-empty and its diagonal is final. A \emph{cosifted limit} in $\m{C}$ is a limit over a cosifted diagram.
\end{defin}

\begin{rem}
    Unfortunately, a ``final functor'' in the language of \cite{lurieHigherToposTheory2009} is a ``cofinal functor'' in the languge of \cite{pardon-DSm}; we adopt the former convention as it is more standard in the literature.
\end{rem}

Now, given a perfect topological site which admits finite products $\m{C}$, we may associate a derived (topological) site ${\rm D}\m{C}$ by ``formally adjoining'' (1) finite cosifted limits and (2) finite limits modulo finite products -- the derived site satisfies a suitable universal property, cf. \cite[Theorem 2.8.88]{pardon-DSm}, which we now recall in the case of $\Sm$. 

\begin{defin}[Definition 2.9.3 in \cite{pardon-DSm}]
The topological site $\DSm$ of \textit{derived smooth manifolds} together with the functor $\Sm \rightarrow \DSm$ is defined by the following properties.
\begin{itemize}
\item $\Sm\to\DSm$ is a strict topological functor between perfect topological sites.
\item $\Sm\to\DSm$ is fully faithful and preserves finite products.
\item $\DSm$ admits all finite limits; moreover, every $X\in\DSm$ is locally isomorphic to a finite limit of smooth manifolds.
\item $\rlz{-}:\DSm\to\Top$ preserves finite limits.
\item For any $Y\in\Sm$, the Yoneda sheaf $\Map_{\DSm}(-,Y)\in\Shv(\DSm)$ topologically preserves finite cosifted limits (cf. \cite[Definition 2.9.2]{pardon-DSm}).
\item For any complete perfect topological site $\m{D}$, precomposition with $\Sm\to\DSm$ induces an equivalence of $\infty$-categories between topological functors $\DSm$ to $\m{D}$ preserving finite cosifted limits and topological functors $\Sm$ to $\m{D}$.
\end{itemize}
\end{defin}

\begin{rem}
In fact, the functor $\Sm\to\DSm$ preserves all finite transverse limits (in particular, transverse pullbacks), cf. \cite[Definition 2.9.8 \& Corollary 2.9.22]{pardon-DSm}.
\end{rem}

\begin{rem}\label{rem: point-set topology of DSm}
In fact, every $X\in\DSm$ is locally isomorphic to the limit of a truncated cosimplical smooth manifold, cf. \cite[Lemma 2.9.11]{pardon-DSm}.
As a consequence, $\rlz{X}$ is locally homeomorphic to a subspace of a smooth manifold \cite[2.9.12]{pardon-DSm}. 
Consequently, $\rlz{X}$ is locally metrizable.
\end{rem}

\subsection{Amplitude}\label{subsec:amplitude}
The following is a reformulation of Pardon's notion of amplitude at most $1$ maps \cite[Definition 2.9.23]{pardon-DSm}; this is possible by Pardon's minimal amplitude factorization result \cite[Proposition 2.9.32]{pardon-DSm}.

\begin{defin}
Let $f:X\to Y\in\DSm$.
\begin{itemize}
\item $f$ is said to have \emph{amplitude $0$} (or is a \emph{submersion}) if, for every $x\in X$, there exists an open neighborhood $U\ni x$ such that $f\vert_U$ is of the form
    \begin{equation}\label{eqn:submersion}
    \begin{tikzcd}
    U=\bbR^k\times Y\arrow[r,"f\vert_U"]\arrow[d]\arrow[dr, phantom, "\lrcorner", very near start] & Y\arrow[d] \\
    \bbR^k\arrow[r] & *.
    \end{tikzcd}
    \end{equation}

\item $f$ is said to have \emph{amplitude $1$} (or is an \emph{immersion}) if, for every $x\in X$, there exists an open neighborhood $U\ni x$ such that $f\vert_U$ is of the form
    \begin{equation}\label{eqn:immersion}
    \begin{tikzcd}
    U=*\times_{\bbR^k}Y\arrow[r,"f\vert_U"]\arrow[d]\arrow[dr, phantom, "\lrcorner", very near start] & Y\arrow[d] \\
    *\arrow[r] & \bbR^k.
    \end{tikzcd}
    \end{equation}
\item $f$ is said to have \emph{amplitude at most $1$} if, for every $x\in X$, there exists an open neighborhood $U\ni x$ such that $f\vert_{U}$ is equivalent to an immersion followed by a submersion.
\end{itemize}
\end{defin}

\begin{lem}
Being an immersion, submersion, or amplitude at most 1 is preserved under composition and pullback.
\end{lem}

\begin{proof}
Cf. \cite[Exercise 2.9.24]{pardon-DSm}.
\end{proof}\begin{rem}
We say $X\in\DSm$ is \emph{amplitude} 0, \emph{amplitude} 1, or \emph{amplitude at most $1$} if $X\to*$ is, respectively. Observe, the subcategory of amplitude 0 derived smooth manifolds with amplitude at most 1 maps between them is precisely $\Sm$. An amplitude at most 1 derived smooth manifold is also known in the literature as a (non-orbifold) \emph{global Kuranishi chart}.
\end{rem}

\begin{defin}
Let $f:X\to Y\in\DSm$. 
We say \textit{$f$ is of standard presentation} if there exists a diagram in $\DSm$ of the form
    \begin{equation}
    \begin{tikzcd}
    X=*\times_{\bbR^k}\big(Y\times\bbR^N\big)\arrow[r]\arrow[d]\arrow[rr,bend left,"f"]\arrow[dr, phantom, "\lrcorner", very near start] & Y\times\bbR^N\arrow[r]\arrow[d] & Y \\
    * \arrow[r] & \bbR^k, &
    \end{tikzcd}
    \end{equation}
where the second morphism in the first row is the obvious projection.
We call such a diagram a \textit{standard presentation of $f$}.
We say \textit{$f$ is locally of standard presentation} if, locally on the source, $f$ is of standard presentation.
By a \textit{local standard presentation of $f$}, we mean an open subset $\rlz{U}\subset\rlz{X}$ and a standard presentation of $f|_{U}$.
\end{defin}

\begin{lem}
If $f:X\to Y\in\DSm$ is amplitude at most 1, then $f$ is locally of standard presentation.
\end{lem}

\begin{proof}
By definition, $f$ is locally of the form
    \begin{equation}
    U\hookrightarrow\widetilde{U}\to V,
    \end{equation}
where the first morphism is an immersion and the second morphism is a submersion. Since the first morphism is an immersion, by shrinking $U$, we may assume the first morphism is of the form \eqref{eqn:immersion}. Since the second morphism is a submersion, by shrinking $\widetilde{U}$, we may assume the second morphism is of the form \eqref{eqn:submersion}. Finally, by further shrinking $U$, we may assume the image of the first morphism is contained in $\bbR^N\times V$, whence the claim.
\end{proof}

\begin{rem}
\cite[Definition 2.9.23]{pardon-DSm} defines a notion of amplitude at most $I\subset\bbZ_{\geq0}$; moreover, any morphism in $\DSm$ admits a so called ``minimal amplitude factorization'' by \cite[Proposition 2.9.32]{pardon-DSm}. Our definitions of immersion and submersion are special cases of Pardon's amplitude definition; however, our definition of amplitude at most 1 is equivalent to Pardon's by the aforementioned minimal amplitude factorization result.
\end{rem}

The following lemmata amount to the ``calculus'' of standard presentations; they will be required in the sequel.

\begin{lem}\label{lem:standardpresentation1}
Consider $f:X\to Y,g:Y\to Z\in\DSm$ which are amplitude at most 1. A local standard presentation of $f$ and $g$ induces a local standard presentation of $g\circ f$.
\end{lem}

\begin{proof}
As the claim is local, it is enough to prove it when $f$ and $g$ are (globally) of standard presentation. By considering a standard presentation of $f$ and $g$, we have that $g\circ f$ is of the following form:
    \begin{equation}
    \begin{tikzcd}
    X\arrow[r]\arrow[d]\arrow[dr, phantom, "\lrcorner", very near start] & Y\times\bbR^N\arrow[d,"h"]\arrow[r] & Y\arrow[d]\arrow[r]\arrow[dr, phantom, "\lrcorner", very near start] & Z\times\bbR^M\arrow[d,"h'"]\arrow[r] & Z \\
    * \arrow[r] & \bbR^k & *\arrow[r] & \bbR^\ell. &
    \end{tikzcd}
    \end{equation}
Observe, we may define an immersion 
    \begin{equation}
    Y\times\bbR^N\hookrightarrow (Z\times\bbR^M)\times\bbR^N=Z\times\bbR^{N+M}
    \end{equation}
which exhibits $Y\times\bbR^N$ as a closed subset. By using a smooth bump function, we may extend $h$ to a function 
    \begin{equation}
    h:Z\times\bbR^{N+M}\to\bbR^k
    \end{equation}
satisfying $h^{-1}(0)=X$. Meanwhile, we may extend $h'$ to a function 
    \begin{equation}
    h':Z\times\bbR^{N+M}\to\bbR^\ell
    \end{equation}
satisfying $(h')^{-1}(0)=Y\times\bbR^N$. Therefore, the diagram
    \begin{equation}
    \begin{tikzcd}
    X\arrow[r]\arrow[d]\arrow[dr, phantom, "\lrcorner", very near start] & Z\times\bbR^{N+M}\arrow[r]\arrow[d,"{(h,h')}"] & Z \\
    * \arrow[r] & \bbR^{k+\ell}, &
    \end{tikzcd}
    \end{equation}
is a standard presentation of $g\circ f$, whence the claim. 
\end{proof}

The proofs of the remaining lemmata are straightforward.

\begin{lem}
Consider the following pullback in $\DSm$:
    \begin{equation}
    \begin{tikzcd}
    \overline{X}\arrow[r,"\overline{f}"]\arrow[d]\arrow[dr, phantom, "\lrcorner", very near start] & \overline{Y}\arrow[d] \\
    X\arrow[r,"f",swap] & Y,
    \end{tikzcd}
    \end{equation}
where $f$ is amplitude at most 1. A local standard presentation of $f$ induces a local standard presentation of $\overline{f}$.
\end{lem}

\begin{lem}
Consider $f:X\to Y,f':X'\to Y'\in\DSm$ which are amplitude at most 1. We have that $f\times f'$ is amplitude at most 1; moreover, a local standard presentation of $f$ and $f'$ induces a local standard presentation of $f\times f'$.
\end{lem}

\begin{lem}
Consider $f:X\to Y,f':X'\to Y'\in\DSm$ which are amplitude at most 1. We have that $f\sqcup f'$ is amplitude at most 1; moreover, a local standard presentation of $f$ and $f'$ induces a local standard presentation of $f\sqcup f'$.
\end{lem}

\section{Perfect complexes and virtual bundles}\label{sec:perfectcomplexandvirtual}
In this section, we introduce the notion, due to John Pardon, of a perfect complex on a derived smooth manifold $X$.
We show that any such perfect complex canonically determines a virtual bundle on $\rlz{X}$, whenever $\rlz{X}$ is paracompact Hausdorff. 
This will permit us to make sense of stable (and virtual) tangential structures on derived smooth manifolds.

\subsection{Perfect complexes on derived smooth manifolds}\label{subsection: perf and tangent}

\begin{defin}
    Let $\m{E}$ be an \category which admits finite products. A \textit{commutative ring object in $\m{E}$} is a product preserving functor
    \begin{equation}
    R : \mathrm{Poly}^{{\rm fin},\op} \rightarrow \m{E},
    \end{equation}
    where $\mathrm{Poly}^{\rm fin}$ is the full subcategory of commutative rings spanned by objects of the form $\bbZ[x_{1},\dots,x_{n}]$, $n \geq 0$. 
    We denote by 
    \begin{equation}
    \Ring(\m{E})\equiv \Fun^{\times}(\mathrm{Poly}^{{\rm fin},\op},\m{E})
    \end{equation}
    the \category of such objects. 
\end{defin}

\begin{rem}
    Note that for $\m{E} = \Set$, we recover the $1$-category of commutative rings: $\Ring \simeq \Ring(\Set)$. Moreover, there is a fully faithful functor $\Ring \rightarrow \Ring(\Spaces)$ whose essential image consists of the commtuative ring objects in $\Spaces$ with $R \simeq \pi_{0}R$. 
\end{rem}

\begin{example}
    The real line $\bbR$ is a commutative ring object in $\Sm$ (and therefore, also in $\DSm$ and $\Top$).
\end{example}

\begin{rem}
    The objects of the \category $\Ring(\Spaces)$ are often called ``animated rings'' in the literature. 
    We prefer the the phrase \textit{commutative ring object} as it makes sense for other \categories of coefficients such as $\Sm$ or $\DSm$. 
\end{rem}

\begin{defin}
    Given $X \in \DSm$, we define the \textit{ring of smooth functions} on $X$ to be the commutative ring object
    \begin{equation}
    C^{\infty}(X)\equiv \Map_{\DSm}(X,\bbR) \in \Ring(\Spaces).
    \end{equation}
    In fact, there is a functor 
    \begin{equation}
    C^{\infty}(-)\equiv\Map_{\DSm}(-,\bbR) : \DSm^{\op} \rightarrow \Ring(\Spaces).
    \end{equation}
\end{defin}

\begin{defin}
    Given $X \in \Top$, we define the \textit{ring of continuous functions} on $X$ to be the commutative ring
    \begin{equation}
    C(X) \equiv \Map_{\Top}(X,\bbR) \in \Ring. 
    \end{equation}
In fact, there is a functor 
    \begin{equation}
    C(-)\equiv\Map_{\Top}(-,\bbR) : \Top^{\op} \rightarrow \Ring.
    \end{equation}
\end{defin}

\begin{rem}
Observe, given $X\in\DSm$, the functor $\rlz{-}: \DSm \rightarrow \Top$ determines a canonical map
    \begin{equation}
    C^\infty(X)\to C(\rlz{X})
    \end{equation}
which is natural in $X$.
Furthermore, because $C(\rlz{X})$ is discrete, this morphism must factor uniquely through $\pi_{0}C^{\infty}(X)$. 
In the case where $X \in \Sm$, we already have $\pi_{0}C^{\infty}(X) = C^{\infty}(X)$, and the map above is given by viewing a smooth function as a continuous one. 
\end{rem}

Any $R \in \Ring(\Spaces)$ canonically determines a connective $\einf$-ring spectrum which we also denote by $R$, cf. \cite[Construction 25.1.2.1]{lurieSAG}.

\begin{notation}
For $R \in \Ring(\Spaces)$, let $\Mod_{R}\equiv\Mod_{R}(\Sp)$ denote the stable presentably symmetric monoidal \category of $R$-module spectra. Furthermore, let:
\begin{itemize}
    \item $\Perf_{R}\subset\Mod_R$ denote the idempotent complete stably symmetric monoidal subcategory of compact $R$-modules;
    \item $\Stab_{R}\subset\Mod_R$ denote the smallest stable subcategory containing $R$ (note, stably symmetric monoidal); and
    \item $\Free^{\omega}_{R}$ denote the full subcategory of finitely generated free\footnote{Here, by ``free'' we mean a direct sum of copies of $R$. Notably, we do not allow suspension.} $R$ modules (note, this is additive and symmetric monoidal).    
\end{itemize}
We have the following relations: 
    \begin{equation}
    \Free^{\omega}_{R}\subset\Stab_{R}\subset\Perf_R.
    \end{equation}
\end{notation}

\begin{rem}
If $R$ an ordinary commutative ring, we can identify $\Stab_{R}$ with $\m{D}^{b}(\Free^{\omega}_{R})$, where the latter is the bounded derived \category of finitely generated free $R$-modules.
For instance, we have 
    \begin{equation}
    \Stab_{C(X)} \simeq \m{D}^{b}(\Free_{C(X)}^{\omega}),\;\;X\in\Top.
    \end{equation}
   
\end{rem}

\begin{rem}
Observe, the inclusion $\Free_{R}^{\omega} \rightarrow \Stab_{R}$ is natural in $R$; therefore, we can view $R \mapsto \Free_{R}^{\omega}$ as a subfunctor of $R\mapsto \Stab_{R}$, where we view both of these as functors from $\Ring(\Spaces)$ to $\Cat_\infty$.
\end{rem}

Let $\Cat^{\rm ex}_{\infty}$ denote the \category of stable \categories and exact functors between them which, by \cite[Theorem 4.22]{blumbergUniversalCharacterizationHigher2013}, is a compactly generated \category. 
This is equipped with a forgetful functor $\Cat_{\infty}^{\ex} \rightarrow \Cat_{\infty}$ which preserves small limits and filtered colimits, cf. Lemma~\ref{lem: forgetful lex to Cat} below.

Before continuing, we recall that, for any functor $\m{F} : \Top^{\op} \rightarrow \m{E}$, where $\m{E}$ is a presentable $\infty$-category, we may define its sheafification $\m{F}^{\rm sh}$ via a transfinite iteration of the dagger construction:
    \begin{equation}
    \m{F} \mapsto \m{F}^{\dagger},
    \end{equation}
cf. \cite[Remark 6.2.2.12]{lurieHigherToposTheory2009} and the proof of \cite[Proposition 6.2.2.7]{lurieHigherToposTheory2009}. (Of course, we can more generally define sheafification for any $\m{E}$-valued presheaf on a Grothendieck site; again, cf. \emph{loc. cit.})

\begin{defin}[Pardon]\label{definition: Perf of DSm}
Let
    \begin{equation}
    \Perf : \DSm^{\op} \rightarrow \Cat_{\infty}^{\ex}
    \end{equation}
denote the sheafification of the functor 
\begin{equation}
\Stab_{C^{\infty}(-)} : \DSm^{\op} \rightarrow \Cat^{\ex}_{\infty}.
\end{equation}
Given $X\in\DSm$, the \emph{$\infty$-category of perfect complexes} on $X$ is 
    \begin{equation}
    \Perf_{X} \equiv \Gamma(X;\Perf).
    \end{equation}
Analogously, let
    \begin{equation}
    \Vect^{\rm sm}_{(-)}:\DSm^{\op} \rightarrow \Cat_{\infty}
    \end{equation}
denote the sheafification of the functor 
\begin{equation}
\Free^{\omega}_{C^{\infty}(-)} : \DSm^{\op} \rightarrow \Cat_{\infty}. 
\end{equation}
The \emph{$\infty$-category of smooth vector bundles} on $X$ is 
    \begin{equation}
    \Vect^{\rm sm}_{X} \equiv \Gamma(X;\Vect^{\rm sm}).
    \end{equation}
\end{defin}

\begin{rem}
For $X\in\Sm$, it is not hard to see that $\Vect_{X}^{\rm sm}$ is equivalent to the nerve of the $1$-category of smooth vector bundles on $X$.
\end{rem}

Now, we will explain how to construct the tangent complex of a derived smooth manifold and the relative tangent/normal complex of a map of derived smooth manifolds; these are our main examples of perfect complexes on derived smooth manifolds.
Consider the following commutative diagram:
    \begin{equation}
    \begin{tikzcd}
    \int^{\rm c}_{X\in\Sm}\Vect^{\rm sm}_X\arrow[d]\arrow[r] & \int^{\rm c}_{X\in\DSm}\Perf_X\arrow[d] \\
    \Sm\arrow[r] & \DSm.
    \end{tikzcd}
    \end{equation}
Assigning to a smooth manifold its tangent bundle determines a section
    \begin{equation}
    T:\Sm\to\int^{\rm c}_{X\in\Sm}\Vect^{\rm sm}_X;
    \end{equation}
by composing, we obtain a functor 
    \begin{equation}\label{eqn:tangentfunctorcompose}
    T:\Sm\to\int^{\rm c}_{X\in\DSm}\Perf_X.
    \end{equation}

\begin{defin}
The \emph{tangent complex} of $X\in\DSm$ is $T_X\in\Perf_X$, where $T$ is the relative right Kan extension of \eqref{eqn:tangentfunctorcompose} determined by the universal property of $\DSm$:
    \begin{equation}
    T:\DSm\to\int^{\rm c}_{X\in\DSm}\Perf_X.
    \end{equation}
\end{defin}

\begin{rem}
Given $f:X\to Y$ in $\DSm$, we obtain a morphism
    \begin{equation}
    (X,T_{X}) \rightarrow (Y,T_{Y});
    \end{equation}
this amounts to the specification of a map 
    \begin{equation}
    Df:T_X\to f^*T_Y\in\Perf_X.
    \end{equation}
\end{rem}

\begin{defin}
Let $f:X\to Y\in\DSm$.
\begin{enumerate}
\item The \emph{relative tangent complex} of $f$ is the homotopy fiber $T_f\in\Perf_X$ of $Df$.
\item The \emph{normal complex} of $f$ is the homotopy cofiber $N_f\in\Perf_X$ of $Df$. 
\end{enumerate}
\end{defin}

The proof of the following two results is straightforward.

\begin{lem}
If
    \begin{equation}
    \begin{tikzcd}
    \overline{X}\equiv X\times_Y\overline{Y}\arrow[d,"\overline{g}",swap]\arrow[r,"\overline{f}"]\arrow[dr, phantom, "\lrcorner", very near start] & \overline{Y}\arrow[d,"g"] \\
    X\arrow[r,"f",swap] & Y
    \end{tikzcd}
    \end{equation}
is a pullback diagram in $\DSm$, then there is an induced pullback diagram in $\Perf_{\overline{X}}$:
    \begin{equation}
    \begin{tikzcd}
    T_{\overline{X}}\arrow[d,"D\overline{g}",swap]\arrow[r,"D\overline{f}"]\arrow[dr, phantom, "\lrcorner", very near start] & \overline{f}^*T_{\overline{Y}}\arrow[d,"\overline{f}^*Dg"] \\
    \overline{g}^*T_X\arrow[r,"\overline{g}^*Df",swap] & \overline{g}^*f^*T_Y=\overline{f}^*g^*T_Y.
    \end{tikzcd}
    \end{equation}
In particular, we have the following two equivalences:
    \begin{equation}
    \overline{g}^*T_f\simeq T_{\overline{f}}\;\;{\rm and}\;\;\overline{g}^*N_f\simeq N_{\overline{f}}.
    \end{equation}
\end{lem}

\begin{lem}
If $f:X\to Y,g:Y\to Z\in\DSm$, then we have the following two cofiber sequences: 
    \begin{equation}
    T_{f} \rightarrow T_{g\circ f} \rightarrow f^{\ast}T_{g}\;\;{\rm and}\;\;N_{f} \rightarrow N_{g\circ f} \rightarrow f^{\ast}N_{g}. 
    \end{equation}
\end{lem}

To make sense of tangential structures on derived smooth manifolds, we need a means to convert (relative) tangent complexes into virtual bundles.

\begin{notation}
    Let $\Sing : \Top \rightarrow \Spaces$ denote the singular complex functor. 
    If $X \in \DSm$, we write $\Sing(X) \equiv \Sing(\rlz{X})$. 
\end{notation}

\begin{defin}
    For $X \in \Top$, the \textit{topological K-theory of $X$} is defined to be 
    \begin{equation}
    \KO(X) \equiv \Map_{\Spaces}\big(\Sing(X),\oko\big),
    \end{equation}
    where $\ko$ denotes the connective real topological K-theory spectrum. 
    We refer to the points of $\KO(X)$ as \textit{virtual bundles on $X$}.
    As $\oko$ is a grouplike $\einf$-monoid, the above formula determines a functor 
    \begin{equation}
    \KO : \Top^{\op} \rightarrow \CMongp.
    \end{equation}
\end{defin}

\begin{warning}
    By real Bott periodicity, $\Omega^{\infty}\ko = \BOP$ as $\einf$-monoids. 
    However, the space $\BOP$ can be equipped with a \textit{different} $\einf$-monoid structure by forming the \textit{product} of the $\einf$-monoids $\bbZ$ and $\BO$; this will be discussed later in Recollection~\ref{recollection:ko}. 
    As is hopefully now apparent, the notation $\BOP$ can suggest that the ``Bott periodic'' $\einf$-monoid structure is split. 
    To avoid any potential confusion, we have opted for the notation $\Omega^{\infty}\ko$ in place of $\BOP$.  
\end{warning}

\begin{lem}
    $\KO : \Top^{\op} \rightarrow \CMongp$ is a sheaf; so is its restriction to $\DSm$.
\end{lem}

Our desired association 
    \begin{equation}
    T_X\rightsquigarrow T^\vir_X\in\KO(X)
    \end{equation}
should also split cofiber sequences; this will yield the familiar formula
\begin{equation}
T^{\vir}_{f} \simeq T^{\vir}_{X} - f^{\ast}T_{Y}^{\vir}. 
\end{equation}
We are therefore lead to consider the algebraic K-theory space of a derived smooth manifold $X$:
    \begin{equation}
    \K(X)\equiv \K(\Perf_X).
    \end{equation}
By the universal property of algebraic K-theory, there is a natural map of $\bbE_\infty$-monoids,
    \begin{equation}
    \Perf_X^\simeq\to\K(X),
    \end{equation}
which splits cofiber sequences in $\Perf_{X}$. The remainder of the present section is devoted to constructing a natural map
    \begin{equation}\label{equation: K to KO}
    \K(X)\to\KO(X)
    \end{equation}
whenever $X \in \DSm$ is paracompact and Hausdorff. 

\subsection{Continuous perfect complexes and amplitude}
In order to construct the map in \eqref{equation: K to KO}, we must pass through the stable $\infty$-category of ``continuous perfect complexes'' on a topological space; we investigate this stable $\infty$-category and some of its natural subcategories in this subsection. Note, much of the work in this subsection can be adapted to study the $\infty$-category of perfect complexes on a derived smooth manifold.  

\begin{defin}
Let
\begin{equation}
\cperf : \Top^{\op} \rightarrow \Cat^{\ex}_{\infty}
\end{equation}
denote the sheafification of the functor 
\begin{equation}
\Stab_{C(-)} : \Top^{\op} \rightarrow \Cat^{\ex}_{\infty}.
\end{equation}
The \emph{stable $\infty$-category of continuous perfect complexes} on $X$ is
\begin{equation}
\cperf_{X} \equiv \Gamma(X;\cperf).
\end{equation}
\end{defin}

To define subcategories of $\cperf_{X}$, we proceed by sheafifying subfunctors of $R \mapsto \Stab_{R}$. 
Therefore, we need to recall some basic properties of modules over ring spectra. 
Throughout, we will work with $R \in \Ring(\Spaces)$, though all the definitions can be phrased for connective $\ee{1}$-rings, cf. \cite[Subsection 7.2]{lurieHigherAlgebra2017}.

\begin{notation}
For any $k \in \bbZ$, let $(\Stab_{R})_{\geq k}$ resp. $(\Stab_{R})_{\leq k}$ denote the full subcategories of $k$-connective resp. $k$-coconnective objects in $\Stab_{R}$. The former is closed under finite colimits; the latter is closed under finite limits; both are closed under extensions. Moreover, let $\Stab_{R}^{\rm cn}\subset\Stab_{R}$ denote the full subcategory of connective objects. 
\end{notation}

\begin{defin}
    We have the following two similar definitions.
    \begin{enumerate}
    \item We say $M \in \Mod_{R}$ has \emph{Tor-amplitude $\leq b$} provided that, for every discrete $R$-module $N$, 
    \begin{equation}
    \pi_{i}(M\otimes_{R}N) = 0,\;\;i>b.
    \end{equation}
    We denote by $\Stab_{R}^{\leq b}\subset\Stab_{R}$ the full subcategory of such objects.
    \item Analogously, we say $M \in \Mod_{R}$ has \emph{Tor-amplitude in $[a,b]$} provided that, for every discrete $R$-module $N$,
    \begin{equation}
    \pi_{i}(M\otimes_{R}N) = 0,\;\;i\notin[a,b].
    \end{equation}
    We denote by $\Stab_{R}^{[a,b]} \subset \Stab_{R}$ the full subcategory of such objects.
    \end{enumerate}
\end{defin}

\begin{rem}
    Since any $M \in \Stab_{R}$ is $a$-connective for some $a \in \bbZ$ and has Tor-amplitude $\leq b$ for some $a \leq b$, we see $M \in \Stab_{R}^{[a,b]}$, i.e.,
    \begin{equation}
    \colim_{a \leq b} \Stab_{R}^{[a,b]} \simeq \Stab_{R}.
    \end{equation}
\end{rem}

We require the following lemma in the sequel; its proof is straighforward.

\begin{lem}
    Let $R$ be an ordinary commutative ring and $M \in \Stab_{R}^{[a,b]}$. 
    There exists a morphism $\Sigma^{a} F \rightarrow M$, where $F$ is a finitely generated free $R$-module, whose cofiber has Tor-amplitude in $[a+1,b]$.
\end{lem}


With these subcategories prescribed, we turn to the  sheafification construction; this requires some preliminary work. 
Let $\Cat^{\rm rex}_{\infty}$ denote the \category of (small) \categories which admit finite colimits and functors which preserve finite colimits. 
Similarly, let $\Cat^{\rm lex}_{\infty}$ denote the \category of (small) \categories which admit finite limits and functors which preserve finite limits. 
The equivalence of \categories $(-)^{\op} : \Cat_{\infty} \simeq \Cat_{\infty}$ exhibits these subcategories as ``dual'' in the sense that the essential image of $\Cat^{\rm rex}_{\infty}$ is given by $\Cat^{\rm lex}_{\infty}$, and vice versa. The following lemma is standard; the proof is virtually identical to \cite[Proposition 5.5.7.11]{lurieHigherToposTheory2009}. 

\begin{lem}\label{lem: forgetful lex to Cat}
    The forgetful functors
    \begin{equation}
    \Cat^{\rm rex}_{\infty} \rightarrow \Cat_{\infty}, \ \  \Cat^{\rm lex}_{\infty} \rightarrow \Cat_{\infty}, \ \ \text{and} \ \ \Cat^{\rm ex}_{\infty} \rightarrow \Cat_{\infty}
    \end{equation}
    preserve small limits and filtered colimits. 
\end{lem}


In the sequel, we typically denote by 
    \begin{equation}
    \Phi: \Stab^{S} \rightarrow \Stab\in\Fun\big(\Ring(\Spaces),\Cat_\infty\big)
    \end{equation}
an arbitrary subfunctor of $\Stab$.

\begin{lem}
    If
    \begin{equation}
    \Phi: \Stab^{S} \rightarrow \Stab\in\Fun\big(\Ring(\Spaces),\Cat_\infty\big)
    \end{equation}
    is levelwise fully faithful, then
    \begin{equation}
    \Phi^{\sh}: \cperf^{S}_{C(-)} \rightarrow \cperf_{C(-)}\in\Shv(\Top;\Cat_\infty),
    \end{equation}
    obtained by sheafification and restriction, is levelwise fully faithful. 
\end{lem}

\begin{proof}
After restricting along $C:\Top^{\op} \rightarrow \Ring(\Spaces)$,  we see the induced natural transformation of functors from $\Top^{\op}$ to $\Cat_{\infty}$,
    \begin{equation}
    \Stab^{S}_{C(-)} \rightarrow \Stab_{C(-)},
    \end{equation}
is levelwise fully faithful. 
    Now, since the formation of mapping spaces commutes with both limits and filtered colimits of $\infty$-categories, we see that the induced transformation $(\Stab^{S}_{C(-)})^{\dagger} \rightarrow \Stab_{C(-)}^{\dagger}$ is also levelwise fully faithful. 
    Finally, as sheafification is a transfinite iteration of $(-)^{\dagger}$, we find that, for any $X \in \Top$, the induced functor 
    \begin{equation}
    \cperf^{S}_{X} \rightarrow \cperf_{X}
    \end{equation}
    is fully faithful.
\end{proof}

Given any $X\in\Top$ together with a point $x\in X$, we denote by 
    \begin{equation}
    \cperf^{S}_{X,x}\;\;{\rm resp.}\;\;\cperf_{X,x}
    \end{equation}
the stalk at $x$. The next result concerns ``gluing'' continuous perfect complexes.

\begin{lem}\label{lem: gluing complexes}
Let $\Phi : \Stab^{S} \rightarrow \Stab\in\Fun\big(\Ring(\Spaces),\Cat_\infty\big)$ and $X\in\Top$ together with a point $x\in X$. We abuse notation and denote by $x$ the inclusion $\{x\}\to X$.
    \begin{enumerate}
        \item For any $V \in \cperf_{X}$, if $x^{\ast}V$ belongs to the essential image of $\cperf^{S}_{X,x}$, then there exists an open neighborhood $U_{x}$ of $x$ such that $V|_{U_x} \in \cperf^{S}_{U_x}$. 
        \item For any $V \in \cperf_{X}$, if $\{U_{i}\}_{i \in I}$ is an open cover of $X$ such that $V|_{U_i} \in \cperf^{S}_{U_i}$ for all $i \in I$, then $V \in \cperf^{S}_{X}$. 
        \item The functor 
        \begin{equation}
        (x^{\ast})_{x \in X} : \cperf_{X} \rightarrow \prod_{x \in X} \cperf_{X,x}
        \end{equation}
        is conservative. 
    \end{enumerate}
\end{lem}

\begin{proof}
Consider (1). Given any $\m{F} \in \PShv(X;\Cat_{\infty})$, since $\Cat_{\infty}$ is compactly generated,
    \begin{equation}
    \m{F}^{\sh}_{x} \simeq \colim_{U \in \Open_{x}(X)^{\op}} \m{F}^{\sh}(U) \simeq \colim_{U \in \Open_{x}(X)^{\op}} \m{F}(U) \simeq \m{F}_{x}, 
    \end{equation}
    where $\Open_x(X)$ denotes the subcategory of open sets containing $x$.
    Hence, by the compactness of $[0] \in \Cat_{\infty}$, if $x^{\ast}V \in \cperf^{S}_{X,x}$, then we can lift $x^{\ast}V$ to $\cperf^{S}_{U_{x}}$ for some open neighborhood $U_{x}$. Meanwhile, both (2) and (3) are simple applications of the sheaf condition.     
\end{proof}

The following result shows that certain properties of $\Phi$ are inherited by $\Phi^{\sh}$. 
\begin{prop}\label{prop: sheafified subcats have (co)limits}
    Let $\Phi: \Stab^{S} \rightarrow \Stab\in\Fun\big(\Ring(\Spaces),\Cat_\infty\big)$ be levelwise fully faithful.
    \begin{enumerate}
        \item If $\Phi$ factors through $\Cat^{\rm rex}_{\infty}$ levelwise, then, for any $X\in\Top$, the \category $\cperf^{S}_{X}$ admits finite colimits and the functor $\Phi_{X}^{\sh}$ preserves them.   
        \item If $\Phi$ factors through $\Cat^{\rm lex}_{\infty}$ levelwise, then, for any $X\in\Top$, the \category $\cperf^{S}_{X}$ admits finite limits and the functor $\Phi_{X}^{\sh}$ preserves them. 
        \item If, for any $X\in\Top$, the essential image of $\Phi_{X}$ is closed under extensions, then $\cperf^{S}_{X}$ is closed under extensions in $\cperf_{X}$ for all $X \in \Top$.  
    \end{enumerate}
\end{prop}

\begin{proof}
For (1), it is clear from Lemma~\ref{lem: forgetful lex to Cat} that $\cperf^{S}_{X}$ admits finite colimits and that $\cperf^{S}_{X} \rightarrow \cperf_{X}$ preserves these; the argument for (2) is similar. For (3), assume we are given a cofiber sequence $V_{0} \rightarrow V_{1} \rightarrow V_{2}\in\cperf_{X}$, where $V_{0},V_{2} \in \cperf^{S}_{X}$.
This induces, for all $x \in X$, a cofiber sequence in $\cperf_{X,x}$; moreover, our assumption on $\Phi_{X}$ implies the essential image of $\Phi_{X,x}$ is closed under extensions. Hence, we may lift $x^{\ast}V_{1}$ to $\cperf^{S}_{U_{x}}$ for some open neighborhood $U_{x}$; we conclude via the previous lemma.
\end{proof}

\begin{defin}
For any $k\in\bbZ$, let
    \begin{equation}
    (\cperf)_{\geq k} : \Top^{\op} \rightarrow \Cat^{\rm rex}_{\infty}
    \end{equation}
    denote the sheafification of the functor 
    \begin{equation}
    (\Stab_{C(-)})_{\geq k} : \Top^{\op} \rightarrow \Cat^{\rm rex}_{\infty}. 
    \end{equation}
    The \textit{\category of $k$-connective perfect complexes} on $X$ is
        \begin{equation}
        (\cperf_{X})_{\geq k} \equiv \Gamma\big(X;(\cperf)_{\geq k}\big);
        \end{equation}
    when $k = 0$, we write $\cperf^{\rm cn}_{X} \equiv (\cperf_{X})_{\geq 0}$.
\end{defin}

\begin{rem}
For any $X\in\DSm$, we can analogously define the \emph{\category of connective perfect complexes} on $X$, which we denote by $\Perf_{X}^{\rm cn}$; this subcategory of $\Perf_X$ is closed under finite colimits and extensions. Observe, for any $X \in \DSm$, we have $T_{X} \in \Perf_{X}^{\rm cn}$ and, for any $f: X \rightarrow Y\in\DSm$, we have $N_{f} \in \Perf_{X}^{\rm cn}$.
\end{rem}

\begin{defin}
Let
    \begin{equation}
    \cperf^{\leq b} : \Top^{\op} \rightarrow \Cat_{\infty}^{\rm lex}\;\;{\rm resp.}\;\;\cperf^{[a,b]} : \Top^{\op} \rightarrow \Cat_\infty
    \end{equation}
denote the sheafification of the functor
    \begin{equation}
    \Stab_{C(-)}^{\leq b} : \Top^{\op} \rightarrow \Cat_{\infty}^{\rm lex}\;\;{\rm resp.}\;\;\Stab_{C(-)}^{[a,b]} : \Top^{\op} \rightarrow \Cat_\infty.
    \end{equation}
    The \textit{\category of perfect complexes of Tor-amplitude $\leq b$} on $X$ resp. \emph{\category of perfect complexes of Tor-amplitude in $[a,b]$} is 
        \begin{equation}
        \cperf_{X}^{\leq b}\equiv\Gamma(X;\cperf^{\leq b})\;\;{\rm resp.}\;\;\cperf_{X}^{[a,b]} \equiv \Gamma(X;\cperf^{[a,b]}).
        \end{equation}
\end{defin}

\begin{rem}
    By Proposition~\ref{prop: sheafified subcats have (co)limits}, each subcategory of $\cperf_{X}$ in the previous definition is closed under extensions.
\end{rem}

\begin{lem}\label{lem: properties of Tor-amplitude}
Let $X\in\Top$ and $a\leq b$.
    \begin{enumerate}
        \item If $V \in \cperf_{X}^{\leq b}$ resp. $V\in\cperf_{X}^{[a,b]}$, then, for any $k \in \bbZ$, we have $\Sigma^{k}V \in \cperf_{X}^{\leq b+k}$ resp. $\Sigma^{k}V \in \cperf^{[a+k,b+k]}_{X}$. 
        \item If $V_{0} \rightarrow V_{1} \rightarrow V_{2}\in\cperf_X$ is a cofiber sequence such that $V_{0}$ has Tor-amplitude $\leq b_{0}$ and $V_{1}$ has Tor-amplitude $\leq b_{1}$, then $V_{2}$ has Tor-amplitude $\leq \max\{b_{0},b_{1}\}+1$. 
        \item If $V \in \cperf_{X}^{\leq b}$ is $a$-connective, then $V \in \cperf_{X}^{[a,b]}$.
    \end{enumerate}
\end{lem}

\begin{proof}
The properties of Tor-amplitude for module spectra (cf. \cite[Proposition 2.13]{antieau-gepner-brauer}) combined with Lemma~\ref{lem: gluing complexes} yield (1) and (2).
For (3), by shifting, we may assume that $a = 0 \leq b$.
By Lemma~\ref{lem: gluing complexes}, it is enough to check on stalks, and the result follows as $M\otimes_{R}N$ is connective whenever $R$, $M$, and $N$ are connective. 
\end{proof}


\begin{rem}
    For any $X \in \Top$, it is not difficult to see that 
    \begin{equation}
    \cperf^{[0,0]}_{X} \simeq \Vect_{X},
    \end{equation}
    where $\Vect_{X}$ denotes the nerve of the $1$-category of real vector bundles on $X$.
    Consequently, for any $a \in \bbZ$, we see
    \begin{equation}
    \Sigma^{-a}: \cperf^{[a,a]}_{X} \xrightarrow{\sim} \cperf^{[0,0]}_{X}\simeq \Vect_{X}.
    \end{equation}
\end{rem}



\begin{prop}\label{prop: locally bounded tor amp}
    Let $X\in\Top$ and $V \in \cperf_{X}$. 
    For each $x \in X$, there exists an open neighborhood $U_{x}\ni x$ such that $V|_{U_x}$ has bounded Tor-amplitude. 
    Consequently, if $X$ is compact, we have an equivalence of \categories:
    \begin{equation}
    \colim_{a\leq b} \cperf_{X}^{[a,b]} \simeq  \cperf_{X}.
    \end{equation}  
\end{prop}

\begin{proof}
    Let $\cperf|_{X} \in \Shv(X;\Cat_\infty)$ denote the restriction of $\cperf$ to a sheaf on $X$ and note that 
    \begin{equation}
    \Gamma(U;\cperf|_{X}) \simeq \cperf_{U}
    \end{equation}
    for any $U \subset X$ open. 
    Because $\Cat_\infty$ is compactly generated, there is an equivalence 
    \begin{equation}
    (\cperf|_{X})_{x} \simeq \colim_{U \in \Open_{x}(X)^{\op}} \cperf_{U}.
    \end{equation}
    However, we see that
    \begin{equation}
    (\cperf|_{X})_{x} \simeq (\Stab_{C(-)}|_{X})_{x} \simeq \colim_{U \in \Open_{x}(X)^{\op}} \Stab_{C(U)};
    \end{equation}
    this implies 
    \begin{equation}
    (\cperf|_{X})_{x} \simeq \colim_{a\leq b} (\Stab^{[a,b]}_{C(-)}|_{X})_{x} \simeq \colim_{a\leq b} (\cperf^{[a,b]}|_{X})_{x}.
    \end{equation}
    In particular, we can choose some open neighborhood $U_{x}\ni x$ such that $V|_{U_{x}} \in \cperf^{[a,b]}_{U_x}$, whence the first claim. Now, suppose in addition $X$ is compact; it will be enough to show every perfect complex has bounded Tor-amplitude.
    Let $\{U_{i}\}_{i \in I}$ be a covering such that $V|_{U_{i}}$ has bounded Tor-amplitude. 
    By compactness, we may choose a finite subcover and therefore integers $a \leq b$ such that $V$ has Tor-amplitude $[a,b]$ on each open, whence the second claim.
\end{proof}


\subsection{Finiteness properties}
For $X\in\Top$, we consider the sheaf of continuous functions which determines a commutative ring object:
    \begin{equation}
    C_X\in\Ring\big(\Shv(X;\mathrm{Ab})\big).
    \end{equation}
By forming Eilenberg-MacLane spectra, we may view $C_{X}$ as a commutative ring object in  $\Shv(X;\Sp)$, which in turn determines an $\einf$-algebra. 
We define
    \begin{equation}
    \Mod_X\equiv\Mod_{C_X}\big(\Shv(X;\Sp)\big),
    \end{equation}
and, by \cite[Proposition 2.1.1.1 \& Remark 2.1.2.1]{lurieSAG}, $\Mod_{X}$ admits a t-structure where $\m{F} \in \Mod_{X}$ is:
\begin{itemize}
\item $n$-connective if, for any $k<n$, the homotopy sheaf $\pi_{k}\m{F}$ vanishes;
\item or $n$-truncated if, for any open $U\subset X$, the spectrum $\m{F}(U)$ is $n$-truncated.
\end{itemize}
Furthermore, we have 
    \begin{equation}
    \Mod_{X}^{\heartsuit} \simeq \Mod_{C_X}\big(\Shv(X;\mathrm{Ab})\big).
    \end{equation}
There is a fully-faithful additive functor $\Vect_{X} \rightarrow \Mod_{X}^{\heartsuit}$ which sends a vector bundle to its sheaf of continuous sections; this identifies $\Vect_X$ with the full subcategory of $\Mod_X$ consisting of locally free $C_X$-modules of finite rank. Furthermore, we can identify $\Free^{\omega}_{C(X)}$ with the smallest additive subcategory of $\Mod_{X}$ containing $C_X$, and, if $X$ is additionally compact Hausdorff, $\Vect_{X}$ is equivalent to the smallest idempotent complete additive subcategory containing $C_X$.

\begin{prop}
    The presheaf $\Mod : \Top^{\op} \rightarrow \bigcat$, given by $X \mapsto \Mod_{X}$, is a sheaf; consequently, we have subfunctors 
    \begin{equation}
    \Stab_{C(-)} \rightarrow \cperf \rightarrow \Mod. 
    \end{equation}
\end{prop}

\begin{proof}
    By universal descent for $\infty$-topoi \cite[Lemma 6.1.3.7]{lurieHigherToposTheory2009} it follows that  
    \begin{equation}
        \Shv(-;\sphere) : \Top^{\op} \rightarrow \bigcat
    \end{equation}
    is a sheaf of symmetric monoidal \categories. 
    Because $C_{X} \in \CAlg\big(\Shv(X;\sphere)\big)$, under the equivalence of \categories 
    \begin{equation}
        \CAlg\Big(\lim_{C(\calU)^{\op}} \Shv(U;\sphere)\Big)\simeq \lim_{C(\calU)^{\op}} \CAlg(\Shv(U;\sphere)),
    \end{equation}
    we may identify $C_{X}$ with the family of $\einf$-algebras $(C_{U})_{U\in C(\calU)^{\op}}$ since $C_{X}|_{U} = C_{U}$. 
    Therefore, by \cite[Theorem 5.10]{linskensGlobalHomotopyTheory2025}, we have an equivalence 
    \begin{equation}
        \Mod_{X} \xrightarrow{\sim} \lim_{U\in C(\calU)^{\op}} \Mod_{U},
    \end{equation}
    so that $X\mapsto \Mod_{X}$ is a sheaf as claimed.    
    The fully faithful functor 
    \begin{equation}
    \Free^{\omega}_{C(X)} \rightarrow \Mod_{X}
    \end{equation}
    extends to a morphism of presheaves
    \begin{equation}
    \Stab_{C(-)} \rightarrow \Mod
    \end{equation}
    which is levelwise fully faithful. 
    Sheafifying the source yields the claim.     
\end{proof}


We now prove a smattering of lemmata.

\begin{lem}
For any $X\in\Top$, we denote by $\Vect^{\rm triv}_X$ the nerve of the 1-category of trivial vector bundles on $X$. Any $\m{F} \in \Mod_{X}$ is in the essential image of $\cperf_{X}$ if and only if, for each $x\in X$, there exists an open neighborhood $U_x\ni x$ such that $\m{F}|_{U_x} \in \Stab(\Vect^{\rm triv}_{U_x})$. 
\end{lem}

\begin{proof}
    Locally, every perfect complex has this form. 
    On the other hand, if we are given such an $\m{F}$, then it locally belongs to $\cperf_{U_x}$; now, by descent, this implies it belongs to $\cperf_{X}$. 
\end{proof}

\begin{lem}
   Any $E \in \Stab(\Vect^{\rm triv}_{X})$ which is connective as a $C_{X}$-module is equivalent in $\Mod_{X}$ to a geometric realization of finitely generated free $C_X$-modules. 
\end{lem}

\begin{proof}
    Let $\m{P}^{\rm fin}_{\Sigma}(\Vect_{X}^{\rm triv})$ denote the smallest subcategory of $\m{P}_{\Sigma}(\Vect_{X}^{\rm triv})$ which contains the representable objects and is closed under finite colimits.     
    By definition, there is a fully faithful functor 
    \begin{equation}
        \m{P}^{\rm fin}_{\Sigma}(\Vect_{X}^{\rm triv}) \rightarrow \Mod_{X}
    \end{equation}
    which preserves finite sifted colimits and an equivalence 
    \begin{equation}
        \mathrm{SW}\big(\m{P}^{\rm fin}_{\Sigma}(\Vect_{X}^{\rm triv})\big) \simeq \Stab(\Vect_{X}^{\rm triv}),
    \end{equation}
    where $\mathrm{SW}$ denotes the Spainer-Whitehead category \cite[C.1.1.1]{lurieSAG}. 
    Therefore, up to desuspension, any object $E \in \Stab(\Vect_{X}^{\rm triv})$ can be written as a finite sifted colimit of trivial vector bundles; in particular, if $E$ is connective no desuspensions are necessary. 
    Finally, by \cite[1.4.239]{pardon-DSm} and the Bousfield-Kan formula, we see that any such connective $E$ is equivalent to the geometric realization of a simplicial diagram whose terms are trivial vector bundles, proving the claim. 
\end{proof}

\begin{lem}
    For any $E \in \cperf_{X}^{\rm cn}$, $\pi_{0}E \in \Mod_{X}^{\heartsuit}$ is locally of finite presentation.  
\end{lem}

\begin{proof}
    The functor $\pi_{0} : \Mod_{X}^{\rm cn} \rightarrow \Mod_{X}^{\heartsuit}$ is given by $\tau_{\leq 0}$ which is a localization and hence a left adjoint. Choose a covering $\mathcal U$ of $X$ such that $E|_{U} \in \Stab(\Vect_{U}^{\rm triv})$ for all $U \in \mathcal U$. 
    By the previous lemma, there exists a simplicial object $E^{U}_{\bullet} \in \Fun(\Delta^{\op},\Vect_{U}^{\rm triv})$ such that $E|_{U} \simeq \colim_{\Delta^{\op}}(E^{U}_{\bullet})$. 
    Thus,
    \begin{equation}
    \pi_{0}(E|_{U}) \cong \pi_0\Big(\colim_{\Delta^{\op}}(E^{U}_{\bullet})\Big)\cong\colim\big(\pi_0E^{U}_{1} \rightrightarrows \pi_0E^{U}_{0}\big)
    \end{equation}
    shows $\pi_{0}(E|_{U})$ is of finite presentation, and the claim follows from the isomorphism $\pi_{0}(E|_{U})\cong \pi_{0}(E)|_{U}$.
\end{proof}

\begin{lem}\label{lem: lifting maps from vector bundles}
If $X\in\Top$ is compact Hausdorff, then
then the $C_{X}$-module associated to $E \in \Vect_{X}$ is a projective object in $\Mod_{X}^{\rm cn}$. 
     Consequently, for any vector bundle $V$ and $\m{F} \in \Mod_{X}^{\rm cn}$, there is a surjection 
     \begin{equation}
     \pi_{0}\Map_{\cperf_{X}}(V,\m{F}) \rightarrow \pi_{0}\Map_{\Mod_{X}}\big(V,\pi_{0}(\m{F})\big).
     \end{equation}
\end{lem}

\begin{proof}
    Vector bundles on compact Hausdorff spaces are summands of free $C_X$-modules; moreover, projective objects are closed under retracts, whence the first claim.
    The second claim is immediate from \cite[Proposition 7.2.2.6]{lurieHigherAlgebra2017}.
\end{proof}

Recall, a sheaf on a topological space is \emph{soft} if, for any closed subset, every section over it can be extended to the entirety of the space. Morever, recall that, for any paracompact Hausdorff space $X$, any $C_X$-module is soft, cf. \cite[9.4 Example]{Bredon-sheaf-theory}.
Roughly speaking, this means the ringed space $(X,C_X)$ satisfies Serre's cohomological criterion for affineness.

\begin{lem}
    Let $X\in\Top$ be compact Hausdorff and $\m{F}\in\Mod_{X}^{\heartsuit}$ locally of finite presentation.
    There exists a surjective map of $C_X$-modules,  
    \begin{equation}
    V \rightarrow \m{F}\in\Mod_X,
    \end{equation}
    where $V\in\Mod_X$ is locally free of finite rank. 
\end{lem}

\begin{proof}
    Choose a covering $\mathcal U$ of $X$ on which $\m{F}$ is locally the cokernel of a map between free finite rank $C_X$.
    By compactness, there is a natural number $m\in\bbN$ such that, for any $x\in X$, $\m{F}_{x}$ is generated as a $C_{X,x}$-module by at most $m$ elements. By the softness of $\m{F}$ and the compactness of $X$, there is a finite set of global sections $T\subset \Gamma(X;\m{F})$ which restrict to the stalkwise generators at every point $x \in X$, i.e., we have a surjective map
    \begin{equation}
    V\equiv(C_X)^{\oplus\rlz{T}} \rightarrow \m{F},
    \end{equation}
    as desired.
\end{proof}

We have now arrived at the upshot of the present subsection.

\begin{prop}\label{prop: map from VB}
Let $X\in\Top$ be compact Hausdorff and $E \in \cperf_{X}$ be $k$-connective. There exists a cofiber sequence
    \begin{equation}
    V[k] \rightarrow E \rightarrow E'\in\cperf_{X},
    \end{equation}
    where $V$ a vector bundle and $E'$ is $(k+1)$-connective.
\end{prop}

\begin{proof}
    By shifting, it suffices to consider $k=0$. 
    By Lemma~\ref{lem: lifting maps from vector bundles}, we can choose a surjection of sheaves $V \rightarrow \pi_{0}(E)$ which lifts to a map $V[0] \rightarrow E$. 
    If $E'$ denotes the cofiber of this lift, then the long exact sequence of homotopy sheaves shows that $E'$ is $1$-connective, whence the claim. 
\end{proof}


\subsection{Topological K-theory}
Given a topological space $X$, there at least four ``topological K-theory spaces'' one can consider.

\begin{enumerate}[itemsep=5pt,topsep=5pt]
    \item $\K(\cperf_{X})$, the algebraic K-theory of the stable \category of perfect complexes on $X$;
    \item $\K(\Vect_{X})$, the algebraic K-theory of the exact 1-category of real vector bundles on $X$; 
    \item $\K^{\rm add}(\Vect_{X}) = \left(\Vect_{X}^{\simeq}\right)^{\rm gp}$, the algebraic K-theory of the additive 1-category of real vector bundles on $X$; and 
    \item $\KO(X) = \Map_{\Spaces}\big(\Sing(X),\Omega^{\infty}\ko\big)$, the space represented by the grouplike $\einf$-monoid $\Omega^{\infty}\ko$. 
\end{enumerate}
In general, these will yield different grouplike $\einf$-monoids in $\Spaces$.
When $X$ is compact Hausdorff, the first three are equivalent and moreover 
\begin{equation}
\K^{\rm add}_{0}(\Vect_{X}) \cong \pi_{0}\KO(X) = [X,\Omega^{\infty}\ko].
\end{equation}
The goal of this subsection is to carefully explain the relationships between these four different topological K-theory spaces. 

The following lemma is a well-known consequence of Quillen's group-completion theorem; see for instance \cite[Chapter IV, Theorem 7.1]{weibel-k-book}. 

\begin{lem}
    For $X$ paracompact Hausdorff, the exact structure on $\Vect_{X}$ is split, thereby inducing an equivalence
    \begin{equation}
    (\Vect_{X}^{\simeq})^{\rm gp} \xrightarrow{\sim} \K(\Vect_{X})
    \end{equation}
    which is natural in $X$.
\end{lem}

The next proposition uses Saunier's notion of a \textit{heart structure} \cite[Definition 2.3]{saunier-weight-heart}, a weaker version of a weight structure, which permits a comparison between the algebraic K-theory of stable \categories and exact \categories. 
As we will not use heart structures outside of the following proposition, we direct the reader to \cite{saunier-weight-heart} for details.

\begin{prop}\label{prop: heart structure cperf}
    Let $X\in\Top$ be compact Hausdorff.
    The pair $(\cperf_{X}^{\leq 0},\cperf^{\rm cn}_{X})$ determines a bounded heart structure on $\cperf_{X}$ with 
    \begin{equation}
    \Vect_{X} \simeq \cperf_{X}^{\heartsuit}.
    \end{equation}
    Consequently, the induced map 
    \begin{equation}
    \K(\Vect_{X}) \rightarrow \K(\cperf_{X})
    \end{equation}
    is an equivalence.
\end{prop}

\begin{proof}
    To see the pair $(\cperf_{X}^{\leq 0},\cperf^{\rm cn}_{X})$ determines a heart structure, it remains to verify that, for any $E \in \cperf_{X}$, there is a cofiber sequence 
    \begin{equation}
    A \rightarrow E \rightarrow \Sigma B
    \end{equation}
    with $A \in \cperf^{\leq 0}_{X}$ and $B\in \cperf_{X}^{\rm cn}$.
    We prove that for each $E$ there is a $d$-connective map $\varphi_{d} :A \rightarrow E$ where $A$ has Tor-amplitude $\leq d$, and we argue by induction on $d$. 
    By the compactness of $X$, we may assume that $E$ is $k$-connective for some $k \in \bbZ$, so we may take $A = 0$ when $k<d$. 
    Now assume we are given $\varphi_{d} : A_{\leq d} \rightarrow E$ $d$-connective, where $A$ has Tor-amplitude $\leq d$. 
    By Proposition~\ref{prop: map from VB}, we may choose a $d$-connective map $\psi: V[d] \rightarrow \mathrm{fib}(\varphi_{d})$, where $\mathrm{cofib}(\psi)$ is $d$-connective. 
    Let $A_{\leq d+1}$ denote the cofiber of the composite $V[d] \rightarrow \mathrm{fib}(\varphi_{d}) \rightarrow A_{\leq d}$ and note that $\varphi_{d}$ factors as 
    \begin{equation}
    A_{\leq d} \rightarrow A_{\leq d+1} \xrightarrow{\varphi_{d+1}} E. 
    \end{equation}
    Using the cofiber sequence associated to a composite, it follows $\varphi_{d+1}$ has $(d+1)$-connective cofiber. 
    Furthermore, the closure of Tor-amplitude under extensions shows $A_{\leq d+1}$ has Tor-amplitude $\leq (d+1)$.   
    
    By Lemma~\ref{lem: properties of Tor-amplitude} we have
    \begin{equation}
    \cperf_{X}^{[a,b]} = \cperf_{X}^{\leq b} \cap (\cperf_{X})_{\geq a}
    \end{equation}
    which, combined with Proposition~\ref{prop: locally bounded tor amp}, shows the heart structure is bounded. 
    Finally, we have 
    \begin{equation}
    \cperf_{X}^{\heartsuit} = \cperf_{X}^{\leq 0} \cap \cperf^{\rm cn}_{X} = \cperf^{[0,0]}_{X} \simeq \Vect_{X},
    \end{equation}
    and the equivalence on K-theory spaces follows from \cite[Theorem 0.6]{saunier-weight-heart}.
\end{proof}

\begin{rem}
    Observe, $\Vect_{\bbR^{0}}$ is equivalent to the $1$-category of finite-dimensional real vector spaces. 
    Consequently, the groupoid core is given by 
    \begin{equation}
    \Vect_{\bbR^{0}}^{\simeq} \simeq \bigsqcup_{n} \BGL_{n}(\bbR), 
    \end{equation}
    where $\BGL_{n}(\bbR)$ denotes the category with one object $\ast$ and space of endomorphisms $\GL_{n}(\bbR)$ with the \textit{discrete} topology. 
    As $\Vect_{\bbR^{0}}$ is an additive $1$-category, $\Vect_{\bbR^{0}}^{\simeq}$ is a $\einf$-monoid in $\Spaces$. 
    On the other hand, the compact-open topology provides an enrichment of $\Vect_{\bbR^{0}}$ whose groupoid core is an $\einf$-monoid equivalent to
    \begin{equation}
    \bigsqcup_{n} \BO(n). 
    \end{equation}
    Group-completing this $\einf$-monoid, we have an equivalence 
    \begin{equation}
    \left(\bigsqcup_{n} \BO(n)\right)^{\rm gp} \simeq \Omega^{\infty}\ko,
    \end{equation}
    and therefore a map of grouplike $\einf$-monoids
    \begin{equation}
    (\Vect^{\simeq}_{\bbR^{0}})^{\rm gp} \rightarrow \Omega^{\infty}\ko. 
    \end{equation}    
\end{rem}

\begin{prop}
    For $X\in\Top$, there is a natural map
    \begin{equation}
    (\Vect^{\simeq}_{X})^{\rm gp} \rightarrow \KO(X). 
    \end{equation}
    If $X$ is additionally compact Hausdorff, this map is an isomorphism on $\pi_{0}$.
\end{prop}

\begin{proof}
     Recall that $\Vect \in \Shv(\Top;\Cat_\infty)$ is obtained by sheafifying 
     \begin{equation}
     X \mapsto \Free^{\omega}_{C(X)}. 
     \end{equation}
     It follows that $\Vect^{\simeq} \in \Shv(\Top)$ is obtained by sheafifying 
     \begin{equation}
     X \mapsto (\Free^{\omega}_{C(X)})^{\simeq} \simeq \bigsqcup_{n} \BGL_{n}\big(C(X)\big), 
     \end{equation}
     where $\BGL_{n}\big(C(X)\big)$ is the classifying space of the (discrete) group $\GL_{n}\big(C(X)\big) \simeq \Map_{\Top}\big(X,\GL_{n}(\bbR)\big)$.
     The collection of groups $\{\GL_{n}(\bbR)\}_{n\geq 0}$ with block sum determines a $\bbZ_{\geq 0}$-graded commutative monoid in the category of topological groups with the Cartesian symmetric monoidal structure.      
     By colimit interchange, we obtain a nautral transformation of functors      
     \begin{equation}
     \bigsqcup_{n} \rmB\Big(\Map_{\Top}\big(-,\GL_{n}(\bbR)\big)\Big) \rightarrow \Map_{\Spaces}\bigg(\Sing(-),\bigsqcup_{n} \BO(n)\bigg);
     \end{equation}
     here, we have used the identification $\Sing\big(\GL_{n}(\bbR)\big)\simeq \rmO(n)$ to describe the target above. 
     Furthermore, because $\Sing : \Top \rightarrow \Spaces$ is symmetric monoidal, a routine verification shows this natural transformation is one of $\einf$-space valued presheaves.   
     Post-composing with the map to $\Omega^{\infty}\ko$, we obtain a natural transformation of $\einf$-space valued presheaves 
     \begin{equation}
     (\Free^{\omega}_{C(-)})^{\simeq} \rightarrow \KO.
     \end{equation}
     As the target is already a sheaf, this induces a natural transformation of sheaves valued in $\einf$-spaces:
     \begin{equation}
     \Vect^{\simeq} \rightarrow \KO;
     \end{equation}
     but the target is levelwise group-complete, so we obtain the desired natural transformation: 
     \begin{equation}
     (\Vect^{\simeq})^{\rm gp} \rightarrow \KO. 
     \end{equation}
     If $X$ is compact Hausdorff, the isomorphism on $\pi_{0}$ is classical.  
\end{proof}

In summary, we have produced a diagram of presheaves on $\Top$ valued in grouplike $\einf$-spaces:
\begin{equation}\label{equation: zig zag of K-theory}
\begin{tikzcd}
	{\K(\cperf)} & {\K(\Vect)} & {\K^{\rm add}(\Vect)} & \KO.
	\arrow[from=1-2, to=1-1]
	\arrow[from=1-3, to=1-2]
	\arrow[from=1-3, to=1-4]
\end{tikzcd}
\end{equation}
Furthermore, we have also seen that, for any $X\in\Top$ which is compact Hausdorff, the left-facing arrows are equivalences; hence, we have a map
\begin{equation}
\K(\cperf_{X}) \rightarrow \KO(X). 
\end{equation}
Let $\mathrm{CH} \subset \mathrm{kTop} \subset \Top$ denote the full subcategories of compact Hausdorff spaces and of k-spaces, respectively, and let $j : \mathrm{CH} \rightarrow \mathrm{kTop}$ denote the inclusion. 
Recall, essentially by definition, any k-space $X$ can be written as a small colimit of compact Hausdorff spaces in $\Top$.

\begin{prop}
    There is a natural transformation of functors from $\mathrm{kTop}^{\op}$ to $\mathrm{CMon^{gp}}(\Spaces)$:
    \begin{equation}
    \K(\cperf) \rightarrow \KO. 
    \end{equation}
    Consequently, we have a natural transformation of $\mathrm{CMon}(\Spaces)$-valued sheaves on $\mathrm{kTop}$,
    \begin{equation}
    [-]:\cperf^{\simeq} \rightarrow \KO,
    \end{equation}
    which splits cofiber sequences. 
\end{prop}

\begin{proof}
    By restricting and right Kan extending the zig-zag of natural transformations in \eqref{equation: zig zag of K-theory} along $j$, we obtain the following diagram of presheaves on $\mathrm{kTop}$ valued in grouplike $\einf$-spaces:
    \begin{equation}\begin{tikzcd}
	{\K(\cperf)} & {\K(\Vect)} & {\K^{\rm add}(\Vect)} & \KO \\
	{j_{\ast}j^{\ast}\K(\cperf)} & {j_{\ast}j^{\ast}\K(\Vect)} & {j_{\ast}j^{\ast}\K^{\rm add}(\Vect)} & {j_{\ast}j^{\ast}\KO}.
	\arrow[from=1-1, to=2-1]
	\arrow[from=1-2, to=1-1]
	\arrow[from=1-2, to=2-2]
	\arrow[from=1-3, to=1-2]
	\arrow[from=1-3, to=1-4]
	\arrow[from=1-3, to=2-3]
	\arrow[from=1-4, to=2-4]
	\arrow[from=2-2, to=2-1]
	\arrow[from=2-3, to=2-2]
	\arrow[from=2-3, to=2-4]
    \end{tikzcd}\end{equation}
    By our work above, all the left-facing arrows in the bottom rows are equivalences. 
    Furthermore, because any k-space is a colimit of compact spaces and $\Sing$ preserves small colimits, the rightmost vertical arrow is also an equivalence.
    Therefore, we obtain the claimed natural transformations:
    \begin{equation}
    \cperf^{\simeq} \rightarrow \K(\cperf) \rightarrow \KO. 
    \end{equation}
\end{proof}

\begin{rem}
    The transformation $[-]$ admits a convenient informal description. 
    Given $V\in\cperf_X$ for $X$ a k-space, this determines a point in the space 
    \begin{equation}
    \Gamma\big(X;j_{\ast}j^{\ast}\K(\cperf)\big) \simeq \lim_{f \in (\CH_{/
    X})^{\op}} \K(\cperf_{K}),
    \end{equation}
    where $f: K \rightarrow X$ ranges over all maps from compact Hausdorff spaces into $X$.
    Informally, this point is given by the system $([f^{\ast}V])_{f \in (\CH_{/X})^{\op}}$ which can be explicitly encoded as a section of a fibration. 
    Identifying each algebraic K-theory class $[f^{\ast}V]$ with a point in $\K^{\rm add}(\Vect_{K})$, we can produce a system 
    \begin{equation}
    \big([f^{\ast}V] : \Sing(K) \rightarrow \oko \big)_{f \in (\CH_{/X})^{\op}} \in \lim_{f\in (\CH_{/X})^{\op}} \KO(K)
    \end{equation}
    which can be glued to give a point in $\KO(X)$. 
\end{rem}

\subsection{Virtual and stable tangent bundles}\label{subsection: virtual tangent bundles}
Having better understood the relationship between perfect complexes and topological K-theory, we can now establish the main results of the present subsection. 

\begin{assu}\label{assu:paracompacthausdorff}
Throughout the rest of the present article, we assume all derived smooth manifolds $X$ have underlying topological space $\rlz{X}$ which is paracompact Hausdorff. 
Consequently, by Remark~\ref{rem: point-set topology of DSm}, we learn $\rlz{X}$ is also metrizable. 
By abuse of notation, we write $\DSm$ for the full subcategory of derived smooth manifolds which are paracompact Hausdorff. 
It is not hard to see this \category remains a perfect topological site closed under finite limits.
\end{assu}

The following proposition summarizes our previous technical work.

\begin{prop}
There is a natural transformation of functors from $\DSm^{\op}$ to $\CMongp$:
    \begin{equation}
    (-)^{\vir} : \K(\Perf) \rightarrow \KO. 
    \end{equation}
\end{prop}

In particular, given any perfect complex $E\in\Perf_X$, there is a canonical virtual bundle $E^{\vir} \in\KO(X)$. 
Given $f: X \rightarrow Y$ in $\DSm$, we say that $T^{\vir}_{X}$ is the \textit{virtual tangent bundle of $X$} and that $T^{\vir}_{f}$ is the \textit{virtual relative tangent bundle of $f$}. 
As algebraic K-theory splits cofiber sequences, we obtain the following formal properties of the virtual relative tangent bundle. 

\begin{lem}\label{lem: virtual tangent bundle properties}
\begin{enumerate}
\item Given $f : X \rightarrow Y$ in $\DSm$, there are canonical equivalences
    \begin{equation}
    T^\vir_{f} \simeq T^\vir_{X}-f^*T^\vir_Y\;\;\textrm{and}\;\;N^\vir_{f} \simeq -T^\vir_{f}.
    \end{equation}
\item Given a composite $X \xrightarrow{f} Y \xrightarrow{g}Z\in\DSm$, there is a canonical equivalence
    \begin{equation}
    T^\vir_{g\circ f}\simeq T^\vir_{f}+f^{\ast}T^\vir_{g}.
    \end{equation}
\item Given a pullback sqaure 
    \begin{equation}
    \begin{tikzcd}
    \overline{X}\equiv X\times_Y\overline{Y}\arrow[d,"\overline{g}",swap]\arrow[r,"\overline{f}"]\arrow[dr, phantom, "\lrcorner", very near start] & \overline{Y}\arrow[d,"g"] \\
    X\arrow[r,"f",swap] & Y
    \end{tikzcd}
    \end{equation}
in $\DSm$, there is a canonical equivalence 
    \begin{equation}
    T^\vir_{\overline{f}}\simeq \overline{g}^* T^\vir_f. 
    \end{equation}
\end{enumerate}
\end{lem}

To study various flavors of cobordism, we must also contend with \textit{stable vector bundles}.

\begin{recollection}\label{recollection:ko}
The Postnikov truncation $\mathrm{ko} \rightarrow \tau_{\leq 0}\mathrm{ko}\simeq \mathrm{H}\bbZ$ induces a pullback square of grouplike $\einf$-monoids:
\begin{equation}
\begin{tikzcd}
\BO \arrow[r] \arrow[d]\arrow[dr, phantom, "\lrcorner", very near start] & \Omega^{\infty}\mathrm{ko} \arrow[d,"\pi"] \\
\ast \arrow[r,"0"'] & \mathbb{Z}.
\end{tikzcd}
\end{equation}
By looping the composite ${\rm U}(1)\to{\rm U}\to{\rm U}/{\rm O}$ and applying real Bott periodicity, we obtain an $\ee{1}$-section of $\pi$
    \begin{equation}
    s: \mathbb{Z} \rightarrow \Omega^{\infty}\mathrm{ko};
    \end{equation}
observe, $s$ is not $\einf$, otherwise the first $k$-invariant of $\mathrm{ko}$ would be trivial. 
Writing $\rk\equiv s \circ \pi$, there is an induced map  
\begin{equation}
\mathrm{id} - \rk : \Omega^{\infty}\mathrm{ko} \rightarrow \BO
\end{equation}
furnishing an equivalence of $\ee{1}$-monoids:
\begin{equation}
(\pi, \id - \rk ) : \Omega^{\infty}\mathrm{ko} \xrightarrow{\sim} \bbZ \times \BO, 
\end{equation}
where the latter $\einf$-monoid now carries the split $\einf$-monoid structure rather than the Bott periodic one.
\end{recollection}

\begin{defin}
    Given $X \in \DSm$ and $\xi \in \KO(X)$, the \textit{virtual rank function of $\xi$} is defined to be 
    \begin{equation}
    \rk\xi \equiv \rk \circ \; \xi : \Sing(X) \rightarrow \oko, 
    \end{equation}
    which we may view as a function $\rlz{X} \rightarrow \bbZ$.
    We say that $\xi \in \KO(X)$ \textit{has virtual rank $n \in \bbZ$} provided $\rk\xi$ is constant with value $n$; equivalently, this means $\xi$ factors through the subspace $\{n\}\times\BO \rightarrow \oko$.
\end{defin}

\begin{rem}
    Given two $\xi,\zeta \in \KO(X)$, we have $\rk(\xi + \zeta) = \rk\xi + \rk\zeta$.
\end{rem}

\begin{notation}
    Given $E \in \Perf_{X}$, we define $\rk E \equiv \rk E^{\vir}$, and in the special cases where $E = T_{X}$ or $E = T_{f}$, we simply write $\rk X \equiv \rk T^{\vir}_{X}$ and $\rk f \equiv \rk T_{f}^{\vir}$, respectively. 
\end{notation}

\begin{example}
    When $f:X\to Y \in \Sm$ it is not difficult to see
    \begin{equation}
    \rk X(x) = \dim(T_{X,x})\;\;\text{and}\;\; \rk f(x) = \dim(T_{X,x}) - \dim(T_{Y,f(x)}),
    \end{equation}
    where $T_{X,x}$ denotes the tangent space at $x \in X$.
\end{example}

\begin{defin}
We denote by $\TKO$ the functor
    \begin{equation}
    \Map_{\Spaces}\big(\Sing(-),\BO\big):\DSm^{\op} \rightarrow\CMongp.
    \end{equation}
\end{defin}

By Recollection~\ref{recollection:ko}, we have a splitting of functors from $\DSm^{\op}$ to $\rm{Mon}^{\rm gp}$:
\begin{equation}
\KO \simeq \Map_{\Spaces}\big(\Sing(-),\bbZ\big) \times \TKO. 
\end{equation}
In particular, we can view $\xi \in \KO(X)$ as a pair 
\begin{equation}
(\rk\xi,\xi-\rk\xi) \in \Map_{\Spaces}\big(\Sing(X),\bbZ\big) \times \TKO(X). 
\end{equation}

\begin{defin}
Let $f:X\to Y \in \DSm$.
The \textit{stable tangent bundle of $X$} is
\begin{equation}
\tau_{X} \equiv T_{X}^{\vir} - \rk X \in \TKO(X).
\end{equation}
The \textit{stable relative tangent bundle of $f$} is 
\begin{equation}
\tau_{f} \equiv T_{f}^{\vir} - \rk f \in \TKO(X). 
\end{equation}
\end{defin}

\begin{rem}
    The properties established in Lemma~\ref{lem: virtual tangent bundle properties} also hold for the stable relative tangent bundle. 
\end{rem}

\section{Structure functors and 6-functor formalisms}\label{sec:structurefun}
\subsection{Correspondences and 6-functor formalisms}
\begin{defin}
    An \textit{adequate triple} is a triple $(\m{C};\m{C}_{L},\m{C}_{R})$ with $\m{C}$ an \category and $\m{C}_{L},\m{C}_{R} \subset \m{C}$ wide subcategories, such that, for every $X \rightarrow Y\in\m{C}_{R}$ and $Y' \rightarrow Y\in\m{C}_{L}$, there exists a pullback square in $\m{C}$    \begin{equation}
    \begin{tikzcd}
    X'\arrow[r]\arrow[d]\arrow[dr, phantom, "\lrcorner", very near start] & Y'\arrow[d] \\
    X \arrow[r] & Y
    \end{tikzcd}
    \end{equation}
    and, for any such pullback, $X' \rightarrow Y'\in\m{C}_{R}$ and $X' \rightarrow X\in\m{C}_{L}$. Such pullback squares are called \textit{ambigressive} in \cite{barwick,haugsengTwovariableFibrationsFactorisation2023}.
    A \textit{morphism of adequate triples},
    \begin{equation}
        F: (\m{C};\m{C}_{L},\m{C}_{R}) \rightarrow (\m{D};\m{D}_{L},\m{D}_{R}),
    \end{equation}
    is a functor $F : \m{C} \rightarrow \m{D}$ which preserves ambigressive pullbacks. 
    There is an \category of adequate triples, denoted by $\mathrm{AdTrip}$, which is Cartesian symmetric monoidal, cf. \cite[Lemma 2.4]{haugsengTwovariableFibrationsFactorisation2023}.
\end{defin}

Every adequate triple determines an \category of correspondences (or spans) $\Cor(\m{C};\m{C}_{L},\m{C}_{R})$. 
Informally, this \category has the same objects as $\m{C}$, its morphisms are diagrams of the form 
\begin{equation}
    X \xleftarrow{f} W \xrightarrow{g} Y,
\end{equation}
with $f\in\m{C}_{L}$ and $g\in\m{C}_{R}$, and composition is given by forming pullbacks of correspondences.
The assignment of an adequate triple to its \category of correspondences determines a functor 
\begin{equation}
    \Cor: \mathrm{AdTrip} \rightarrow \Cat_\infty,
\end{equation}
and there are canonical inclusions of subcategories  
\begin{equation}
    \m{C}_{L}^{\op} \rightarrow \Cor(\m{C};\m{C}_{L},\m{C}_{R}) \leftarrow \m{C}_{R}
\end{equation}
which refine to natural transformations of functors from $\mathrm{AdTrip}$ to $\Cat_\infty$.

When $\m{C}$ has a symmetric monoidal structure compatible with the adequate triple structure, the \category of correspondences inherits a symmetric monoidal structure. 

\begin{lem}\label{lem: adtrip sm gives cor sm}
    Let $\m{C}^{\otimes}$ be a symmetric monoidal \category and $(\m{C};\m{C}_{L},\m{C}_{R}) \in \mathrm{AdTrip}$. 
    If $\m{C}_{L},\m{C}_{R} \subset \m{C}$ are symmetric monoidal subcategories and $\otimes : \m{C} \times \m{C} \rightarrow \m{C}$ preserves ambigressive pullback squares, then there exists a canonical symmetric monoidal \category $\Cor(\m{C};\m{C}_{L},\m{C}_{R})^{\otimes}$ and symmetric monoidal functors
    \begin{equation}
        (\m{C}^{\op}_{L})^{\otimes} \rightarrow \Cor(\m{C};\m{C}_{L},\m{C}_{R})^{\otimes} \leftarrow \m{C}_{R}^{\otimes}.
    \end{equation}
\end{lem}

\begin{proof} 
    The hypotheses above easily imply $(\m{C};\m{C}_{L},\m{C}_{R})$ is a commutative monoid in $\mathrm{AdTrip}$; the forgetful functor $\mathrm{AdTrip} \rightarrow \Fun(\Lambda^{2}_{2},\Cat_{\infty})$ is symmetric monoidal. 
    By \cite[Theorem 2.18]{haugsengTwovariableFibrationsFactorisation2023}, the functor $\Cor$ preserves $\einf$-monoids. 
    Furthermore, the hypotheses easily imply $\m{C}^{\op}_{L}$ and $\m{C}_{R}$ are symmetric monoidal subcategories. 
\end{proof}

\begin{rem}
    See \cite[2.10, 2.14]{barwick-glasman-shahii}, \cite[Section 5]{cnossenUniversalitySpan2categories2026}, and \cite[2.3.3]{6FF} for results of a similar flavor.
\end{rem}

One special kind of adequate triple, called a geometric setup, plays a key role in the study of abstract 6-functor formalisms \cite{6FF,scholze6FF}.

\begin{defin}
A \emph{geometric setup} consists of a pair $(\m{C};E)$, where $\m{C}$ is an \category which admits finite limits and $E$ is a homotopy class of edges in $\m{C}$ which:
\begin{enumerate}
\item contains all equivalences,
\item is closed under composition and pullbacks along edges in $\m{C}$, 
\item and, for every $f:X\to Y\in E$, the (relative) diagonal,
    \begin{equation}
    \Delta_f:X\to X\times_Y X,
    \end{equation}
also lies in $E$.
\end{enumerate}
\end{defin}

\begin{rem}
If $(\m{C},E)$ is a geometric setup, it is routine to verify that $(\m{C};\m{C},E)$ is an adequate triple. 
In this case, we write $\Cor(\m{C};E) \equiv \Cor(\m{C};\m{C},E)$. 
\end{rem}

The following definition is extremely useful for constructing 6-functor formalisms \cite{6FF,scholze6FF,cnossenUniversalitySpan2categories2026}.

\begin{defin}\label{defin: suitable decomposition}
    Let $\m{C}$ be an \category with two wide subcategories $I,P \subseteq \m{C}$, and denote by $E$ the collection of maps $e$ in $\m{C}$ which factor as $e \simeq p\circ i$ for $i \in I$ and $p\in P$. 
    The pair $(I,P)$ is a \textit{suitable decomposition of $(\m{C},E)$} provided the following conditions are satisfied:
    \begin{enumerate}
        \item both $I$ and $P$ are left cancellative and closed under base change,
        \item $E$ is closed under composition,
        \item and every map in $I \cap P$ is $n$-truncated for some $n \geq -2$.\footnote{Recall, a morphism $f: X \rightarrow Y$ is $n$-truncated if composition with $f$ induces an $n$-truncated map of spaces $\Map_{\m{C}}(Z,X) \rightarrow \Map_{\m{C}}(Z,Y)$ for all $Z \in \m{C}$.}
    \end{enumerate}
\end{defin}

\begin{rem}
    If $E$ is a wide subcategory that is left cancellative and closed under base change, the pair $(\m{C}^{\simeq},E)$ forms a suitable decomposition of $(\m{C},E)$. 
    It is also easy to see that $(\m{C},E)$ is a geometric setup. 
\end{rem}

\begin{rem}
    Given a geometric setup $(\m{C},E)$, we may endow $\m{C}$ with the Cartesian symmetric monoidal structure $\m{C}^{\times}$. 
    It is not difficult to verify that $(\m{C};\m{C},E)$ determines a commutative monoid in $\mathrm{AdTrip}$.
    Therefore, $\Cor(\m{C};E)$ admits a canonical symmetric monoidal structure $\Cor(\m{C};E)^{\otimes}$ by Lemma~\ref{lem: adtrip sm gives cor sm}.
\end{rem}

\begin{defin}
A \emph{3-functor formalism} is a lax symmetric monoidal functor
    \begin{equation}
    D:\Cor(\m{C};E)\to\widehat{\Cat}_\infty
    \end{equation} 
which, in particular, encodes the following functors.
\begin{enumerate}
\item Let $f:X\to Y\in E$ and consider the correspondence 
    \begin{equation}
    X\xleftarrow{\id}X\xrightarrow{f}Y.
    \end{equation}
We refer to the induced map 
    \begin{equation}
    f_!:D(X)\to D(Y)
    \end{equation}
as an \emph{exceptional pushforward}.
\item Let $f:X\to Y\in\m{C}$ and consider the correspondence 
    \begin{equation}
    Y\xleftarrow{f}X=X.
    \end{equation}
We refer to the induced map 
    \begin{equation}
    f^*:D(Y)\to D(X)
    \end{equation}
as a \emph{pullback}.
\item Given $\Delta:X\to X\times X\in\m{C}$, we refer to the induced map 
    \begin{equation}
    D(X)\otimes D(X)\to D(X\times X)\xrightarrow{\Delta^*}D(X)
    \end{equation}
as an \emph{internal tensor product}. Thus, for every $A\in D(X)$, we have a functor 
    \begin{equation}
    (-)\otimes A:D(X)\to D(X).
    \end{equation}
\end{enumerate}
We say that $D$ is a \emph{6-functor formalism} provided that the functors $f_!$, $f^*$, and $-\otimes A$ admit right adjoints. 
We denote the right adjoints of $f_!$ and $f^*$ by $f^!$ and $f_*$, which we refer to as an \emph{exceptional pullback} and \emph{pushforward}, respectively.
\end{defin}

\begin{rem}
In the sequel, we will consider lax symmetric monoidal functors $D: \Cor(\m{C};\m{C}_{L},\m{C}_{R}) \rightarrow \bigcat$, where $(\m{C};\m{C}_{L},\m{C}_{R})$ is an adequate triple. 
Furthermore, these functors will behave similarly to abstract 6-functor formalisms, but do not quite fit the definition above. 
\end{rem}

In the present article, we will be chiefly concerned with the 6-functor formalism on locally compact Hausdorff spaces constructed by Volpe \cite{volpe} (which was previously studied in great depth by Kashiwara-Schapira \cite{Kashiwara-Schapira}).
We denote by $\LCH\subset\Top$ the full subcategory of locally compact Hausdorff spaces. Recall, we say $f:X\to Y\in\LCH$ is \emph{proper} if, for any compact subset $K\subset Y$, $f^{-1}(K)\subset X$ is compact; this notion is left cancelative and stable under composition and base change, and we denote by $\mathcal{P}$ the collection of proper maps in $\LCH$. 
Therefore, we obtain an \category of correspondences $\Cor(\LCH;\calP)$ which forms a wide subcategory of the symmetric monoidal \category $\Cor(\LCH)$.
    
\begin{thm}[Volpe]\label{theorem: Volpe 6FF}
    Let $R \in \CAlg(\Sp)$. There is a 6-functor formalism
    \begin{equation}
    \Shv(-;R) \equiv \Shv(-;\Mod_{R}) :\Cor(\LCH)\to\widehat{\Cat}_{\infty}
    \end{equation}
    with the following properties.
    \begin{enumerate}
        \item For any $X \in \LCH$, the $\infty$-category $\Shv(X;R)$ is stable presentably symmetric monoidal.
        \item For any $f:W\to X\in\LCH$, the induced functor
        \begin{equation}
        f^{\ast} : \Shv(X;R)\to\Shv(W;R)
        \end{equation}
        is given by pullback of sheaves. In fact, $f^*$ is colimit-preserving and symmetric monoidal, thereby admitting a right adjoint $f_*$.
        \item For any $g:W\to Y$, the induced functor 
        \begin{equation}
        g_!:\Shv(W;R) \rightarrow \Shv(Y;R)
        \end{equation}
        is given by direct image with compact support, which admits a right adjoint $g^{!}$ given by exceptional inverse image. 
        Furthermore, if $g \in \calP$, there is a natural equivalence $g_! \simeq g_*$
    \end{enumerate}
\end{thm}

\begin{rem}
We would like to emphasize that the results from Volpe's thesis \cite{volpe} are considerably more general than the theorem stated above. 
\end{rem}

\subsection{Structure functors and correspondences}
The following definition is heavily inspired by Dold \cite{doldGeometricCobordismFixed1978}.

\begin{defin}
    Let $\m{C}$ be an \category which admits finite limits and finite coproducts.  
    By a \textit{structure functor on $\m{C}$} we mean a product-preserving functor 
    \begin{equation}
    F : \m{C}^{\op} \rightarrow \CMongp. 
    \end{equation}
    A \textit{morphism of structure functors on $\m{C}$} is a natural transformation $\eta : F \rightarrow G$.
\end{defin}

Given a structure functor $F : \m{C}^\op \rightarrow \CMongp$, we can deloop it to produce another structure functor
\begin{equation}
    \rmB F : \m{C}^{\op} \rightarrow \CMongp
\end{equation}
and, as $* \in \CMongp$ is initial, there is an essentially unique morphism of structure functors
\begin{equation}
\const(*) \rightarrow \rmB F.
\end{equation}
Post-composing with the forgetful functor to $\Spaces$, we obtain
\begin{equation}
    \rmB F : \m{C}^{\op} \rightarrow \Spaces. 
\end{equation}

\begin{defin}
We define the \textit{$F$-structuring of $\m{C}$} to be
    \begin{equation}
    \pi_{F} : \m{C}_F\equiv\int^{\mathrm{c}}_{\m{C}} \rmB F \rightarrow \m{C}.
    \end{equation}
Note that $\const(*) \rightarrow \rmB F$ induces a section $\sigma_{F}: \m{C} \rightarrow \m{C}_{F}$ of $\pi_{F}$. 
\end{defin}

\begin{rem}
We can describe $\m{C}_{F}$ informally as follows. 
\begin{enumerate}
\item The objects of $\m{C}_{F}$ are pairs 
    \begin{equation}
    (X,*_X),
    \end{equation}
where $*_X$ denotes the (essentially) unique object in $\rmB F(X)$; in particular, we can simply view objects of $\m{C}_F$ as objects in $\m{C}$.
\item A morphism in $\m{C}_{F}$,
    \begin{equation}
    (X,\ast_{X}) \rightarrow (Y,\ast_{X}),
    \end{equation}
consists of a morphism $f :X \rightarrow Y\in\m{C}$ together with a morphism $\ast_{X} \rightarrow \rmB F(f)(\ast_{Y}) = \ast_{X}$; the latter is simply a point in
    \begin{equation}
    \alpha \in F(X) \simeq \mathrm{End}_{\rmB F(X)}(\ast_{X}).
    \end{equation}
In particular, we can view a morphism in $\m{C}_F$ as a pair 
    \begin{equation}
    \big(f:X\to Y\in\m{C},\alpha\in F(X)\big).
    \end{equation}
\item Composition in $\m{C}_F$ is given by
    \begin{equation}
    \left(X \xrightarrow{(f,\alpha)} Y \xrightarrow{(g,\beta)} Z \right)\mapsto \left(X\xrightarrow{\big(g\circ f, \alpha+f^{\ast}(\beta)\big)}Z \right);
    \end{equation}
the identity morphism is $(\id_{X},0): X \rightarrow X$.
\end{enumerate}
\end{rem}

\begin{rem}
    The essential image of $\sigma_{F}$ consists of the wide subcategory spanned by morphisms in $\m{C}_{F}$ of the form $(f,0) : X \rightarrow Y$. 
    By abuse of notation, we will write $\m{C}$ for this wide subcategory of $\m{C}_{F}$. 
\end{rem}

The mapping spaces in $\m{C}_{F}$ are simple to compute, and the following lemma justifies the informal description above.

\begin{lem}\label{lem: mapping spaces in C_F}
    For any $X,Y \in \m{C}$, there is a canonical equivalence 
    \begin{equation}
    \Map_{\m{C}_{F}}(X,Y) \xrightarrow{\sim} \Map_{\m{C}}(X,Y) \times F(X).
    \end{equation}
\end{lem}

\begin{proof}
    There is a fiber sequence 
    \begin{equation}
    F(X) \rightarrow \Map_{\m{C}_{F}}(X,Y) \rightarrow \Map_{\m{C}}(X,Y)
    \end{equation}
    via \cite[Proposition 2.4.4.2]{lurieHigherToposTheory2009}. 
    Moreover, there is a section $\m{C} \rightarrow \m{C}_{F}$ of $\pi_{F}$ by the discussion above; this provides a splitting of the map on the right of the aforementioned fiber sequence, whence the claim.
\end{proof}

\begin{lem}\label{lem: C_F admits pullbacks and coproducts}
Let $F$ be a structure functor on $\m{C}$. The \category $\m{C}_{F}$ admits pullbacks and finite coproducts. 
\end{lem}

\begin{proof}
    The first claim follows immediately from \cite[\href{https://https://kerodon.net/tag/02KT}{Tag 02KT}]{kerodon}.
    To see that $\m{C}_{F}$ admits finite coproducts, consider the cocone 
    \begin{equation}
    X \xrightarrow{(i_{X},0)} X \sqcup Y \xleftarrow{(i_{Y},0)} Y 
    \end{equation}
    which induces a map
    \begin{equation}
    \Map_{\m{C}_{F}}(X \sqcup Y, Z) \rightarrow \Map_{\m{C}_{F}}(X,Z) \times \Map_{\m{C}_{F}}(Y,Z)
    \end{equation}
    natural in $Z \in \m{C}_{F}$. 
    It is easy to see this map is an equivalence by Lemma~\ref{lem: mapping spaces in C_F}.     
\end{proof}

\begin{rem}\label{rem: C_F no terminal object}
Note, $\m{C}_{F}$ does not necessarily have a terminal object. This is because $\mathrm{End}_{\m{C}_{F}}(\ast_{\m{C}}) \simeq F(\ast_{\m{C}})$, where $\ast_{\m{C}}\in\m{C}$ is terminal.
However, $\m{C}_{F}$ does have an initial object $\varnothing_{\m{C}}$ as any object $X$ receives an essentially unique morphism 
\begin{equation}
\varnothing_{\m{C}} \xrightarrow{(i,0)} X,
\end{equation}
where $0 \in F(\varnothing_{\m{C}}) \simeq \ast$. 
\end{rem}

For $\m{C}$ as above, there is a canonical Cartesian symmetric monoidal structure $\m{C}^{\times} \rightarrow \Fin_\ast$, in the sense of \cite[2.4.1]{lurieHigherAlgebra2017}, 
as well as a canonical coCartesian symmetric monoidal structure $\m{C}^{\op,\sqcup} \rightarrow \Fin_\ast$ in the sense of \cite[2.4.3]{lurieHigherAlgebra2017}. 
Furthermore, by \cite[\href{https://kerodon.net/tag/04D0}{Tag 04D0}]{kerodon}, there is an equivalence between $\m{C}^{\op,\sqcup} \rightarrow \Fin_\ast$ and the dual coCartesian fibration obtained from $\m{C}^{\times} \rightarrow \Fin_\ast$:
\begin{equation}
    \m{C}^{\times,\vee} \rightarrow \Fin_\ast.
\end{equation}
By the discussion in \cite[Subsection 6.1]{keenan-peroux}, the \category of lax symmetric monoidal functors $\m{C}^{\times,\vee} \rightarrow \m{D}^{\otimes}$ is canonically equivalent to the \category of functors $\m{C}^{\op} \rightarrow \CAlg(\m{D})$.
Viewing $\CMongp$ as a symmetric monoidal \category under the Cartesian product, the equivalence $\CAlg(\CMongp) \simeq \CMongp$ implies the following lemma. 

\begin{lem}
    Any structure functor $F: \m{C}^{\op} \rightarrow \CMongp$ admits an essentially unique extension to a lax symmetric monoidal functor 
    \begin{equation}
    F^{\times} : \m{C}^{\times,\vee} \rightarrow (\CMongp)^{\times}. 
    \end{equation}
\end{lem}

\begin{rem}
    Informally, $F^{\times}$ encodes the pairing 
    \begin{equation}
    F(X) \times F(Y) \rightarrow F(X\times Y)
    \end{equation}
    given by $(\alpha,\beta) \mapsto \pr_{1}^{\ast}\alpha + \pr_{2}^{\ast}\beta$, and we write $\alpha \times \beta$ for this point in $F(X\times Y)$.
\end{rem}

\begin{prop}\label{prop: C_F symmetric monoidal}
    If $F$ is a structure functor on $\m{C}$, then $\m{C}_{F}$ admits a canonical symmetric monoidal structure $\m{C}_{F}^{\otimes} \rightarrow \mathrm{Fin}_{\ast}$ together with a symmetric monoidal lifts of $\pi_{F}$ and $\sigma_{F}$:
    \begin{equation}
    \pi_{F}^{\otimes} : \m{C}^{\otimes}_{F} \rightarrow \m{C}^{\times} \ \ \text{and} \ \ \sigma_{F}^{\otimes} : \m{C}^{\times} \rightarrow \m{C}_{F}^{\otimes}. 
    \end{equation}
    Furthermore, if $\eta : F \rightarrow G$ is a morphism of structure functors on $\m{C}$, there is an induced symmetric monoidal functor 
    \begin{equation}
    \m{C}_{F}^{\otimes} \rightarrow \m{C}^{\otimes}_{G}
    \end{equation}
    whose underlying functor preserves pullbacks and finite coproducts.
    
\end{prop}

\begin{proof}
    First, recall that, since $\Spaces$ is a Cartesian symmetric monoidal \category, there is a Cartesian structure $\Pi :\Spaces^{\times} \rightarrow \Spaces$ in the sense of \cite[Proposition 2.4.1.5]{lurieHigherAlgebra2017}. 
    As $F$ is multiplicative, so too is $\rmB F$, and this yields a functor
    \begin{equation}
    \Pi\circ\rmB F^{\times} : \m{C}^{\times,\vee} \rightarrow \Spaces^{\times} \rightarrow \Spaces. 
    \end{equation}
    By unstraightening, we obtain a commutative diagram of coCartesian fibrations
    \begin{equation}
    \begin{tikzcd}
    {\int_{\m{C}^{\times,\vee}}^{\rm cc} (\Pi\circ \rmB F^{\times}) } \arrow[rr,"{\pi^{\otimes,\vee}_{F}}"]\arrow[dr] & & {\m{C}^{\times,\vee}} \arrow[dl] \\
    & {\Fin_\ast}, &
    \end{tikzcd}
    \end{equation}
    where the underlying functor of $\pi^{\otimes,\vee}_{F}$ is given by 
    \begin{equation}
    \pi_{F}^{\op} : \m{C}_{F}^{\op} \simeq \int_{\m{C}^{\op}}^{\rm cc} \rmB F \rightarrow \m{C}^{\op}. 
    \end{equation}
    Dualizing the two symmetric monoidal \categories and the functor $\pi_{F}^{\otimes,\vee}$ yields a symmetric monoidal functor of 
    \categories
    \begin{equation}
    \m{C}_{F}^{\otimes} \equiv \Bigg(\int^{\rm cc}_{\m{C}^{\times,\vee}} p\circ \rmB F^{\times}\Bigg)^{\vee} \rightarrow \m{C}^{\times}
    \end{equation}
    whose underlying functor is clearly the Cartesian fibration $\pi_{F}$.     
    Given a morphism of multiplicative structure functors, $\eta$, a similar argument yields the desired symmetric monoidal functor $\m{C}_{F}^{\otimes} \rightarrow \m{C}_{G}^{\otimes}$. 
    This also establishes the claim regarding $\sigma_{F}$.

    The functor $\int^{\rm c}_{\m{C}} \rmB\eta : \m{C}_{F} \rightarrow \m{C}_{G}$ preserves pullbacks immediately by \cite[\href{https://https://kerodon.net/tag/02KT}{Tag 02KT}]{kerodon}. 
    Furthermore, because $\eta_{X} : F(X) \rightarrow G(X)$ is a map of grouplike commutative monoids, it is clear that the colimit cocone 
    \begin{equation}
    X \xrightarrow{(i_{X},0)} X \sqcup Y \xleftarrow{(i_{Y},0)} Y 
    \end{equation}
    in $\m{C}_{F}$ is sent to a colimit cocone in $\m{C}_{G}$. 
\end{proof}

\begin{rem}\label{rem: C_F distributive}
    While $\m{C}^{\otimes}_{F}$ and $\pi_{F}$ are symmetric monoidal, we stress that $\m{C}^{\otimes}_{F}$ is \textit{not} the Cartesian symmetric monoidal structure by Remark~\ref{rem: C_F no terminal object}.  
    In spite of this, the monoidal product is harmonious with the coproduct as we now explain. 
    Let $X \in \m{C}_{F}$, $\big\{(f_i,\alpha_{i}): Y_{i} \rightarrow Z_{i}\big\}_{i \in I}$ a finite set of morphisms in $\m{C}_{F}$, and $\alpha = (\alpha_{i})_{i\in I} \in F(\sqcup_{i\in I}Y_{i})$. 
    There is a commutative square
    \begin{equation}\begin{tikzcd}
	{\bigsqcup_{i \in I} (X\otimes Y_{i})} & {X\otimes\left(\bigsqcup_{i\in I} Y_{i}\right)} \\
	{\bigsqcup_{i\in I}(X\otimes Z_{i})} & {X\otimes\left(\bigsqcup_{i\in I} Z_{i}\right)},
	\arrow["\sim", from=1-1, to=1-2]
	\arrow["{\bigsqcup_{i\in I}(\mathrm{id}_{X}\otimes f_{i},\pr^{\ast}_{2}\alpha)}"', from=1-1, to=2-1]
	\arrow["{(\mathrm{id}_{X}\otimes f,\pr^{\ast}_{2}\alpha)}", from=1-2, to=2-2]
	\arrow["\sim"', from=2-1, to=2-2]
    \end{tikzcd}\end{equation}
where the horizontal arrows are the equivalences in $\m{C}_{F}$ determined by the interchange equivalence in $\m{C}$ together with $0 \in F(\bigsqcup_{i \in I} X\times Y_{i})$.
\end{rem}

\subsection{Structured correspondences and 6-functor formalisms}
In this technical subsection, we explain how to construct 6-functor formalisms which arise from twisting by a structure functor.

\begin{convention}\label{convention: additive GS}
    Throughout this subsection, we fix an \category $\m{C}$ and a wide subcategory $E \subset \m{C}$ satisfying the following properties: 
    \begin{itemize}
        \item[(a)] $\m{C}$ admits finite limits, finite coproducts, and products distribute over finite coproducts;
        \item[(b)] and $E \subset \m{C}$ preserves pullbacks and coproducts
    \end{itemize}
    The following are immediate consequences: 
    \begin{enumerate}
        \item $(\m{C},E)$ is a geometric setup (in particular, the pair $(\m{C}^{\simeq},E)$ is a suitable decomposition for $E$),
        \item and $\m{C}$ admits a Cartesian symmetric monoidal structure with $E \subset \m{C}$ a symmetric monoidal subcategory.
    \end{enumerate}
\end{convention}

\begin{example}
    Let $\calP \subset \DSm$ denote the wide subcategory spanned by morphisms $p : X \rightarrow Y\in$ whose realization $\rlz{p}$ is proper.
    A routine verification shows $\calP \subset \DSm$ satisfies Convention~\ref{convention: additive GS}. 
\end{example}

\begin{notation}
Let $F$ be a structure functor on $\m{C}$. 
    We denote by $E_{F} \subset \m{C}_{F}$ the wide subcategory of morphisms $(f,\alpha) : X \rightarrow Y$ such that $f \in E$. Also, recall that, as above, we can view $\m{C}$ as a wide subcategory of $\m{C}_{F}$. 
\end{notation}

\begin{lem}
    Let $F$ be a structure functor on $\m{C}$. 
    \begin{enumerate}
        \item $(\m{C}_{F},E_{F})$ is a geometric setup. 
        \item $(\m{C}^{\simeq}, E_{F})$ is a suitable decomposition for $(\m{C}_{F},E_{F})$.
        \item $\Cor(\m{C}_{F};E_{F})$ is a semiadditive \category. 
    \end{enumerate}
\end{lem}

\begin{proof}
    (1) follows from Lemma~\ref{lem: C_F admits pullbacks and coproducts} combined with our assumptions in Convention~\ref{convention: additive GS}, and (2) is a routine check of the definitions.     
    
    Semi-additivity proceeds as follows. By (1) combined with \cite[Proposition 2.2.9]{6FF}, $\Cor(\m{C}_{F};E_{F})$ is an \category. 
    By \cite[Proposition 2.3]{gepnerUniversalityMultiplicativeInfinite2015} it suffices to show the homotopy category of $\Cor(\m{C}_{F},E_{F})$ is semiadditive; this is a routine argument.   
\end{proof}

\begin{lem}\label{lem: C_F product presv E_F}
    $E_{F} \subset \m{C}_{F}$ is a symmetric monoidal subcategory and the tensor product $\otimes : \m{C}_{F} \times \m{C}_{F} \rightarrow \m{C}_{F}$ preserves pullbacks along morphisms in $E_{F} \times E_{F}$. 
\end{lem}

\begin{proof}
    Both claims follow easily from the fact that $E$ is closed under Cartesian products in $\m{C}$.
\end{proof}

\begin{defin}
    We define 
    \begin{equation}
    \Cor(\m{C}_{F},E_{F})^{\otimes} 
    \end{equation}
    to be the symmetric monoidal \category obtained from  Lemmas~\ref{lem: adtrip sm gives cor sm} and \ref{lem: C_F product presv E_F}.
    By construction, there are symmetric monoidal functors 
    \begin{equation}
    (\m{C}^{\op}_{F})^{\otimes} \rightarrow \Cor(\m{C}_{F},E_{F})^{\otimes} \leftarrow E_{F}^{\otimes}. 
    \end{equation}
\end{defin}

\begin{rem}\label{remark: tensor product on Cor(C_F;E_F)}
    Informally, the tensor product $\otimes : \Cor(\m{C}_{F};E_{F})^{\times 2} \rightarrow \Cor(\m{C}_{F};E_{F})$ sends a pair of objects $(X,Y)$ to $X \times Y$ and a pair of correspondences to 
    \begin{equation}
    \begin{tikzcd}
	& {W_{0}\times W_{1}} & \\
	{X_{0}\times X_{1}} && {Y_{0}\times Y_{1}}.
	\arrow["{(f_{0}\times f_{1},\alpha_{0}\times \alpha_{1})}"', from=1-2, to=2-1]
	\arrow["{(g_{0}\times g_{1},\beta_{0}\times \beta_{1}})", from=1-2, to=2-3]
    \end{tikzcd}
    \end{equation}
    The terminal object of $\m{C}$ is the unit for the monoidal product on $\Cor(\m{C}_{F};E_{F})$. 
    Furthermore, this monoidal product preserves direct sums by Remark~\ref{rem: C_F distributive}.
\end{rem}

\begin{defin}
    We say a 3-functor formalism $D : \Cor(\m{C};E) \rightarrow \bigcat$ is \textit{additive} provided the canonical map 
    \begin{equation}
    D(X \sqcup Y) \rightarrow D(X) \times D(Y)
    \end{equation}
    is an equivalence for all $X,Y \in \m{C}$.
\end{defin}

Let $\Pic(D) : \m{C}^{\op} \rightarrow \CMongp$ denote the functor given by $X \mapsto \Pic\big(D(X)\big)$. 
The following lemma is immediate. 

\begin{lem}
    If $D : \Cor(\m{C};E) \rightarrow \bigcat$ is an additive 3-functor formalism, then $\Pic(D)$ is a structure functor.
\end{lem}

Let $D$ be an additive 3-functor formalism on $\Cor(\m{C};E)$ and $\eta : F \rightarrow \Pic(D)$ a morphism of structure functors. 
From this data, we will construct a 3-functor formalism
\begin{equation}
    D_{F} : \Cor(\m{C}_{F};E_{F}) \rightarrow \bigcat
\end{equation}
which ``twists'' the pullback and pushforward functors of $D$ by invertible objects in the essential image of $\eta: F \rightarrow \Pic(D)$.
We first construct the contravariant portion of this putative 3-functor formalism; this will require some technical results from higher category theory.

\begin{recollection}\label{recollection: CAlg(Mod(C))}
Viewing $\bigcat$ as a Cartesian symmetric monoidal \category, we let $\Mod(\bigcat)$ denote the (very large) \category of pairs $(\m{A},\m{M})$, where $\m{A}$ is a symmetric monoidal \category and $\m{M}$ is a module over $\m{A}$ \cite[Definition 3.3.3.8, Definition 4.5.1.1]{lurieHigherAlgebra2017};
this is itself a symmetric monoidal \category with respect to the product
\begin{equation}
    (\m{A},\m{M})\boxtimes (\m{B},\m{N}) = (\m{A}\times\m{B},\m{M}\times\m{N}).
\end{equation}
By \cite[Corollary 3.4.1.5]{lurieHigherAlgebra2017} we have a canonical identification 
\begin{equation}
\CAlg\big(\Mod(\bigcat)\big) \simeq \CAlg(\bigcat)^{[1]}.
\end{equation}
\end{recollection}

\begin{rem}\label{rem: A to (A,*)}
    Because $\bigcat$ is Cartesian symmetric monoidal, $\CAlg(\bigcat)$ has a terminal object given by $\ast \in \bigcat$.
    This implies there is an essentially unique symmetric monoidal functor  
    \begin{equation}
    \CAlg(\bigcat) \rightarrow\CAlg(\bigcat)^{[1]} 
    \end{equation}
    which sends $\m{A} \mapsto (\m{A} \rightarrow \ast)$; under the identification above, this determines a functor $\m{A} \mapsto (\m{A},\ast)$.
\end{rem}

\begin{construction}
    Given $(\m{A},\m{M}) \in \Mod(\bigcat)$, we can consider its bar construction:
    \begin{equation}
        \m{M}\quot \m{A} = \colim_{\Delta^{\op}} \mathrm{Bar}(\ast,\m{A},\m{M}) \in \bigcat. 
    \end{equation}
    We will show this can be refined to a symmetric monoidal functor 
    \begin{equation}
    -\quot -:\Mod(\bigcat) \rightarrow \bigcat.
    \end{equation}
    By \cite[Theorem 4.4.2.8, Remark 4.6.3.4]{lurieHigherAlgebra2017}, the relative tensor product 
    \begin{equation}
    \Mod_{\rm R,L}(\bigcat) \equiv\RMod(\bigcat)\times_{\Alg(\bigcat)} \LMod(\bigcat) \rightarrow \bigcat
    \end{equation}
    is symmetric monoidal because the Cartesian product on $\bigcat$ is compatible with sifted colimits. 
    Therefore, it will be enough to produce a symmetric monoidal functor 
    \begin{equation}
    \Mod(\bigcat) \rightarrow \Mod_{\rm R,L}(\bigcat)
    \end{equation}
    sending a pair $(\m{A},\m{M})$ to the triple $(\ast,\m{A},\m{M})$. 
    By \cite[Proposition 4.5.1.4]{lurieHigherAlgebra2017} and the theory in \cite[3.2.4]{lurieHigherAlgebra2017}, there is a pullback diagram of symmetric monoidal \categories: 
    \begin{equation}
    \begin{tikzcd}
	{\Mod(\bigcat)}\arrow[dr, phantom, "\lrcorner", very near start] & {\LMod(\bigcat)} \\
	{\CAlg(\bigcat)} & {\Alg(\bigcat)}.
	\arrow[from=1-1, to=1-2]
	\arrow[from=1-1, to=2-1]
	\arrow[from=1-2, to=2-2]
	\arrow[from=2-1, to=2-2]
    \end{tikzcd}
    \end{equation}
    Note that the same holds for $\RMod$ in place of $\LMod$. 
    By Remark~\ref{rem: A to (A,*)}, there is a symmetric monoidal section 
    \begin{equation}
    \CAlg(\bigcat) \rightarrow \Mod(\bigcat)
    \end{equation}
    sending $\m{A} \mapsto (\m{A},\ast)$; pre-composing with the forgetful functor, we obtain a symmetric monoidal functor 
    \begin{equation}
    \Mod(\bigcat) \rightarrow \CAlg(\bigcat) \rightarrow \Mod(\bigcat) \rightarrow \RMod(\bigcat).
    \end{equation}
    By the definition of $\Mod_{\rm R,L}(\bigcat)$, we obtain our desired symmetric monoidal functor which sends $(\m{A},\m{M})$ to $\m{M}\quot\m{A}$.
\end{construction}

\begin{rem}\label{rem: describing M//A}
    If $\m{A}$ is an $\einf$-monoid in $\Spaces$, then $\ast\quot \m{A}$ can be identified with $\rmB\m{A}$, the \category with a single object and space of endomorphisms $\m{A}$.
    Similarly, if $\m{M} \in \Mod_{\m{A}}(\bigcat)$, the essentially unique map $(\m{A},\m{M}) \rightarrow (\m{A},\ast)$ induces a functor 
    \begin{equation}
        \m{M}\quot \m{A} \rightarrow \ast\quot \m{A}.
    \end{equation}
    This is a coCartesian fibration classifying the left action of $\m{A}$ on $\m{M}$.
    Therefore, we may informally view $\m{M}\quot \m{A}$ as the \category where: 
    \begin{enumerate}
        \item the objects are the same as those of $\m{M}$,
        \item and a morphism $M\to N$ consists of an object $\alpha \in \m{A}$ together with a morphism $\theta: \alpha \otimes M \rightarrow N$.
        (Given such a pair $(\alpha,\theta)$, it is a coCartesian lift of $\alpha$ just in the case $\theta$ is an equivalence.)
    \end{enumerate}
    If $\m{A}$ is grouplike, then $\m{M}\quot \m{A} \rightarrow \ast \quot \m{A}$ is also a Cartesian fibration.

\end{rem}

\begin{construction}[Contravariant functoriality]
The map of structure functors $\eta: F \rightarrow \Pic(D)$ induces a natural transformation $F \rightarrow D$ in $\Fun\big(\m{C}^{\op},\CAlg(\bigcat)\big)$. 
By reasoning similar to that in Remark~\ref{rem: A to (A,*)}, there is a morphism in $\CAlg(\bigcat)^{[1]}$, from $F \rightarrow D$ to $F \rightarrow \ast$, given by the square 
\begin{equation}
\begin{tikzcd}
    F \arrow[d,"\mathrm{id}"'] \arrow[r] & D \arrow[d] \\
    F \arrow[r] & \ast. 
\end{tikzcd}
\end{equation}
By Recollection~\ref{recollection: CAlg(Mod(C))}, we obtain an essentially unique lax symmetric monoidal functor 
\begin{equation}
    (F,D)^{\otimes} : \m{C}^{\op,\sqcup} \rightarrow \Mod(\bigcat)^{\otimes}
\end{equation}
and an essentially unique lax symmetric monoidal transformation 
\begin{equation}
    (F,D)^{\otimes} \rightarrow (F,\ast)^{\otimes}.
\end{equation}
Post-composing with the bar construction yields another lax symmetric monoidal natural transformation 
\begin{equation}
    (D\quot F)^{\otimes} \rightarrow (\ast\quot F)^{\otimes}
\end{equation}
of functors from $\m{C}^{\op}$ to $\bigcat$. 
After coCartesian unstraightening, we obtain a functor
\begin{equation}
    \Phi_{F}^{\otimes} : \left(\int^{\rm cc}_{\m{C}^{\op}} D\quot F\right)^{\otimes} \equiv \int_{\m{C}^{\op,\sqcup}}^{\rm cc} (D\quot F)^{\otimes} \rightarrow \int_{\m{C}^{\op,\sqcup}}^{\rm cc} (\ast\quot F)^{\otimes} = \m{C}_{F}^{\otimes,\vee}
\end{equation}
between coCartesian fibrations over $\m{C}^{\op,\sqcup}$ (and thus over $\Fin_\ast$); we denote the underlying functor by
\begin{equation}
    \Phi_{F} : \int^{\rm cc}_{\m{C}^{\op}} D\quot F \rightarrow \m{C}_{F}^{\op}. 
\end{equation}

\end{construction}

\begin{lem}\label{lem: D_F lax sm}
    The functor $\Phi_{F}^{\otimes}$ is a coCartesian fibration of \operads.
\end{lem}

\begin{proof} 
    By a well-known lemma regarding coCartesian fibrations (see \cite[Lemma 2.10]{ramziMONOIDALGROTHENDIECKCONSTRUCTION2026}) it is enough to establish the following two claims.
    \begin{enumerate}
        \item For each $\{X_i\}_{i \in I} \in \m{C}^{\op,\sqcup}$, the functor 
        \begin{equation}
        \prod_{i \in I} D(X_i)\quot F(X_i) \rightarrow \prod_{i \in I} \ast \quot F(X_i)
        \end{equation}
        is a coCartesian fibration. 
        \item For each active morphism $(X,Y)\rightarrow Z$ in $\m{C}^{\op,\sqcup}$ (given by a pair of maps $(f,g) : Z \rightarrow X \times Y$), the functor $\boxtimes_{f,g}$ in the diagram below,
    \begin{equation}
\begin{tikzcd}
	{D(X)\quot F(X) \times D(Y)\quot F(Y)} & {D(Z)\quot F(Z)} \\
	{\ast\quot F(X) \times \ast \quot F(Y)} & {\ast\quot F(Z)},
	\arrow["{\boxtimes_{f,g}}", from=1-1, to=1-2]
	\arrow["{(\Phi_{X},\Phi_{Y})}"', from=1-1, to=2-1]
	\arrow["{\Phi_{Z}}", from=1-2, to=2-2]
	\arrow[from=2-1, to=2-2]
\end{tikzcd}
    \end{equation}
    carries $(\Phi_{X},\Phi_{Y})$-coCartesian arrows to $\Phi_{Z}$-coCartesian arrows. 
    \end{enumerate}
    The first claim is immediate from Remark~\ref{rem: describing M//A}. 
    For the second claim, fix a $(\Phi_{X},\Phi_{Y})$-coCartesian morphism
    \begin{equation}
    \big((\alpha,\theta_{X}), (\beta,\theta_{Y})\big) : (\m{F}_{X},\m{F}_{Y}) \rightarrow (\m{G}_{X},\m{G}_{Y})
    \end{equation}
    and note this is sent to the morphism
    \begin{equation}
    (f^{\ast}\alpha + g^{\ast}\beta, f^{\ast}\theta_{X} \boxtimes g^{\ast}\theta_{Y}) : f^{\ast}\m{F}_{X}\boxtimes g^{\ast}\m{F}_{Y} \rightarrow f^{\ast}\m{G}_{X}\boxtimes g^{\ast}\m{G}_{Y},
    \end{equation}
    where $f^{\ast}\theta_{X} \boxtimes g^{\ast}\theta_{Y}$ is the equivalence 
    \begin{equation}
    \big(f^{\ast}\m{F}_{X} \otimes \eta(f^{\ast}\alpha) \big)\boxtimes \big(g^{\ast}\m{F}_{Y} \otimes \eta(g^{\ast}\beta)\big) \simeq f^{\ast}\m{G}_{X}\boxtimes g^{\ast}\m{G}_{Y}. 
    \end{equation}
    However, as we have 
    \begin{equation}
    \big(f^{\ast}\m{F}_{X} \otimes \eta(f^{\ast}\alpha) \big)\boxtimes \big(g^{\ast}\m{F}_{Y} \otimes \eta(g^{\ast}\beta)\big) \simeq \eta(f^{\ast}\alpha+g^{\ast}\beta) \otimes (f^{\ast}\m{F}_{X}\boxtimes g^{\ast}\m{F}_{Y}),
    \end{equation}
    the claim follows. 
    To complete the proof, we observe that $\Phi_{F}^{\otimes}$ satisfies \cite[Proposition 2.1.2.12 (b)]{lurieHigherAlgebra2017} essentially by construction. 
\end{proof}

\begin{defin}
    We define 
    \begin{equation}
    D_{F} : \m{C}_{F}^{\op} \rightarrow \bigcat \ \ \text{and} \ \ D_{F}^{\otimes} : \m{C}_{F}^{\otimes,\vee} \rightarrow \bigcat^{\times}
    \end{equation}
    to be the coCartesian straightenings of the functors $\Phi_{F}$ and $\Phi_{F}^{\otimes}$, the latter of which is lax symmetric monoidal by Lemma~\ref{lem: D_F lax sm}.
\end{defin}

\begin{rem}
    Informally, the lax symmetric monoidal functor $D_{F}$ is given by $D_{F}(X) = D(X)$ and $(f,\alpha)^{\ast} = \eta(\alpha) \otimes f^{\ast}(-) : D(Y) \rightarrow D(X)$. 
    Note that $D_{F}$ really depends on both $F$ and $\eta$.
\end{rem}
Let $\eta : F \rightarrow \Pic(D)$ be a map of structure functors on $\m{C}$. 
As $\Pic\big(D(X)\big)$ acts on $D(X)$, we obtain an action of $F(X)$ on $D(X)$, where $\alpha \in F(X)$ acts on $D(X)$ via 
\begin{equation}
    \eta(\alpha) \otimes (-) : D(X) \rightarrow D(X).
\end{equation}

\begin{prop}\label{prop:twisted3ff}
    There is an essentially unique additive 3-functor formalism 
    \begin{equation}
    D_{F} : \Cor(\m{C}_{F};E_{F}) \rightarrow \bigcat
    \end{equation}
    such that: 
    \begin{enumerate}
        \item $D_{F}(X) = D(X)$;
        \item for any $(f,\alpha) : X \rightarrow Y\in\m{C}_{F}$, we have 
        \begin{equation}
        (f,\alpha)^{\ast} = f^{\ast}(-) \otimes \eta(\alpha) : D(Y) \rightarrow D(X);
        \end{equation}
        \item and for any $(g,\beta) : X \rightarrow Y\in E_{F}$, we have 
        \begin{equation}
        (g,\beta)_{!} = g_{!}\big(-\otimes \eta(-\beta)\big) : D(X) \rightarrow D(Y).
        \end{equation}
    \end{enumerate}
    Furthermore, there is a canonical isomorphism of 3-functor formalisms 
    \begin{equation}
    D_{F} \simeq \eta^{\ast}D_{\Pic},
    \end{equation}
    where $\eta^{\ast}D_{\Pic}$ denotes the pullback of $D_{\Pic}$ along 
    \begin{equation}
    \Cor(\m{C}_{F};E_{F}) \rightarrow \Cor(\m{C}_{\Pic};E_{\Pic}).
    \end{equation}
    Finally, if $D$ is a 6-functor formalism, then so is $D_{F}$. 
\end{prop}

\begin{proof}
    By \cite[Theorem 5.14]{cnossenUniversalitySpan2categories2026} we only need to verify that $(\m{C}_{F}^{\otimes},E_{F}^{\otimes})$ and $D_{F}^{\otimes}$ satisfy \cite[Definition 5.13, (1) and (2b)]{cnossenUniversalitySpan2categories2026}. We will repeatedly use the fact that whenever $(f,\alpha) \in E_{F}$, the functor $(f,\alpha)^{\ast}$ admits a right adjoint given by:
        \begin{equation}
        (f,\alpha)_{!} \equiv f_{!}\big(-\otimes\eta(-\alpha)\big).
        \end{equation}
    
    We begin by showing $\m{D}_{F}$ is $(\m{C}^{\simeq},E_{F})$-biadjointable. 
    Given a pullback square 
    \begin{equation}\begin{tikzcd}
	{\overline{X}}\arrow[dr, phantom, "\lrcorner", very near start] & {\overline{Y}} \\
	X & Y
	\arrow["{(\overline{f},\overline{g}^{\ast}\alpha)}", from=1-1, to=1-2]
	\arrow["{(\overline{g},\overline{f}^{\ast}\beta)}"', from=1-1, to=2-1]
	\arrow["{(g,\beta)}", from=1-2, to=2-2]
	\arrow["{(f,\alpha)}"', from=2-1, to=2-2]
    \end{tikzcd}\end{equation}
    in $\m{C}_{F}$ with $(g,\beta) \in E_{F}$, we must show that 
\begin{equation}\begin{tikzcd}
	{D(\overline{X})} && {D(\overline{Y})} \\
	{D(X)} && {D(Y)}
	\arrow["{(\overline{f},\overline{g}^{\ast}\alpha)^{\ast}}", tail reversed, no head, from=1-1, to=1-3]
	\arrow["{(\overline{g},\overline{f}^{\ast}\beta)^{\ast}}"', tail reversed, no head, from=1-1, to=2-1]
	\arrow["{(g,\beta)^{\ast}}", tail reversed, no head, from=1-3, to=2-3]
	\arrow["{(f,\alpha)^{\ast}}"', tail reversed, no head, from=2-1, to=2-3]
\end{tikzcd}\end{equation}
    is vertically right adjointable, i.e., the induced exchange transformation 
    \begin{equation}
    (f,\alpha)^{\ast}(g,\beta)_{!} \rightarrow (\overline{g},\overline{f}^{\ast}\beta)_{!}(\overline{f},\overline{g}^{\ast}\alpha)^{\ast}
    \end{equation}
    is an equivalence. 
    Unraveling the definitions, this amounts to the claim that, for each $\m{F} \in D(\overline{Y})$, the induced map 
    \begin{equation}
    f^{\ast}g_{!}\big(\m{F}\otimes \eta(-\beta)\big) \otimes \eta(\alpha) \rightarrow \overline{g}_{!}\big(\overline{f}^{\ast}\m{F} \otimes \overline{g}^{\ast}\eta(\alpha) \otimes \overline{f}^{\ast} \eta(-\beta)\big)
    \end{equation}
    is an equivalence; this follows from the projection formula and the equivalence $f^{\ast}g_{!} \simeq \overline{g}_{!}\overline{f}^{\ast}$.
    Next, we establish that $D_{F}^{\otimes}$ satisfies a multiplicative right adjointability condition. 
    To this end, let $X \in \m{C}_{F}$ and $(f,\alpha) : Y \rightarrow Y'\in E_{F}$. 
    We claim the following square is vertically right adjointable:
    \begin{equation}\begin{tikzcd}
	{D_{F}(X)\times D_{F}(Y')} & {D_{F}(X\times Y')} \\
	{D_{F}(X)\times D_{F}(Y)} & {D_{F}(X\times Y)};
	\arrow["\boxtimes", from=1-1, to=1-2]
	\arrow["{\mathrm{id}\times(f,\alpha)^{\ast}}"', from=1-1, to=2-1]
	\arrow["{(\mathrm{id}\times f,\mathrm{pr}_{2}^{\ast}\alpha)^{\ast}}", from=1-2, to=2-2]
	\arrow["\boxtimes"', from=2-1, to=2-2]
    \end{tikzcd}\end{equation}
    that is, the induced transformation 
    \begin{equation}
        (-)\boxtimes f_{!}\big(-\otimes \eta(-\alpha)\big) \rightarrow (\mathrm{id}\times f)_{!}\big((-\boxtimes-)\otimes \eta(-\pr_{2}^{\ast}\alpha)\big)
    \end{equation}
    is an equivalence, but this is merely a consequence of the definitions.
    This completes the construction of the desired 3-functor formalism. 
    It is easy to see that $D_{F}$ must be additive and that, if $D$ is a 6-functor formalism, then so is $D_{F}$.  
\end{proof}

\section{Tangentially structured correspondences}\label{sec:correspondences}
In the remainder of the present article we fix a pair $(B,\varphi)$, where $B$ is a grouplike $\bbE_\infty$-monoid and $\varphi:B\to\BO$ is a map of $\bbE_\infty$-monoids.
    
\begin{defin}\label{defin:tangentialstructure}
Let $f:X\to Y\in \DSm$.
\begin{enumerate}
\item Given $v \in\TKO(X)$, a \textit{$\varphi$-structure} on $v$ is choice of lift in $\Spaces$:
    \begin{equation}
    \begin{tikzcd}
	& B \arrow[d,"\varphi"] \\
    \Sing(X)\arrow[r,"v"]\arrow[ur, dashed] & \BO.
    \end{tikzcd}
    \end{equation}
\item A \emph{stable (tangential) $\varphi$-structure} on $f$ is defined to be a $\varphi$-structure on $\tau_f$.
\end{enumerate}
\end{defin}

Recall, the \textit{(stable) J-homomorphism}, denoted $J$, is the composition of maps of $\ee{\infty}$-monoids 
    \begin{equation}
    \bigsqcup_{n} \BO(n) \xrightarrow{S^{(-)}} \Spaces_{\ast} \xrightarrow{\Sigma^{\infty}} \Sp
    \end{equation}
which sends a real vector space $V$ to $\bbS^V\equiv\Sigma^\infty S^V$; note, $\bbS^V$ is an invertible spectrum. Moreover, since the target is group-complete, we obtain a symmetric monoidal functor 
    \begin{equation}
    J:\Omega^{\infty}\ko \rightarrow \Pic(\bbS). 
    \end{equation}
The restriction of $J$ along the inclusion $\BO=\{0\}\times\BO\to\Omega^{\infty}\ko$ factors through $\{0\}\times \BGL_{1}(\sphere)\to\Pic(\sphere)$; this yields another map of $\einf$-monoids,
    \begin{equation}
    \BO\to{\rm BGL}_1(\bbS),
    \end{equation}
which we still denote by $J$.

\begin{defin}
The \textit{Thom spectrum} associated to $\varphi$ is the colimit
    \begin{equation}
    \rmM\varphi\equiv \colim \big(B \rightarrow \BO \rightarrow {\rm BGL}_1(\bbS) \rightarrow \Sp\big);
    \end{equation}
this admits a canonical $\einf$-ring structure by \cite{antolin-camarenaSimpleUniversalProperty2019}.
\end{defin}

\begin{construction}\label{construction: Thom structure transform}
Let $\PH \subset \LCH$ denote the full subcategory of paracompact Hausdorff topological spaces. 
By the discussion in \cite[Section 7]{volpe}, there is a natural transformation of functors from $\PH^\op$ to $\CMongp$:
    \begin{equation}
    \Th:\KO\to \Pic\big(\Shv(-;\bbS)\big).
    \end{equation}
Meanwhile, there is a natural transformation of functors from $\PH^\op$ to the \category of symmetric monoidal \categories:
    \begin{equation}
    \Fun(\Sing(-),\Sp)\rightarrow \Shv(-;\sphere);
    \end{equation}
this converts local systems into their sheaves of sections, and the natural transformation $\Th$ can be factored as
    \begin{equation}
        \KO \rightarrow \Fun\big(\Sing(-),\Pic(\sphere)\big) \rightarrow \Pic\big(\Shv(-;\sphere)\big).\footnote{Note, Volpe crucially uses paracompactness since they appeal to \cite[7.1.4.3]{lurieHigherToposTheory2009}.}
    \end{equation}
By restricting along the inclusion $\TKO \rightarrow \KO$, we obtain a commutative diagram 
\begin{equation}
\begin{tikzcd}
	\TKO & {\Fun\big(\Sing(-),{\rm BGL}_1(\bbS)\big)} \\
	\KO & {\Fun\big(\Sing(-),\Pic(\sphere)\big)}
	\arrow[from=1-1, to=1-2]
	\arrow[from=1-1, to=2-1]
	\arrow[from=1-2, to=2-2]
	\arrow[from=2-1, to=2-2]
\end{tikzcd}
\end{equation}
and, by abuse of notation, we also write $\Th$ for 
\begin{equation}
\TKO \rightarrow \Fun\big(\Sing(-),\rm{BGL}_{1}(\sphere)\big) \rightarrow \Pic\big(\Shv(-;\sphere)\big).
\end{equation}
\end{construction}

\begin{defin}
Given $X\in\DSm$ and $\xi\in\KO(X)$, the \emph{Thom local system} of $\xi$, denoted $\Th_X(\xi)$, is the composition
    \begin{equation}
    \Sing(X) \xrightarrow{\xi} \oko \xrightarrow{J} \Pic(\sphere). 
    \end{equation}
By the previous discussion, this local system determines an invertible sheaf which we call the \emph{Thom sheaf} associated to $\xi$; by an abuse of notation, we denote this sheaf also by $\Th_X(\xi)$.
If $U \subset X$ is an open embedding, we write $\Th_{X}(\xi)|_{U}$ for the restriction of $\Th_{X}(\xi)$ along $\Sing(U) \rightarrow \Sing(X)$, or for the restriction of the Thom sheaf to $U$.
\end{defin}

\begin{defin}
Let $R$ be an $\bbE_\infty$-ring and $X\in\DSm$.   
\begin{enumerate}
\item Let $L:\Sing(X)\to{\rm BGL}_1(\bbS)$. An \emph{$R$-orientation} of $L$ is a choice of null-homotopy of the composite
    \begin{equation}
    \Sing(X)\xrightarrow{L}{\rm BGL}_1(\bbS)\to{\rm BGL}_1(R).
    \end{equation}
Equivalently, this is a choice of equivalence 
    \begin{equation}
    \theta:R\otimes_\bbS L\xrightarrow{\sim}R_X,
    \end{equation}
where $R\otimes_\bbS L$ denotes the local system obtained by base change along $\bbS\to R$ and $R_X$ denotes the $R$-valued constant local system on $\Sing(X)$.
\item Let $\calL\in\Pic\big(\Shv(X;\bbS)\big)$. An \emph{$R$-orientation} of $\calL$ is a choice of equivalence
    \begin{equation}
    \vartheta:R\otimes_\bbS\calL\xrightarrow{\sim}R_X,
    \end{equation}
where $R\otimes_\bbS\calL$ denotes the sheaf of $R$-modules obtained by extension of scalars along $\bbS\to R$ and $R_X$ denotes the constant sheaf of $R$-modules.
\end{enumerate}
\end{defin}

\begin{notation}
Given a morphism $d_X:\Sing(X)\to\bbZ$, we may view $d_{X}$ as a virtual bundle by post-composing with the $\ee{1}$-section $s: \bbZ \rightarrow \Omega^{\infty}\ko$, and this point $d_{X} \in \KO(X)$ determines the similarly named element of $\pi_{0}\KO(X)$. 
Given an $\einf$-ring $R$, we obtain a $\Pic(R)$-valued local system
\begin{equation}
    R_{X}\{d_{X}\} \equiv R\otimes_{\sphere}\Th_{X}(d_{X}) : \Sing(X) \rightarrow \Pic(R).
\end{equation}
If $C \subset X$ denotes a path-component of $X$, then the induced local system on $\Sing(C)$ is equivalent to the constant local system $\Sigma^{d_{X}(C)}R_{X}$, where $d_{X}(C)\in\bbZ$ denotes the unique integer attained by $d_{X}$ after restricting to $\Sing(C)$.
Additionally, if $\calL \in \Pic(\Shv(X;R))$, we write 
    \begin{equation}
    \calL\{d_X\}\equiv \calL\otimes_{R_{X}} R_X\{d_X\}. 
    \end{equation}
\end{notation}

\begin{rem}\label{rem:baneofourexistence}
We would like to emphasize that the equivalence
    \begin{equation}
    R_X\{d_X\}\otimes_RR_X\{d_{X'}\}\simeq R_X\{d_{X'}\}\otimes_RR_X\{d_X\}
    \end{equation}
may carry an implicit sign as it is a path in the space $\Pic\big(\Shv(X;R)\big)$.
\end{rem}

\begin{rem}[Universal $\rmM\varphi$-orientation]\label{remark: universal phi orientation}
Consider $\Th_B(\varphi)$ over $B$. By \cite[Corollary 3.17]{antolin-camarenaSimpleUniversalProperty2019}, there is a canonical $\rmM\varphi$-orientation
    \begin{equation}
    \mathrm{can}_{\varphi} : \Th_{B}(\varphi)\otimes_\bbS\rmM\varphi \xrightarrow{\sim} \rmM\varphi_{B}
    \end{equation}
arising from the identity map $\rmM\varphi \rightarrow \rmM\varphi$.
\end{rem}

The following result is immediate from the previous remark.

\begin{cor}[Thom isomorphism]
Let $X\in\DSm$ and $v\in\TKO(X)$. Suppose $v$ has a $\varphi$-structure $\alpha$, then $\mathrm{can}_{\varphi}$ induces a canonical $\rmM\varphi$-orientation of $\Th_X(v)$:
    \begin{equation}
     \Th_{X}(v)\otimes_\bbS\rmM\varphi \simeq \Th_{X}(\varphi\circ\alpha)\otimes_\bbS\rmM\varphi \xrightarrow{\sim}\rmM\varphi_{X}.
    \end{equation}
In particular, $\mathrm{can}_{\varphi}$ induces an equivalence
    \begin{equation}
    \Gamma\big(X;\Th_{X}(v)\otimes_\bbS\rmM\varphi\big) \xrightarrow{\sim} F\big(\Sigma^\infty_+\Sing(X), \rmM\varphi\big).
    \end{equation}
\end{cor}

\begin{rem}
Let $\xi \in \KO(X)$ and $\alpha$ a $\varphi$-structure on $\xi - \rk\xi$. 
In this case, the Thom isomorphism above takes the form
    \begin{equation}
    \rmM\varphi_X \simeq \Th_{X}(\varphi\circ\alpha) \otimes_{\sphere} \rmM\varphi \simeq  \Th_X(\xi)\otimes_{\sphere}\rmM\varphi_{X} \{-\rk\xi\}.
    \end{equation}
\end{rem}

\subsection{Tangentially structured correspondences}\label{subsection:tangentiallystructuredcor}
Recall, by $\TKO$, we mean the functor 
    \begin{equation}
    \TKO(-)\equiv\Map_{\Spaces}\big(\Sing(-),\BO\big):\DSm^\op\to\CMongp.
    \end{equation}
We denote by $\DSm_{\TKO}$ the $\TKO$-structuring of $\DSm$. Meanwhile, we have the functor 
    \begin{equation}
    \frakB(-)\equiv\Map_{\Spaces}\big(\Sing(-),B\big):\DSm^\op\to\CMongp;
    \end{equation}
this is a structure functor. We denote by $\DSm_\frakB$ the $\frakB$-structuring of $\DSm$. Of course, $\varphi$ induces a symmetric monoidal functor 
    \begin{equation}
    \DSm_\frakB\to\DSm_{\TKO}.
    \end{equation}

\begin{rem}
Recall, $\DSm$ determines a wide subcategory of $\DSm_{\TKO}$ spanned by the \emph{trivially structured} morphisms, i.e., morphisms of the form $(f,0):X\to Y$.
\end{rem}

The following result is immediate. 

\begin{lem}
Morphisms in $\DSm_{\TKO}$ of the form $(f,\tau_f):X\to Y$ are closed under composition. 
\end{lem}

In particular, we may define the wide subcategory $\DSm_\tau\subset\DSm_{\TKO}$ spanned by the \emph{tangentially structured} morphisms. The following two results are immediate consequences of our prior work on the stable relative tangent bundle.

\begin{lem}\label{lem: pb tang str triv str}
    The pullback of a tangentially structured morphism along a trivially structured morphism is tangentially structured. 
\end{lem}

\begin{lem}
    The wide subcategory $\DSm_{\tau} \subset \DSm_{\TKO}$ is a symmetric monoidal subcategory which is also closed under coproducts.  
\end{lem}

\begin{defin}
    We define the wide subcategory of \textit{proper, tangentially $\varphi$-structured morphisms} to be
    \begin{equation}
        \calP_{\varphi}\equiv \calP_{\mathfrak{B}}\times_{\calP_{\TKO}}\calP_{\tau} \subset \DSm_{\frakB}.
    \end{equation}
\end{defin}

\begin{prop}\label{prop: P_phi adequate triple}
    The triple $(\DSm_{\frakB};\DSm,\calP_{\varphi})$ determines a commutative monoid in $\mathrm{AdTrip}$.
\end{prop}

\begin{proof}
    The triple $(\DSm_{\frakB};\DSm,\calP_{\varphi})$ is easily seen to be adequate by combining Lemma~\ref{lem: pb tang str triv str} with the fact that $\DSm_{\frakB}$ is closed under pullbacks. 
    It is not difficult to see that $\DSm,\calP_{\varphi} \subset \DSm_{\frakB}$ are symmetric monoidal subcategories, and similar reasoning shows that ambigressive pullback squares are closed under the symmetric monoidal product in $\DSm_{\frakB}$.    
\end{proof}

\begin{defin}
Define 
\begin{equation}
\Cor_{\varphi}(\DSm) \equiv \Cor(\DSm_{\frakB};\DSm,\calP_{\varphi});
\end{equation}
this admits a canonical symmetric monoidal structure, $\Cor_{\varphi}(\DSm)^{\otimes}$, by Proposition~\ref{prop: P_phi adequate triple}.
\end{defin}

\begin{rem}
Informally, the \category $\Cor_{\varphi}(\DSm)$ has objects derived smooth manifolds and morphisms given by structured correspondences 
    \begin{equation}
    X \xleftarrow{(f,0)} W \xrightarrow{(g,\alpha)} Y,
    \end{equation}
where $\alpha$ is a choice of $\varphi$-structure on $\tau_g$.
The monoidal structure is entirely analogous to the one described in Remark~\ref{remark: tensor product on Cor(C_F;E_F)}. 
Furthermore, $\Cor_{\varphi}(\DSm)$ is a semiadditive \category, where the direct sum of two derived smooth manifolds is given by $X \sqcup Y$. 
The direct sum of two $\varphi$-structured correspondences is given by 
\begin{equation}
X_{0}\sqcup X_{1} \xleftarrow{(f_{0}\sqcup f_{1},0)} W_{0} \sqcup W_{1} \xrightarrow{(g_{0}\sqcup g_{1},\alpha)} Y_{0}\sqcup Y_{1},
\end{equation}
where $\alpha \in \frakB(W_{0}\sqcup W_{1})$ is identified with the pair $(\alpha_{0},\alpha_{1}) \in \frakB(W_{0})\times \frakB(W_{1})$; note, $\tau_{f_{0}\sqcup f_{1}}$ corresponds to the tuple $(\tau_{f_{0}},\tau_{f_{1}})$.
\end{rem}

\subsection{Structured correspondences and transfers}

\begin{construction}
    Using the $\einf$-map $\sphere \rightarrow \rmM\varphi$ together with Construction~\ref{construction: Thom structure transform}, we obtain a morphism of structure functors 
\begin{equation}
    \Th_{\varphi} \equiv (\Th \circ \varphi)\otimes_\sphere\rmM\varphi : \frakB \rightarrow \Pic\big(\Shv(-;\rmM\varphi)\big),
\end{equation}
where $\Shv(-;\rmM\varphi) : \Cor(\DSm;\calP) \rightarrow \bigcat$ denotes the 6-functor formalism from Theorem~\ref{theorem: Volpe 6FF}. 
\end{construction}

\begin{rem}
    The reader familiar with orientation theory may have noticed $\Th_{\varphi}$ admits a canonical trivialization. 
    We will return to this point in Section~\ref{sec: coherent PT}.
\end{rem}

\begin{defin}
    We let
    \begin{equation}
        \Shv_{\varphi} : \Cor(\DSm_{\frakB};\calP_{\frakB}) \rightarrow \bigcat
    \end{equation}
    denote the 6-functor formalism obtained by twisting $\Shv(-;\rmM\varphi)$ along $\Th_{\varphi}$. 
    We define 
    \begin{equation}
        \m{D}_{\varphi} : \Cor_{\varphi}(\DSm) \rightarrow \bigcat
    \end{equation}
    to be the restriction of $\Shv_{\varphi}$ along the symmetric monoidal inclusion $\Cor_{\varphi}(\DSm) \rightarrow \Cor(\DSm_{\frakB};\calP_{\frakB})$. 
\end{defin}

\begin{defin}
    Define 
    \begin{equation}
        \TrCor_{\varphi}(\DSm)  \subset \int^{\rm cc}_{\Cor_{\varphi}(\DSm)} \m{D}_{\varphi} \equiv \Un(\m{D}_{\varphi})
    \end{equation}
    to be the full subcategory spanned by objects of the form $(X,d_X) \equiv (X,\rmM\varphi_{X}\{d_X\})$, where $d_X : \Sing(X) \rightarrow \bbZ$.
\end{defin}

\begin{rem}
We would like to \emph{really} emphasize that the symbol $(X,d_X)$ is \emph{simply} a stand in for $(X,\rmM\varphi_{X}\{d_X\})$, i.e., the equivalence 
    \begin{equation}
    (X,d_X+d_{X}')\simeq(X,d_{X}'+d_X)
    \end{equation}
utilizes Remark \ref{rem:baneofourexistence}.
\end{rem}

\begin{rem}
Informally, a morphism $(X,d_X) \rightarrow (Y,d_Y)$ is the data of a $\varphi$-structured correspondence 
    \begin{equation}
    X\xleftarrow{(f,0)} W \xrightarrow{(g,\alpha)}Y
    \end{equation}
together with a morphism in $\Shv\big(Y;\rmM\varphi\big)$:
    \begin{equation}
    \theta :g_{!}\big(\Th_W(\tau_{g})\otimes_\bbS\rmM\varphi_W\{f^*d_X\}\big) \rightarrow\rmM\varphi_{Y}\{d_Y\}.
    \end{equation}
Therefore, a morphism in $\TrCor_{\varphi}(\DSm)$ can be viewed as a $\varphi$-structured correspondence together with a choice of \emph{transfer class} $\theta$ as above. Often, we will denote the data of a morphism $(X,d_X) \rightarrow (Y,d_Y)$ as
    \begin{equation}
    (X,d_X)\xleftarrow{(f,0)} W \xrightarrow{(g,\alpha,\theta)}(Y,d_Y).
    \end{equation}
\end{rem}

\begin{lem}\label{lem: TrCor symmetric monoidal}
    $\TrCor_{\varphi}(\DSm) \subset \Un(\m{D}_{\varphi})$ is a symmetric monoidal subcategory. 
\end{lem}

\begin{proof}
    Certainly $(\bbR^{0},0) \in \TrCor_{\varphi}(\DSm)$, so it is enough to show $(X,d_{X}) \boxtimes (Y,d_{Y})$ belongs to $\TrCor_{\varphi}(\DSm)$. 
    However, this follows from the equivalence 
    \begin{equation}
    \pr_{1}^{\ast}\rmM\varphi_{X}\{d_X\} \otimes_{\rmM\varphi_{X\times Y}} \pr_{1}^{\ast}\rmM\varphi_{Y}\{d_Y\} \simeq \rmM\varphi_{X\times Y}\{\pr_1^*d_{X}+\pr_2^*d_{Y}\}.
    \end{equation}
\end{proof}

\begin{lem}\label{lem: TrCor semiadditive}
    $\Un(\m{D}_{\varphi})$ is semiadditive; moreover, $\TrCor_{\varphi}(\DSm)$ is a semiadditive subcategory. 
\end{lem}

\begin{proof}
    By \cite[Lemma 10.7]{gepnerLaxColimitsFree2017}, $\Un(\m{D}_{\varphi})$ admits finite coproducts. 
    Given 
    \begin{equation}
    (Y_{0},\m{G}_{0}), (Y_{1},\m{G}_{1}) \in \Un(\m{D}_{\varphi}),
    \end{equation}
    let us describe their binary coproduct; to do so, first set 
        \begin{equation}
        \m{G} \equiv (\m{G}_{0},\m{G}_{1}) \in \m{D}_{\varphi}(Y_{0}) \times \m{D}_{\varphi}(Y_{1}) \simeq \m{D}_{\varphi}(Y_{0}\sqcup Y_{1}).
        \end{equation}
    For $k = 0,1$, the $\varphi$-structured correspondence 
    \begin{equation}
    Y_{0}\sqcup Y_{1} \xleftarrow{(i_{k},0)} Y_{k} \xrightarrow{(\mathrm{id},0)} Y_{k}
    \end{equation}
    induces a morphism 
    \begin{equation}
    (Y_{0}\sqcup Y_{1},\m{G}) \rightarrow (Y_{k},i_{k}^{\ast}\m{G}) \simeq (Y_{k},\m{G}_{k}).
    \end{equation}
    In turn, this induces a natural transformation
    \begin{equation}
    \Map_{\Un(\m{D}_{\varphi})}\big(-,(Y_{0}\sqcup Y_{1},\m{G})\big) \rightarrow \Map_{\Un(\m{D}_{\varphi})}\big(-,(Y_{0},\m{G}_{0})\big)\times \Map_{\Un(\m{D}_{\varphi})}\big(-,(Y_{1},\m{G}_{1})\big)
    \end{equation}
    which is easily seen to be an equivalence by \cite[2.4.4.2]{lurieHigherToposTheory2009} combined with the semiadditivity of $\Cor_{\varphi}(\DSm)$ and the additivity of $\m{D}_{\varphi}$.
    A similar argument shows that the initial object $(\varnothing,0)$ is also terminal.
    The closure of $\TrCor_{\varphi}(\DSm)$ under direct sums is clear. 
\end{proof}

\subsection{Extended transfer yoga}
Unfortunately, but as expected, we see that $\TrCor_{\varphi}(\DSm)$ has too many morphisms; for instance, if we consider 
    \begin{equation}
    (X,d_X)\xleftarrow{(f,0)} W \xrightarrow{(g,\alpha,\theta)}(Y,d_Y)\in\TrCor_\varphi(\DSm),
    \end{equation}
then $\theta$ is allowed, at the moment, to be quite arbitrary. In particular, we do not have a handle on when a map of derived smooth manifolds admits a transfer class. The following two subsections are dedicated to remedying this issue.

\begin{construction}
Let $f:X\to Y\in\DSm$ be an amplitude at most 1 map which is of standard presentation:
    \begin{equation}
    \begin{tikzcd}
    X\arrow[r,"e"]\arrow[d,"\pi_X",swap] & Y\times\bbR^N\arrow[r,"\pi"]\arrow[d,"h"] & Y \\
    * \arrow[r,"i",swap] & \bbR^k. &
    \end{tikzcd}
    \end{equation}
Observe, $i$ and $\pi$ are canonically $\varphi$-structured, and $f$ is of rank $-k+N$. Consider the morphism $\theta$ given by the adjoint of
    \begin{equation}
    \Th_X(\tau_f)\otimes_\bbS\rmM\varphi_X\{\rk f\}\simeq\rmM\varphi_X\{\rk f\}\to f^!\rmM\varphi_Y,
    \end{equation}
where the previous morphism is given by the interchange map:
    \begin{multline}
    \Th_X(\tau_f)\otimes_\bbS\rmM\varphi_X\{-k+N\}\simeq\rmM\varphi_X\{-k+N\}\simeq\pi_X^*\rmM\varphi_*\{-k+N\}\simeq \\
    \pi_X^*i^!\rmM\varphi_{\bbR^k}\{N\}\to e^!h^*\rmM\varphi_{\bbR^k}\{N\}\simeq e^!\rmM\varphi_{Y\times\bbR^N}\{N\}\simeq e^!\pi^!\rmM\varphi_Y\simeq f^!\rmM\varphi_Y;
    \end{multline}
note, the equivalence $\rmM\varphi_{Y\times\bbR^N}\{N\}\simeq\pi^!\rmM\varphi_Y$ utilizes \cite[Proposition 6.21]{volpe}. We refer to $\theta$, constructed as above, as the \emph{standard transfer} associated to the standard presentation of $f$.\footnote{Often, we will denote both a standard transfer and its adjoint by $\theta$; which we consider will be clear from context.}

\end{construction}

\begin{defin}
Let $f:X\to Y\in\DSm$ be of amplitude at most 1. We say 
    \begin{equation}
    \theta:f_!\big(\Th_X(\tau_f)\otimes_\bbS\rmM\varphi_X\{\rk f\}\big)\to \rmM\varphi_Y
    \end{equation}
is an \emph{extended transfer} if, for any $x\in X$, there exists an open neighborhood $U\ni x$ such that $f\vert_U$ is of standard presentation and $\theta\vert_U$ is equivalent to the transfer induced by this standard presentation.
\end{defin}

The following lemmata amount to the ``calculus'' of extended transfers; the proofs utilize the calculus of standard presentations from Subsection \ref{subsec:amplitude}.

\begin{lem}\label{lem:extendedtransfercomp}
Let $f:X\to Y,g:Y\to Z\in\DSm$ be of amplitude at most 1 together with extended transfers given by the adjoints of
    \begin{align}
    \theta:\Th_X(\tau_f)\otimes_\bbS\rmM\varphi_X\{\rk f\}&\to f^!\rmM\varphi_Y, \\
    \theta':\Th_Y(\tau_g)\otimes_\bbS\rmM\varphi_Y\{\rk g\}&\to g^!\rmM\varphi_Z,
    \end{align}
respectively. We have that the adjoint of the composite
    \begin{multline}
    \Th_X(\tau_{g\circ f})\otimes_\bbS\rmM\varphi_X\{\rk(g\circ f)\}\simeq\big(\Th_X(\tau_f)\otimes_\bbS\rmM\varphi_X\{\rk f\}\big)\otimes_{\rmM\varphi} \\
    \big(f^*\Th_Y(\tau_g)\otimes_\bbS\rmM\varphi_X\{f^*\rk g\}\big)\xrightarrow{\theta\otimes f^*\theta'}
    f^!\rmM\varphi_Y\otimes_{\rmM\varphi} f^*g^!\rmM\varphi_Z\to
    (g\circ f)^!\rmM\varphi_Z,
    \end{multline}
where the second morphism is the projection morphism, is an extended transfer.
\end{lem}

\begin{proof}
As the claim is local, it is enough to prove it when $f$ and $g$ are (globally) of standard presentation. We consider a standard presentation of both $f$ and $g$: 
    \begin{equation}
    \begin{tikzcd}
    X\arrow[r]\arrow[d]\arrow[dr, phantom, "\lrcorner", very near start] & Y\times\bbR^N\arrow[r]\arrow[d] & Y \\
    * \arrow[r] & \bbR^k &
    \end{tikzcd}
    ,\;\;
    \begin{tikzcd}
    Y\arrow[r]\arrow[d]\arrow[dr, phantom, "\lrcorner", very near start] & Z\times\bbR^M\arrow[r]\arrow[d] & Z \\
    * \arrow[r] & \bbR^\ell. &
    \end{tikzcd}
    \end{equation}
Moreover, we recall the diagram producing the induced standard presentation of $g\circ f$ from the proof of Lemma \ref{lem:standardpresentation1}:
    \begin{equation}\label{eqn:pasting}
    \begin{tikzcd}
    X\arrow[d]\arrow[r]\arrow[dr, phantom, "\lrcorner", very near start] & Y\times\bbR^N\arrow[d,"h"]\arrow[r]\arrow[dr, phantom, "\lrcorner", very near start] & Z\times\bbR^{N+M}\arrow[d,"{(h,h')}"]\arrow[r] & Z \\
    *\arrow[r] & \bbR^k\arrow[r] & \bbR^{k+\ell}. &
    \end{tikzcd}
    \end{equation}
Some remarks are in order regarding the previous diagram.
\begin{enumerate}
\item Recall, by using a smooth bump function, we may extend $h$ to a function 
    \begin{equation}
    h:Z\times\bbR^{N+M}\to\bbR^k
    \end{equation}
satisfying $h^{-1}(0)=X$. 
\item Recall, we may extend $h'$ to a function 
    \begin{equation}
    h':Z\times\bbR^{N+M}\to\bbR^\ell
    \end{equation}
satisfying $(h')^{-1}(0)=Y\times\bbR^N$.
\item Every square is a pullback.
\item The outer square is the induced standard presentation of $g\circ f$.
\item The left square is the standard presentation of $f$.
\item Every horizontal morphism, except the last, is an immersion.
\end{enumerate}

Observe, in the previous remarks, we did not say how the right square is related to the standard presentation of $g$; we must remedy this. We will consider the following ``moves'' of pullbacks. First, we may perform the move
    \begin{equation}
    \begin{tikzcd}
    Y\arrow[r]\arrow[d]\arrow[dr, phantom, "\lrcorner", very near start] & Y\times\bbR^M\arrow[d,"{h'}"] \\
    *\arrow[r] & \bbR^{\ell}
    \end{tikzcd}
    \leftrightsquigarrow
    \begin{tikzcd}
    Y\times\bbR^N\arrow[r]\arrow[d]\arrow[dr, phantom, "\lrcorner", very near start] & Y\times\bbR^{N+M}\arrow[d]\\
    Y\arrow[r]\arrow[d]\arrow[dr, phantom, "\lrcorner", very near start] & Y\times\bbR^M\arrow[d,"{h'}"] \\
    *\arrow[r] & \bbR^{\ell}
    \end{tikzcd}
    \end{equation}
by appending the top pullback in the right diagram to the left pullback. Now, the first key thing to observe is that we may immediately relate the interchange map of the left pullback to the pasting of the interchange maps of the right pullback since the interchange map of the top pullback in the right diagram contains essentially no information as it is canonically an equivalence; the relation is given by suspension. Second, we may perform the move
    \begin{equation}
    \begin{tikzcd}
    Y\times\bbR^N\arrow[r]\arrow[d]\arrow[dr, phantom, "\lrcorner", very near start] & Y\times\bbR^{N+M}\arrow[d]\\
    Y\arrow[r]\arrow[d]\arrow[dr, phantom, "\lrcorner", very near start] & Y\times\bbR^M\arrow[d,"{h'}"] \\
    *\arrow[r] & \bbR^{\ell}
    \end{tikzcd}
    \leftrightsquigarrow
    \begin{tikzcd}
    Y\times\bbR^N\arrow[r]\arrow[d]\arrow[dr, phantom, "\lrcorner", very near start] & Y\times\bbR^{N+M}\arrow[d,"{(h,h')}",swap]\arrow[dd,bend left,"{h'}"]\\
    \bbR^k\arrow[r]\arrow[d]\arrow[dr, phantom, "\lrcorner", very near start] & \bbR^{k+\ell}\arrow[d] \\
    *\arrow[r] & \bbR^{\ell}
    \end{tikzcd}
    \end{equation}
since the outer squares of both diagrams are the same pullback. Now, the second key thing to observe is that we may immediately relate the pasting of the interchange maps of the left diagram to the interchange map of the top pullback in the right diagram since the interchange map of the bottom pullback in the right diagram contains essentially no information as it is canonically an equivalence; the relation is given by suspension. Finally, we may perform the move
    \begin{equation}
    \begin{tikzcd}
    Y\times\bbR^N\arrow[r]\arrow[d]\arrow[dr, phantom, "\lrcorner", very near start] & Y\times\bbR^{N+M}\arrow[d,"{(h,h')}",swap]\arrow[dd,bend left,"{h'}"]\\
    \bbR^k\arrow[r]\arrow[d]\arrow[dr, phantom, "\lrcorner", very near start] & \bbR^{k+\ell}\arrow[d] \\
    *\arrow[r] & \bbR^{\ell}
    \end{tikzcd}
    \leftrightsquigarrow
    \begin{tikzcd}
    Y\times\bbR^N\arrow[r]\arrow[d]\arrow[dr, phantom, "\lrcorner", very near start] & Y\times\bbR^{N+M}\arrow[d,"{(h,h')}"]\\
    \bbR^k\arrow[r] & \bbR^{k+\ell}
    \end{tikzcd}
    \end{equation}
by appending the bottom pullback in the left diagram to the right pullback. Now, the final key thing to observe is that the interchange map of the top pullback in the left diagram is precisely the interchange map of the right pullback. In particular, we have now shown how to relate the interchange map of the right pullback of \eqref{eqn:pasting} to the interchange map used to define the standard transfer associated to the standard presentation of $g$. Meanwhile, the interchange map of the left pullback of \eqref{eqn:pasting} is already the interchange map used to define the standard transfer associated to the standard presentation of $f$. 

Since the hard work of relating the interchange maps is now done, the lemma follows by unwinding the definition of $\theta$, $\theta'$, and their composite.
\end{proof}

The proofs of the remaining lemmata are straightforward.

\begin{lem}\label{lem:extendtransferpb}
Consider the following pullback in $\DSm$:
    \begin{equation}
    \begin{tikzcd}
    \overline{X}\arrow[r,"\overline{f}"]\arrow[d]\arrow[dr, phantom, "\lrcorner", very near start] & \overline{Y}\arrow[d] \\
    X\arrow[r,"f",swap] & Y,
    \end{tikzcd}
    \end{equation}
where $f$ is amplitude at most 1, together with an extended transfer given by
    \begin{equation}
    f_!\big(\Th_X(\tau_f)\otimes_\bbS\rmM\varphi_X\{\rk f\}\big)\to \rmM\varphi_Y.
    \end{equation}
We have that the adjoint of the morphism induced by the interchange map,
    \begin{equation}
    \Th_{\overline{X}}(\tau_{\overline{f}})\otimes_\bbS\rmM\varphi_{\overline{X}}\big\{\rk\overline{f}\big\}\to\overline{f}^!\rmM\varphi_{\overline{Y}},
    \end{equation}
is an extended transfer.
\end{lem}

\begin{lem}\label{lem:extendedtransfersm}
Let $f:X\to Y,f':X'\to Y'\in\DSm$ be of amplitude at most 1 together with extended transfers
    \begin{align}
    f_!\big(\Th_X(\tau_f)\otimes_\bbS\rmM\varphi_X\{\rk f\}\big)&\to\rmM\varphi_Y, \\
    f'_!\big(\Th_{X'}(\tau_{f'})\otimes_\bbS\rmM\varphi_{X'}\{\rk f'\}\big)&\to\rmM\varphi_{Y'},
    \end{align}
respectively. 
We have that the morphism induced by the external product,
    \begin{equation}
    (f\times f')_!\big(\Th_{X\times X'}(\tau_{f\times f'})\otimes_\bbS\rmM\varphi_{X\times X'}\big\{\rk(f\times f')\big\}\big)\to\rmM\varphi_{Y\times Y'},
    \end{equation}
is an extended transfer.
\end{lem}

\begin{lem}\label{lem:extendedtransfersa}
Let $f:X\to Y,f':X'\to Y'\in\DSm$ be of amplitude at most 1 together with extended transfers
    \begin{align}
    f_!\big(\Th_X(\tau_f)\otimes_\bbS\rmM\varphi_X\{\rk f\}\big)&\to\rmM\varphi_Y, \\
    f'_!\big(\Th_{X'}(\tau_{f'})\otimes_\bbS\rmM\varphi_{X'}\{\rk f'\}\big)&\to\rmM\varphi_{Y'},
    \end{align}
respectively. We have the the morphism induced by the direct sum, 
    \begin{equation}
    (f\sqcup f')_!\big(\Th_{X\sqcup X'}(\tau_{f\sqcup f'})\otimes_\bbS\rmM\varphi_{X\sqcup X'}\big\{\rk(f\sqcup f')\big\}\big)\to\rmM\varphi_{Y\sqcup Y'},
    \end{equation}
is an extended transfer.
\end{lem}

Of course, when considering maps of smooth manifolds, the situation is significantly simpler.

\begin{construction}\label{construction:classicaltransfers}
Let $f:X\to Y\in\Sm$. Since $f$ is a map of smooth manifolds, we may factor it as an embedding followed by an obvious projection: 
    \begin{equation}
    X\xhookrightarrow{e} Y\times\bbR^N\xrightarrow{\pi} Y.
    \end{equation}
Utilizing \cite[Proposition 6.21 \& Corollary 7.12]{volpe}, we see 
    \begin{equation}
    f^!\rmM\varphi_Y\simeq e^!\pi^!\rmM\varphi_Y\simeq e^!\rmM\varphi_{Y\times\bbR^N}[N]\simeq\Th_X(\tau_f)\otimes_\bbS\rmM\varphi_X\{\rk f\}.\footnote{This computation \emph{crucially} relies on choosing a tubular neighborhood of $e$; note, derived smooth manifolds \emph{do not}, in general, admit tubular neighborhoods in any reasonable sense.}
    \end{equation}
Now, we may take $\theta$ to be the adjoint of the previous expression:
    \begin{equation}
    \theta:f_!\big(\Th_X(\tau_f)\otimes_\bbS\rmM\varphi_X\{\rk f\}\big)\to\rmM\varphi_Y.
    \end{equation}
It is straightforward to see $\theta$ is independent of the factorization of $f$ into $e$ and $\pi$ as any two smooth embeddings will become isotopic after stabilization. We refer to $\theta$, constructed as above, as the \emph{classical transfer} associated to $f$. Moreover, it is immediate that $\theta$ is indeed an extended transfer.
\end{construction}

In fact, the only extended transfers associated to maps of smooth manifolds are the classical transfers.

\begin{lem}\label{lem:extendedtoclassical}
Let $f:X\to Y\in\Sm$. Suppose 
    \begin{equation}
    \theta:f_!\big(\Th_X(\tau_f)\otimes_\bbS\rmM\varphi_X\{\rk f\}\big)\to\rmM\varphi_Y
    \end{equation}
is an extended transfer, then in fact $\theta$ is the classical transfer associated to $f$.
\end{lem}

\begin{proof}
Given any open subset $U\subset X$, $\theta\vert_U$ agrees with the classical transfer associated to $f$ essentially by definition, whence the claim.
\end{proof}

\subsection{Extended $\varphi$-structured correspondences}\label{subsec:extendedtransfers}
\begin{defin}\label{defin:extended}
Let
    \begin{equation}
    (X,d_X)\xleftarrow{(f,0)} W \xrightarrow{(g,\alpha,\theta)}(Y,d_Y)\in\TrCor_\varphi(\DSm),
    \end{equation}
where $g:W\to Y\in\DSm$ is amplitude at most 1 and $Y\in\Sm$ (in particular, $W\in\DSm$ is necessarily amplitude at most 1). We say the aforementioned correspondence is an \emph{extended (tangentially) $\varphi$-structured correspondence} if:
\begin{itemize}
\item $\rk g = f^{\ast}d_{X} - g^{\ast}d_{Y}$; and
\item $\theta$, which is \emph{a priori} a map
    \begin{equation}
    g_{!}\big(\Th_W(\tau_{g})\otimes_\bbS\rmM\varphi_W\{f^*d_X\}\big) \rightarrow\rmM\varphi_{Y}\{d_Y\},
    \end{equation}
is an extended transfer after twisting by $-d_Y$.
\end{itemize}
\end{defin}

\begin{notation}
In particular, whenever we consider 
    \begin{equation}
    (X,d_X)\xleftarrow{(f,0)} W \xrightarrow{(g,\alpha,\theta)}(Y,d_Y)\in\TrCor^{\rm ex}_\varphi(\Sm),
    \end{equation}
such that $g:W\to Y\in\Sm$, $\theta$ will always denote the classical transfer associated to $g$ by virtue of Lemma \ref{lem:extendedtoclassical}.
\end{notation}

By Lemmas \ref{lem:extendedtransfercomp} and \ref{lem:extendtransferpb}, extended $\varphi$-structured correspondences are closed under composition in $\TrCor_{\varphi}(\DSm)$. 

\begin{defin}
We denote by 
    \begin{equation}
    \TrCor_\varphi^{\rm ex}(\Sm)\subset\TrCor_\varphi(\DSm)
    \end{equation}
the subcategory whose objects consist only of objects in $\Sm$ and whose morphisms consist only of extended $\varphi$-structured correspondences.
\end{defin}

\begin{lem}\label{lem: TrCor^ex sym mon and sa}
$\TrCor_\varphi^{\rm ex}(\Sm)$ is a symmetric monoidal and semiadditive subcategory of $\TrCor_{\varphi}(\DSm)$. 
\end{lem}

\begin{proof}
Follows from Lemmas \ref{lem:extendedtransfersm} and \ref{lem:extendedtransfersa}.
\end{proof}

\begin{defin}
    Let
    \begin{equation}
    \Cor_{\varphi}(\Sm) \subset \Cor_{\varphi}(\DSm)
    \end{equation}
be the subcategory whose objects consist only of objects in $\Sm$ and whose morphisms consist only of $\varphi$-structured correspondences 
\begin{equation}
    X \xleftarrow{(f,0)} W \xrightarrow{(g,\alpha)} Y,
\end{equation}
where $f$ and $g$ are of amplitude at most $1$. 
\end{defin}

By construction, there is a forgetful functor 
\begin{equation}
\TrCor_{\varphi}^{\ex}(\Sm) \rightarrow \Cor_{\varphi}(\Sm)
\end{equation}
which forgets the extended transfer class. And, fortunately, the mapping spaces of $\TrCor_{\varphi}^{\ex}(\Sm)$ are closely related to those of $\Cor_{\varphi}(\Sm)$. 

\begin{lem}\label{lem: mapping spaces}
    Let 
    \begin{equation}
    (f,g,\alpha) \equiv X \xleftarrow{(f,0)} W \xrightarrow{(g,\alpha)} Y\in\Cor_{\varphi}(\Sm).
    \end{equation}
    The fiber of 
    \begin{equation}
    \Map_{\TrCor_{\varphi}(\Sm)}\big((X,d_{X}),(Y,d_{Y})\big) \rightarrow \Map_{\Cor_{\varphi}(\Sm)}(X,Y)
    \end{equation} 
    over $(f,g,\alpha)$ can be identified with
    \begin{equation}
    \Map_{\m{D}_{\varphi}(W)}\big(\Th_{W}(\tau_g)\otimes_\sphere\rmM\varphi_W\{\rk g\},g^!\rmM\varphi_Y\big).
    \end{equation}
    Furthermore, the fiber of 
    \begin{equation}
    \Map_{\TrCor^{\ex}_{\varphi}(\Sm)}\big((X,d_{X}),(Y,d_{Y})\big) \rightarrow \Map_{\Cor_{\varphi}(\Sm)}(X,Y)
    \end{equation} 
    over $(f,g,\alpha)$ can be identified with the subspace of 
    \begin{equation}
    \Map_{\m{D}_{\varphi}(W)}\big(\Th_{W}(\tau_g)\otimes_\sphere\rmM\varphi_W\{f^*d_X-g^*d_Y\},g^!\rmM\varphi_Y\big).
    \end{equation}
    spanned by the extended transfer classes.
\end{lem}

\begin{proof}
    Note, $\TrCor_{\varphi}(\Sm)$ is a full subcategory of $\Cor_{\varphi}(\Sm)\times_{\Cor_{\varphi}(\DSm)} \Un(\m{D}_{\varphi})$, i.e., they have equivalent mapping spaces. 
    By \cite[Proposition 2.4.4.2]{lurieHigherToposTheory2009}, the description of the claimed fiber sequence holds for $\Cor_{\varphi}(\Sm)\times_{\Cor_{\varphi}(\DSm)} \Un(\m{D}_{\varphi})$, as it is a coCartesian fibration over $\Cor_{\varphi}(\Sm)$, and therefore it also holds for $\TrCor_{\varphi}(\Sm)$.
    Because $\TrCor_{\varphi}^{\ex}(\Sm)$ is a wide subcategory of $\TrCor_{\varphi}(\Sm)$, its mapping spaces can be identified with subspaces of the mapping spaces of $\TrCor_{\varphi}(\Sm)$; these subspaces are exactly those spanned by the extended transfer classes. 
\end{proof}

\begin{prop}
$\Cor_{\varphi}(\Sm)$ is a symmetric monoidal and semiadditive subcategory. Moreover, the forgetful functor
\begin{equation}
    \TrCor_{\varphi}^{\ex}(\Sm) \rightarrow \Cor_{\varphi}(\Sm)
\end{equation}
is symmetric monoidal, semiadditive, and admits coCartesian lifts over 
\begin{equation}
    X \xleftarrow{(f,0)} W \xrightarrow{(g,0)} Y,
\end{equation}
where $g:W \rightarrow Y$ is an equivalence in $\DSm$.
\end{prop}

\begin{proof}
    It is clear that $\Cor_{\varphi}(\Sm)$ is closed under the monoidal product and direct sums, and that the forgetful functor preserves these.
    For the claim regarding coCartesian lifts, let $(X,d_{X}) \in \TrCor_{\varphi}^{\ex}(\Sm)$ and consider the morphism 
    \begin{equation}
    X \xleftarrow{(f,0)} W \xrightarrow{(g,0,\mathrm{id})} Y
    \end{equation} 
    from $(X,d_{X})$ to $(Y,d_{Y})$, where $d_{Y}\equiv f^*d_X : \Sing(Y)\simeq \Sing(W) \to \bbZ$. 
    Using the description of the mapping spaces from Lemma~\ref{lem: mapping spaces} combined with \cite[Proposition 2.4.4.3]{lurieHigherToposTheory2009}, it follows the morphism constructed above is coCartesian, whence the claim. 
\end{proof}

\section{Tangentially structured Gysin sheaves}\label{sec:gysin}
\subsection{Gysin (pre)sheaves}

\begin{defin}
Let $\m{E}$ be an $\infty$-category. A \textit{tangentially $\varphi$-structured Gysin presheaf} (or simply, \emph{$\varphi$-Gysin presheaf}), valued in $\m{E}$, is a functor 
    \begin{equation}
    \m{F} : \TrCor_{\varphi}^{\ex}(\Sm) \rightarrow \m{E}.
    \end{equation}
    The \category of $\m{E}$-valued $\varphi$-Gysin presheaves is denoted by 
    \begin{equation}
    \mathrm{GysP}_{\varphi}(\Sm;\m{E}) \equiv \Fun\big(\TrCor_{\varphi}^{\ex}(\Sm),\m{E}\big). 
    \end{equation}
\end{defin}

\begin{rem}
When $\m{E} = \Spaces$, we simply write $\GysP_{\varphi}(\Sm)\equiv\mathrm{GysP}_{\varphi}(\Sm;\Spaces)$ and refer to these objects as spatial $\varphi$-Gysin presheaves. Additionally, we refer to objects in $\GysP_{\varphi}(\Sm;\Sp)$ as spectral $\varphi$-Gysin presheaves.
\end{rem}

\begin{example}
    The most important spatial $\varphi$-Gysin presheaf is the one represented by the unit:
    \begin{equation}
    \Tr_{\varphi} \equiv \Map_{\TrCor_{\varphi}^{\ex}(\Sm)}\big((\bbR^{0},0),-\big) : \TrCor_{\varphi}^{\ex}(\Sm) \rightarrow \Spaces.
    \end{equation}
    Informally, $\Tr_{\varphi}(X,d_X)$ is the space of proper $\varphi$-structured maps 
    \begin{equation}
    (f,\alpha): W \rightarrow X
    \end{equation}
    and extended transfers
    \begin{equation}
    \theta:\Th_{W}(\tau_{f})\otimes_{\sphere} \rmM\varphi_W\{-f^*d_X\}\to f^{!}\rmM\varphi_{X}. 
    \end{equation}   
    By the semiadditivity of $\TrCor_{\varphi}^{\ex}(\Sm)$, the spatial $\varphi$-Gysin presheaf $\Tr_{\varphi}$ is canonically $\CMon(\Spaces)$-valued with basepoint given by the proper $\varphi$-structured map $\varnothing \rightarrow X$; we will return to this observation in the sequel. 
\end{example}

\begin{rem}
    If $\m{E} \in \PrL$ (resp. $\PrSt$), then  
    $\GysP_{\varphi}(\Sm;\m{E}) \in \PrL$ (resp. $\PrSt$). 
    This follows, for instance, from the equivalence 
    \begin{equation}
    \GysP_{\varphi}(\Sm) \otimes\m{E} \xrightarrow{\sim} \GysP_{\varphi}(\Sm;\m{E}),
    \end{equation}   
    where $\otimes$ denotes the Lurie tensor product in $\PrL$ (resp. $\PrSt$).
\end{rem}

\begin{lem}
    Let $\m{E}$ be a presentably symmetric monoidal \category. 
    The Day convolution product exhibits $\GysP_{\varphi}(\Sm;\m{E})$ as presentably symmetric monoidal with unit given by the composition 
    \begin{equation}
    \unit_{\m{E}}[\Tr_{\varphi}]: \TrCor_{\varphi}^{\ex}(\Sm) \xrightarrow{\Tr_{\varphi}} \Spaces \xrightarrow{\unit_{\m{E}}[-]} \m{E};
    \end{equation}
    here, $\unit_{\m{E}}[-]$ denotes the essentially unique symmetric monoidal functor given by sending $\ast \in \Spaces$ to $\unit_{\m{E}}\in\m{E}$. 
\end{lem}

\begin{proof}
    This follows from formal properties of the Day convolution product; cf. for example \cite[Corollary 3.29]{linskensGlobalHomotopyTheory2025} or \cite[Proposition 3.12]{keenan-peroux}.
\end{proof}

To consider $\varphi$-Gysin sheaves, we will endow $\TrCor_{\varphi}^{\ex}(\Sm)^{\op}$ with a natural Grothendieck topology. 
To accomplish this, we will make heavy use of the results in Appendix~\ref{appendix: sites}. 

\begin{construction}
Let $F : \DSm^{\op} \rightarrow \CMongp$ be a structure functor which is also a sheaf on $\DSm$; this yields a right fibration $\pi_{F} : \DSm_{F} \rightarrow \DSm$ and a canonical section $\sigma_{F} : \DSm \rightarrow \DSm_{F}$. 
By the results of \cite[Section 2.8]{pardon-DSm}, the composite 
\begin{equation}
    \DSm_{F} \xrightarrow{\pi_{F}} \DSm \xrightarrow{\rlz-}\Top
\end{equation}
exhibits $\DSm_{F}$ as a topological site such that both $\pi_{F}$ and $\sigma_{F}$ are strict topological functors.
By the general yoga of topological sites, we obtain Grothendieck sites such that $\pi_{F}$ and $\sigma_{F}$ are continuous functors. 
\end{construction}

\begin{defin}\label{defin: Cech vs sieve descent}
    Let $S$ be a covering sieve of $X \in \DSm$ generated by an open covering $\calU$; we denote by $C(\calU) \subset S$ the full subcategory spanned by open embeddings $U \rightarrow X$, where $U$ is a nonempty finite intersection of opens in $\calU$. 
\end{defin}

Observe, by \cite[Lemma 2.2.20]{pardon-DSm}, the inclusion $C(\calU) \rightarrow S$ is cofinal. We now summarize some of the basic features of the Grothendieck topology on $\DSm_{F}$. 

\begin{lem}\label{lem: DSm_F site properties}
    The Grothendieck site $\DSm_{F}$ has the following properties.
    \begin{enumerate}
        \item A sieve $S$ on $X\in\DSm_F$ is covering if and only if it is generated by an open cover $\mathcal{U} = \{U_{i} \rightarrow X\in\DSm_F\}_{i \in I}$, where, for each $i \in I$, the morphism $U_{i} \rightarrow X$ belongs to the essential image of $\sigma_{F}$. 
        \item $\DSm_{F}$ is fine in the sense of Definition~\ref{defin: fine site}.
        \item A presheaf $\m{F} \in \PShv(\DSm_{F})$ is a sheaf if and only if $\sigma^{\ast}_{F}\m{F}$ is a sheaf.
        \item It is subcanonical. 
    \end{enumerate}
\end{lem}

\begin{proof}
    By definition, a sieve $S \subset (\DSm_{F})_{/X}$ is covering if and only if it is generated by a family $\big\{(f_{i},\alpha_{i}): U_{i} \rightarrow X\big\}_{i\in I}$, where each $(f_{i},\alpha_{i})$ is an open embedding. 
    However, note that $(f_{i},\alpha_{i}) :U_{i} \rightarrow X$ can be factored as 
    \begin{equation}
    U_{i}\xrightarrow{(\id,\alpha_{i})}U_{i} \xrightarrow{(f_{i},0)} X
    \end{equation}
    which proves (1). 

    The fineness of $\DSm_{F}$ can be deduced from our assumptions about the point-set topology of derived smooth manifolds, specifically, we have assumed they are all paracompact Hausdorff; this proves (2).  

    For (3), one direction is clear, so assume $\sigma_{F}^{\ast}\m{F}$ is a sheaf, and let $S$ be a covering sieve of $X \in \DSm_{F}$. 
    By (1), $S$ is generated by an open cover $\mathcal{U}$ in the essential image of $\sigma_{F}$. 
    It is enough to show the map 
    \begin{equation}
        \m{F}(X) \rightarrow \lim_{U \in C(\calU)^{\op}} \m{F}(U)
    \end{equation}
    is an equivalence. 
    However, note that the composition 
    \begin{equation}
    C(\calU) \rightarrow S \rightarrow \DSm_{F}
    \end{equation}
    factors through $\sigma_{F}$, whence the claim.

    To establish (4), first, recall the equivalence 
    \begin{equation}
    \Map_{\DSm_{F}}(-,X) \simeq \Map_{\DSm}(-,X) \times F(-)
    \end{equation}
    from Lemma~\ref{lem: mapping spaces in C_F}. 
    Because $F$ is a sheaf on $\DSm$ and $\DSm$ is subcanonical, the claim follows from the fact that $C(\calU)$ is cosifted.
\end{proof}

By Lemma~\ref{lem: C_F admits pullbacks and coproducts}, $\DSm_{F}$ admits pullbacks. 
Subsequently, for each $(f,\alpha) : X \rightarrow Y\in\DSm_F$, the functor given by post-composition,
\begin{equation}
    (\DSm_{F})_{/X} \rightarrow (\DSm_{F})_{/Y},
\end{equation}
admits a right adjoint $(-)\times_{Y} X$; this implies the coCartesian fibration which evaluates at the source,
\begin{equation}
    \ev_{s}: \DSm_{F}^{[1]} \rightarrow \DSm_{F},
\end{equation}
is also a Cartesian fibration which is classified by a functor 
\begin{equation}
    (\DSm_{F})_{/(-)} : \DSm_{F}^{\op} \rightarrow \Cat_{\infty}
\end{equation}
which sends $(f,\alpha) : X \rightarrow Y$ to $(-)\times_{Y} X: (\DSm_{F})_{/Y} \rightarrow (\DSm_{F})_{/X}$. 
By abuse of notation, we will write 
\begin{equation}
    (\DSm_{F})_{/(-)} : \DSm^{\op} \rightarrow \Cat_{\infty}
\end{equation}
for the restriction to $\DSm^{\op}$.
The next proposition is crucial. 

\begin{prop}\label{prop: DSm_F descent}
    The functor 
    \begin{equation}
    (\DSm_{F})_{/(-)} : \DSm^{\op} \rightarrow \Cat_{\infty}
    \end{equation}
    is a sheaf. 
    Furthermore, if $(\DSm_{F}^{\calP})_{/(-)}$ denotes the (full) subfunctor spanned by proper morphisms, this too is a sheaf. 
\end{prop}

\begin{proof}
    Let $S$ be a covering sieve on $X\in\DSm$ generated by a covering family $\calU$ and consider the commutative square 
    \begin{equation}
    \begin{tikzcd}
	{(\DSm_{F})_{/X}} & {\lim_{U \in C(\calU)^{\op}} (\DSm_{F})_{/U}} \\
	{\Shv(\DSm_{F})_{/X}} & {\lim_{U \in C(\calU)^{\op}} \Shv(\DSm_{F})_{/U}}.
	\arrow[from=1-1, to=1-2]
	\arrow[hook', from=1-1, to=2-1]
	\arrow[hook', from=1-2, to=2-2]
	\arrow[from=2-1, to=2-2]
    \end{tikzcd}
    \end{equation}
    Because $\Shv(\DSm_{F})$ is a topos, \cite[Theorem 6.1.3.9]{lurieHigherToposTheory2009} implies the bottom horizontal functor is an equivalence. 
    Therefore, it is enough to show the top horizontal map is essentially surjective.
    Let $\Theta \in (\DSm_{F})_{/U}$; we may view this as a Cartesian section
    \begin{equation}
    \Theta: C(\calU) \rightarrow \int^{\rm c}_{U\in C(\calU)} (\DSm_{F})_{/U} \simeq C(\calU)\times_{\DSm_{F}} \DSm_{F}^{[1]}.
    \end{equation} 
    Let $\Theta(U)=(f_{U},\alpha_{U}): W_{U} \rightarrow U$. The Cartesianness of $\Theta$ implies that, for every $V \subset U$ in $S$, the commutative square 
    \begin{equation}
    \begin{tikzcd}\label{equation: Theta Cartesian}
	{W_{V}}\arrow[dr, phantom, "\lrcorner", very near start] & {W_{U}} \\
	V & U
	\arrow[from=1-1, to=1-2]
	\arrow["{(f_{V},\alpha_{V})}"', from=1-1, to=2-1]
	\arrow["{(f_{U},\alpha_{U})}", from=1-2, to=2-2]
	\arrow[hook, from=2-1, to=2-2]
    \end{tikzcd}
    \end{equation}
    is a pullback in $\DSm_{F}$. 
    Post-composing with the Yoneda embedding, we obtain a Cartesian section 
    \begin{equation}
        \Theta^{\Shv} : C(\calU) \rightarrow \int^{\rm c}_{U\in C(\calU)} \Shv(\DSm_{F})_{/U} \simeq C(\calU)\times_{\DSm_{F}} \Shv(\DSm_{F})^{[1]}
    \end{equation}
    which induces a morphism of sheaves 
    \begin{equation}
        W\equiv \colim_{U \in C(\calU)} W_{U} \rightarrow \colim_{U \in C(\calU)} U \simeq X. 
    \end{equation}
    To complete the proof of the first claim, we will show that $W$ is representable and that the pullback of $W \rightarrow X$ along $U \rightarrow X$ is equivalent to 
    \begin{equation}
        (f_{U},\alpha_{U}) :W_{U} \simeq U\times_{X}W \rightarrow U
    \end{equation}
    for all $U \in C(\calU)$. 
    By \eqref{equation: Theta Cartesian}, the functor 
    \begin{equation}
    \ev_{s} \circ \Theta : C(\calU) \rightarrow \DSm_{F}
    \end{equation}
    factors through $\sigma_{F}:\DSm \rightarrow \DSm_{F}$. 
    Because $\sigma_{F}$ is a continuous functor between subcanonical sites, it follows that 
    \begin{equation}
    W =\colim_{U \in C(\calU)} W_{U} \simeq \colim_{U\in C(\calU)}(\sigma_{F})_{!}W_{U} \simeq (\sigma_{F})_{!}\left(\colim_{U \in C(\calU)} W_{U}\right),
    \end{equation}
    where the last colimit is formed in $\Shv(\DSm)$. 
    As $\DSm$ is a perfect topological site, representability is a local property (cf. \cite[2.8.47]{pardon-DSm}), so $\colim_{U \in C(\calU)} W_{U} \in \Shv(\DSm)$ is representable, and therefore so is  $(\sigma_{F})_{!}\left(\colim_{U \in C(\calU)} W_{U}\right)$; we may now identify $W \rightarrow X$ with a morphism $(f,\alpha) : W \rightarrow X\in\DSm_{F}$. 
    By the universality of colimits in $\Shv(\DSm_{F})$, the pullback of $(f,\alpha)$ along $U \rightarrow X$ is given by     
    \begin{equation}
    \begin{tikzcd}
	{\colim_{V\in C(\calU)}\big((U\times_{X}V)\times_{V}W_{V}\big)}\arrow[dr, phantom, "\lrcorner", very near start] & {\colim_{V\in C(\calU)}W_{V} \simeq W} \\
	{U\simeq \colim_{V\in C(\calU)}(U\times_{X}V)} & {\colim_{V\in C(\calU)}V\simeq X}
	\arrow[from=1-1, to=1-2]
	\arrow[from=1-1, to=2-1]
	\arrow[from=1-2, to=2-2]
	\arrow[from=2-1, to=2-2]
    \end{tikzcd}
    \end{equation}
    in $\Shv(\DSm_{F})$. 
    However, as $U\times_{X}V \rightarrow V$ is an open embedding for all $V \in C(\calU)$, we have an identification 
    \begin{equation}
    \colim_{V\in C(\calU)}\big((U\times_{X}V)\times_{V}W_{V}\big) \simeq \colim_{V \in C(\calU)}(W_{U\times_{X}V}) \simeq \colim_{V \in V}(W_{U}\times_{X} V) \simeq W_{U},
    \end{equation}
    which completes the proof of the first claim. Finally, properness is local on the target, so $(\DSm_{F}^{\calP})_{/(-)}$ is also a sheaf. 
\end{proof}

Given $X \in \DSm_{\TKO}$, let $(\DSm^{\calP}_{\tau})_{/X} \subset (\DSm^{\calP}_{\TKO})_{/X}$ denote the full subcategory spanned by objects of the form $(f,\tau_{f}) :W \rightarrow X$; this determines a subfunctor 
    \begin{equation}
    (\DSm^{\calP}_{\tau})_{/(-)} \subset (\DSm^{\calP}_{\TKO})_{/(-)}.
    \end{equation}

\begin{lem}\label{lem: DSm_t sheaf}
    The functor 
    \begin{equation}
    (\DSm^{\calP}_{\tau})_{/(-)} : \DSm^{\op} \rightarrow \Cat_{\infty}
    \end{equation}
    is a sheaf. 
\end{lem}

\begin{proof}
    By Proposition~\ref{prop: DSm_F descent}, we have a commutative square 
    \begin{equation}
    \begin{tikzcd}
	{(\DSm^{\calP}_{\tau})_{/X}} & {\lim_{U\in C(\calU)^{\op}}(\DSm^{\calP}_{\tau})_{/U}} \\
	{(\DSm^{\calP}_{\TKO})_{/X}} & {\lim_{U \in C(\calU)^{\op}}(\DSm^{\calP}_{\TKO})_{/U}}.
	\arrow[from=1-1, to=1-2]
	\arrow[hook, from=1-1, to=2-1]
	\arrow[hook, from=1-2, to=2-2]
	\arrow["\sim"', from=2-1, to=2-2]
    \end{tikzcd}
    \end{equation}
    Therefore, it is enough to show the top horizontal map is essentially surjective. 
    Tracing through the proof of Proposition~\ref{prop: DSm_F descent}, an object in $\lim_{U \in C(\calU)^{\op}}(\DSm^{\calP}_{\tau})_{/U}$ determines a morphism $(f,v) : W \rightarrow X$, where $v \in \TKO(W)$ corresponds to $(\tau_{f_{U}})_{U \in C(\calU)^{\op}} \in \lim_{U \in C(\calU)^{\op}}\TKO(W_{U})$. 
    However, as $\TKO$ is a sheaf, it follows that $v \simeq \tau_{f}$; the lemma follows.
\end{proof}

\begin{prop}\label{prop: Cor_phi(Sm) topology}
    There is a Grothendieck topology on $\Cor_{\varphi}(\DSm)^{\op}$ in which a sieve $S$ on $X$ is covering if and only if it is generated by a family
    \begin{equation}
        \mathcal{U} = \{U_{i}\xleftarrow{\id}U_{i} \hookrightarrow X\}_{i \in I},
    \end{equation}
    where $\{U_{i}\}_{i \in I}$ is an open cover of $X$.
    This site has the following properties.
    \begin{enumerate}
    \item The inclusion $\DSm \hookrightarrow \Cor_{\varphi}(\DSm)^{\op}$ is continuous. 
        \item A presheaf $\m{F}\in\PShv\big(\Cor_{\varphi}(\DSm);\m{E}\big)$, for $\m{E}\in\PrL$, is a sheaf if and only if $\m{F}|_{\DSm^{\op}}$ is a sheaf. 
        \item It is subcanonical.
    \end{enumerate}
    Furthermore, there is an induced Grothendieck topology on the subcategory 
    \begin{equation}
    \Cor_{\varphi}(\Sm)^{\op} \subset \Cor_{\varphi}(\DSm)^{\op}
    \end{equation}
    satisfying the analogues of (1)-(3) above. 
    
\end{prop}

\begin{proof}
    It is routine to verify that 
    \begin{equation}
    \Cor_{\varphi}(\DSm)^{\op} \simeq \Cor(\DSm_{\frakB};\calP_{\varphi},\DSm)
    \end{equation}
    satisfies the hypotheses of Lemma~\ref{lem: Cor Grothendieck topology}. Corollary~\ref{cor: sheaves on Cor(C;L,R)} applied to this situation implies (1) and (2) above.
    To prove (3), note that we have a canonical equivalence
    \begin{equation}
    \Map_{\Cor_{\varphi}(\DSm)^{\op}}(-,X) \simeq \left(\DSm_{/X} \times_{\DSm_{\frakB}} (\DSm_{\varphi}^{\calP})_{/(-)}\right)^{\simeq}, 
    \end{equation}
    where 
    \begin{equation}
    \DSm_{\varphi}^{\calP} = \DSm^{\calP}_{\tau} \times_{\DSm^{\calP}_{\TKO}} \DSm^{\calP}_{\frakB}.  
    \end{equation}
    Using the fact that $(-)^{\simeq}$ preserves small limits, we can conclude by combining Proposition~\ref{prop: DSm_F descent} and Lemma~\ref{lem: DSm_t sheaf}; this proves the first part.

    For the second part, it is not difficult to see that the subcategory $\Cor_{\varphi}(\Sm)^{\op}$ inherits the claimed Grothendieck topology satisfying (1) and (2). 
    To establish the analogue of (3), let $\DSm^{\leq 1}$ denote the subcategory of $\DSm$ spanned by amplitude at most $1$ derived smooth manifolds and amplitude at most $1$ morphisms between them. 
    As above, we have a canonical equivalence of functors 
    \begin{equation}
    \Map_{\Cor_{\varphi}(\Sm)^{\op}}(-,X) \simeq \left(\DSm^{\leq 1}_{/X} \times_{\DSm^{\leq 1}_{\frakB}} (\DSm^{\leq 1,\calP}_{\varphi})_{/(-)} \right)^{\simeq}.
    \end{equation}
    All the arguments above can be adapted \textit{mutatis mutandis} to show that this presheaf is in fact a sheaf; the key point is that the amplitude of a map is local on the target.     
\end{proof}

\begin{prop}
    There is a Grothendieck topology on $\TrCor^{\ex}_{\varphi}(\Sm)^{\op}$ in which a sieve $S$ on $(X,d_{X})$ is covering if and only if it is generated by a family of the form 
    \begin{equation}
        \left\{(U_i,d_{X}|_{U_i})\xleftarrow{\id} (U_i,d_{X}|_{U_i}) \hookrightarrow (X,d_{X}) \right\}_{i \in I},
    \end{equation}
    where $\{U_{i} \rightarrow X\}_{i \in I}$ is an open cover in $\Sm$. 
    We call this the open cover topology on $\TrCor_{\varphi}^{\ex}(\Sm)^{\op}$.
\end{prop}

\begin{proof}
    This follows directly from Proposition~\ref{prop: Cor_phi(Sm) topology} and Lemma~\ref{lem: Cartesian lift of site}. 
\end{proof}

\begin{defin}
Let $\m{E}$ be an $\infty$-category. We say $\m{F} \in \GysP_{\varphi}(\Sm;\m{E})$ is a \textit{tangentially $\varphi$-structured Gysin sheaf} (or simply, \emph{$\varphi$-Gysin sheaf}), valued in $\m{E}$, provided $\m{F}$ satisfies descent for the open cover topology. The $\infty$-category of $\m{E}$-valued $\varphi$-Gysin sheaves is denoted by
    \begin{equation}
        \Gys_{\varphi}(\Sm;\m{E}) \subset \GysP_{\varphi}(\Sm;\m{E}).
    \end{equation}
\end{defin}

Observe, the inclusion of $\varphi$-Gysin sheaves into $\varphi$-Gysin presheaves admits a left adjoint given by sheafification:
    \begin{equation}
        L_{\sh} : \GysP_{\varphi}(\Sm;\m{E}) \rightarrow \Gys_{\varphi}(\Sm;\m{E}). 
    \end{equation}

\begin{rem}\label{rem: Gysin Day conv Lurie tensor}
    By \cite[Remark 1.3.1.6]{lurieSAG}, for any $\m{E} \in \PrL$, we have 
    \begin{equation}
    \Gys_{\varphi}(\Sm) \otimes \m{E} \xrightarrow{\sim} \Gys_{\varphi}(\Sm;\m{E}) 
    \end{equation}
    as well as a commutative diagram 
    \begin{equation}
    \begin{tikzcd}
	{\GysP_{\varphi}(\Sm)\otimes \m{E}} & {\GysP_{\varphi}(\Sm;\m{E})} \\
	{\Gys_{\varphi}(\Sm)\otimes \m{E}} & {\Gys_{\varphi}(\Sm;\m{E})}.
	\arrow["\sim", from=1-1, to=1-2]
	\arrow["{L_{\sh}\otimes\m{E}}"', from=1-1, to=2-1]
	\arrow["{L_{\sh}}", from=1-2, to=2-2]
	\arrow["\sim"', from=2-1, to=2-2]
    \end{tikzcd}
    \end{equation}
    It follows that $\Gys_{\varphi}(\Sm;\m{E})$ is presentable; if $\m{E}$ is also stable, then $\Gys_{\varphi}(\Sm;\m{E})$ is stable.
    Finally, note that the top horizontal map in the diagram above is an equivalence of presentably symmetric monoidal \categories.
\end{rem}

\begin{prop}\label{prop: Gysin sheaves monoidal}
    Let $\m{E}$ be a presentably symmetric monoidal \category. 
    The functor $L_{\sh} : \GysP_{\varphi}(\Sm;\m{E}) \rightarrow \Gys_{\varphi}(\Sm;\m{E})$ is compatible with the Day convolution product in the sense of \cite[2.2.1.6]{lurieHigherAlgebra2017}: that is, for all $\m{F},\m{G} \in \GysP_{\varphi}(\Sm;\m{E})$, the induced map 
    \begin{equation}
        \m{F} \otimes \m{G} \rightarrow (L_{\sh}\m{F})\otimes \m{G}
    \end{equation}
    is an $L_{\sh}$-local equivalence. Consequently, $\Gys_{\varphi}(\Sm;\m{E})$ inherits the structure of a symmetric monoidal \category with respect to which $L_{\sh}$ is a symmetric monoidal functor.    
\end{prop}

\begin{proof}
    By Remark~\ref{rem: Gysin Day conv Lurie tensor} we can reduce to the universal case $\m{E} =\Spaces$.
    Because the Day convolution product is compatible with small colimits and $L_{\sh}$ is a left adjoint, we may assume $\m{G}$ is representable:
    \begin{equation}
    \m{G}=\Map_{\TrCor^{\rm{ex}}_\varphi(\Sm)}\big((Y,d_Y),-)\big)\equiv\Tr_{\varphi}^{(Y,d_{Y})}.
    \end{equation}
    By the Yoneda lemma, it will be enough to show that the induced map 
    \begin{equation}
    \Map_{\GysP_{\varphi}(\Sm)}(L_{\sh}\m{F} \otimes \Tr_{\varphi}^{(Y,d_{Y})},\m{H}) \rightarrow \Map_{\GysP_{\varphi}(\Sm)}(\m{F} \otimes \Tr_{\varphi}^{(Y,d_{Y})},\m{H})
    \end{equation}
    is an equivalence for all $\m{H} \in \Gys_{\varphi}(\Sm)$. 
    As $\GysP_{\varphi}(\Sm)$ is presentably symmetric monoidal, it admits an internal Hom which we denote by $\Hom(-,-)$. 
    By rewriting the map above using the tensor-Hom adjunction in $\GysP_{\varphi}(\Sm)$, the desired result follows provided  
    \begin{equation}
    \Hom(\Tr_{\varphi}^{(Y,d_{Y})},\m{H}) \simeq \m{H}\big(-\boxtimes(Y,d_{Y})\big)
    \end{equation}
    is a $\varphi$-Gysin sheaf; observe, the aforementioned equivalence follows from the description of the internal Hom for the Day convolution product \cite[Proposition 3.11]{nikolausStable$infty$OperadsMultiplicative2016}.
    Let $\mathcal{U}$ be an open covering of $(X,d_{X})$ determined by an open covering $\{U_{i} \rightarrow X\}_{i\in I}$ and consider the natural map 
    \begin{equation}
    \m{H}\big((X,d_{X})\boxtimes(Y,d_{Y})\big) \rightarrow \lim_{U \in C(\calU)^{\op}} \m{H}\big((U,d_{X}|_{U})\boxtimes(Y,d_{Y})\big);
    \end{equation}
    this is clearly an equivalence as the family $\{U_{i}\times Y \rightarrow X \times Y\}_{i \in I}$ is an open covering.  
\end{proof}

The following lemmata provide a useful connection between $\varphi$-Gysin sheaves and sheaves on smooth manifolds. 
For each integer $d \in \bbZ$, let 
\begin{equation}
    i_{d} : \Sm^{\op} \rightarrow \TrCor_{\varphi}^{\ex}(\Sm) 
\end{equation}
denote the functor given by $X\mapsto (X,d)$. 

\begin{lem}\label{lem: sheaf detection}
    If $\m{F} \in \GysP_{\varphi}(\Sm;\m{E})$, then $\m{F}$ is a sheaf (with respect to the open cover topology) if and only if the following two conditions are satisfied.
    \begin{enumerate}
        \item For every countable set of objects $\big\{(X_{j},d_{j})\big\}_{j \in J}$ in $\TrCor_{\varphi}^{\ex}(\Sm)$, with $d_{j} \in \bbZ$ for all $j$, the canonical map 
        \begin{equation}
            \m{F}(X,d_{X}) \rightarrow \prod_{j\in J}\m{F}(X_{j},d_{j})
        \end{equation}
        is an equivalence, where $X\equiv\sqcup_{j \in J}X_{j}$ and $d_{X} : \Sing(X) \rightarrow \bbZ$ is the induced function.
        \item For every $d \in \bbZ$, the presheaf $i_{d}^{\ast}\m{F} = \m{F}(-,d)$ is actually a sheaf. 
    \end{enumerate}
\end{lem}

\begin{proof}
    The only if direction is clear, so assume $\m{F}$ is a $\varphi$-Gysin presheaf satisfying (1) and (2) above. 
    Let $\calU\equiv\big\{(U_{i},d_{U_i}) \rightarrow (X,d_{X})\big\}_{i \in I}$ be an open cover. 
    As $X \in \Sm$, it is either path-connected or a countable disjoint union of path components, each of which is an open submanifold of $X$. 
    If $X$ is path-connected, there is nothing to prove as $d_{X} : \Sing(X) \rightarrow \bbZ$ is constant and so (2) implies $\m{F}$ satisfies descent for the given cover. 
    If $X = \sqcup_{j \in J} X_{j}$, then the restriction of $d_{X}$ to $\Sing(X_{j})$ is constant, i.e., given by a fixed $d_{j} \in \bbZ$. 
    Furthermore, for each $j \in J$, we obtain a covering $\calU_{j}$ of $X_{j}$ by pulling back $\calU$ along the inclusion.
    However, as $\{X_{j}\}_{j\in J}$ is itself an open covering of $X$, the collection of covering families $\calU_{j}$ determines another covering family of $(X,d_{X})$ given by 
    \begin{equation}
        \mathcal{V}\equiv\big\{(U_{i}\cap X_{j},d_{j}) \rightarrow (X,d_{X})\big\}_{i,j};
    \end{equation}
    observe, $\mathcal{V}$ is a refinement of $\mathcal{U}$. 
    Furthermore, as the $X_{j}$'s are all disjoint, we find that there is an equivalence of \v Cech nerves
    \begin{equation}
        \bigsqcup_{j \in J} C(\calU_{j}) \simeq  C(\calV). 
    \end{equation}
    Subsequently, we obtain the following commutative diagram 
    \begin{equation}
    \begin{tikzcd}
	{\m{F}(X,d_{X})} & {\lim_{U\in C(\calU)}\m{F}(U,d_{U})} \\
	{\prod_{j\in J}\m{F}(X_{j},d_{j})} & {\prod_{j \in J}\lim_{U_{j}\in C(\calU_{j})}\m{F}(U_{j},d_{j})},
	\arrow[from=1-1, to=1-2]
	\arrow[from=1-1, to=2-1]
	\arrow["\sim", from=1-2, to=2-2]
	\arrow[from=2-1, to=2-2]
    \end{tikzcd}
    \end{equation}
    where the right vertical arrow is an equivalence since $\calV$ refines $\calU$. 
    Therefore, $\m{F}$ will be a $\varphi$-Gysin sheaf as soon as the left vertical and bottom horizontal arrows are equivalences, but these are by assumptions (1) and (2) above.     
\end{proof}

The proof of the preceding lemma easily implies the following useful corollary. 

\begin{cor}\label{cor: i_d conservative family}
    The collection of functors 
    \begin{equation}
        \big\{i_{d}^{\ast} : \Gys_{\varphi}(\Sm;\m{E}) \rightarrow \Shv(\Sm;\m{E})\big\}_{d \in \bbZ}
    \end{equation}
    is a jointly conservative family. 
\end{cor}

\begin{lem}\label{lem: sheafification i_d}
    For all $\m{F} \in \GysP_{\varphi}(\Sm;\m{E})$, the canonical map
    \begin{equation}\label{eqn:sheafification i_d}
        L_{\sh}i_{d}^{\ast}\m{F} \rightarrow i_{d}^{\ast}L_{\sh}\m{F},\;\;d\in\bbZ
    \end{equation}
    is an equivalence of sheaves on $\Sm$.     
\end{lem}

\begin{proof}
    By the construction of the sheafification functor on a site \cite[6.2.2.9]{lurieHigherToposTheory2009}, it will be enough to show the canonical map 
    \begin{equation}
        (i^{\ast}_{d}\m{F})^\dagger \rightarrow i_{d}^{\ast}\m{F}^{\dagger}
    \end{equation}
    is an equivalence of presheaves on $\Sm$. 
    Given $X \in \Sm$, the value of \eqref{eqn:sheafification i_d} is given by
    \begin{equation}\label{eqn:auxpsheaf}
        \colim_{S\in \Cov(X)} \lim_{U \in C(\calU)^{\op}} \m{F}(U,d) \rightarrow \colim_{T \in \Cov(X,d)} \lim_{V \in T^{\op}} \m{F}(V,d). 
    \end{equation}
    By the definition of the Grothendieck topology on $\TrCor_{\varphi}^{\ex}(\Sm)^{\op}$, there is an isomorphism between the poset of covering families on $X \in \Sm$ and the poset of covering families on $(X,d)\in\TrCor_{\varphi}^{\ex}(\Sm)^{\op}$. 
    Finally, by the equivalence of sieve and \v Cech descent, \eqref{eqn:auxpsheaf} must be an equivalence.
\end{proof}

\begin{prop}
The $\varphi$-Gysin presheaf 
\begin{equation}
\Tr_{\varphi} : \TrCor_{\varphi}^{\ex}(\Sm) \rightarrow \Spaces
\end{equation}
is in fact a $\varphi$-Gysin sheaf. 
\end{prop}

\begin{proof}
    
    To see that $\Tr_{\varphi}$ is a sheaf, let $\calU$ be a covering family of $(X,d_{X})$ and consider the following commutative square:
    \begin{equation}
\begin{tikzcd}
	{\Tr_{\varphi}(X,d_{X})} & {\lim_{U \in C(\calU)^{\op}}\Tr_{\varphi}(U,d_{X}|_{U})} \\
	{\Map_{\Cor_{\varphi}(\Sm)}(\bbR^{0},X)} & {\lim_{U \in C(\calU)^{\op}}\Map_{\Cor_{\varphi}(\Sm)}(\bbR^{0},U)}.
	\arrow[from=1-1, to=1-2]
	\arrow[from=1-1, to=2-1]
	\arrow[from=1-2, to=2-2]
	\arrow["\sim"', from=2-1, to=2-2]
\end{tikzcd}
    \end{equation}
    The claim will follow as soon as we can show the aforementioned diagram is actually a pullback in $\Spaces$ for every choice of basepoint $(f,\alpha) \in \Map_{\Cor_{\varphi}(\Sm)}(\bbR^{0},X)$. 
    To this end, fix such a $\varphi$-structured map $(f,\alpha) : W \rightarrow X$ whose image under the bottom equivalence in the aforementioned diagram is given by the system $\big\{(f_{U},\alpha_{U}) : W_{U} \rightarrow U\big\}_{U \in C(\calU)^{\op}}$. 
    By Lemma~\ref{lem: mapping spaces}, the induced map between the vertical fibers,
    \begin{equation}
    \mathrm{fib}_{(f,\alpha)} \rightarrow \lim_{U \in C(\calU)^{\op}} \mathrm{fib}_{(f_{U},\alpha_{U})},
    \end{equation}
    also fits into a commutative square:
    \begin{equation} 
    \begin{tikzcd}[center picture]
	{\mathrm{fib}_{(f,\alpha)}} & {\lim_{U\in C(\calU)^{\op}}\mathrm{fib}_{(f_{U},\alpha_{U})}} \\
	{\Map(\Th_{W}(\tau_f)\otimes_\bbS\rmM\varphi_W\{-f^*d_{X}\},f^{!}\rmM\varphi)} & {\lim_{U\in C(\calU)^{\op}}\Map\big(\Th_{W_{U}}(\tau_{f_U})\otimes_\bbS\rmM\varphi_W\{-f_U^*d_{X}|_{U}\},f_{U}^{!}\rmM\varphi\big)},
	\arrow[from=1-1, to=1-2]
	\arrow[hook, from=1-1, to=2-1]
	\arrow[hook, from=1-2, to=2-2]
	\arrow[from=2-1, to=2-2]
    \end{tikzcd}
    \end{equation}
    where, by Lemma~\ref{lem: mapping spaces}, the vertical arrows are inclusions of subspaces given by the extended transfer classes .
    Using the trivializations $\Th_{W}(\tau_{f_U}) \simeq \rmM\varphi_{W}$ induced by $\alpha_U$'s and the equivalences $f^{!}_{U}j_{U}^{\ast} \simeq j_{W_{U}}^{\ast}f_{U}^{!}$, the bottom horizontal map of the previous diagram can be identified with
    \begin{equation}
    \Omega^{\infty}\Gamma\big(W;f^{!}(\rmM\varphi_{X}\{d_{X}\})\big)\rightarrow \lim_{U\in C(\calU)^{\op}} \Omega^{\infty}\Gamma\big(W_{U};j_{W_{U}}^{\ast}f^{!}(\rmM\varphi_{X}\{d_{X}\})\big),
    \end{equation}
    where the latter is an equivalence since $f^{!}(\rmM\varphi_{X}\{d_{X}\})$ is a sheaf. 
    However, as being an extended transfer class is a property local on the target, this equivalence restricts to the top horizontal map, whence the claim. 
\end{proof}

\subsection{Homotopy invariant Gysin (pre)sheaves}
\begin{defin}
Let $\m{F}\in\PShv(\Sm;\m{E})$. We say $\m{F}$ is \emph{homotopy invariant} (or \emph{$\bbR$-invariant}) if, for every $X\in\Sm$, the map 
    \begin{equation}
    \m{F}(X)\to\m{F}(X\times\bbR)
    \end{equation}
induced by the obvious projection $X\times\bbR\to X$ is an equivalence. We denote by 
    \begin{equation}
    \PShv_\bbR(\Sm;\m{E})\subset\PShv(\Sm;\m{E})
    \end{equation}
the full subcategory of such objects. Analogously, we say $\m{F}\in\Shv(\Sm;\m{E})$ is \emph{homotopy invariant} if its underlying presheaf is; we denote by 
    \begin{equation}
    \Shv_\bbR(\Sm;\m{E})\subset\Shv(\Sm;\m{E})
    \end{equation}
the full subcategory of such objects.
\end{defin}

Recall, we have the constant sheaf/global sections adjunction: 
    \begin{equation}
    \Gamma^*:\m{E}\rightleftarrows\Shv(\Sm;\m{E}):\Gamma_*.
    \end{equation}
The following result of Dugger \cite{dugger} shows that, once one requires homotopy invariance, differential cohomology theories are determined by homotopy theory.

\begin{thm}[\cite{dugger}]
For any presentable $\infty$-category $\m{E}$, 
    \begin{equation}
    \Gamma_*:\Shv_\bbR(\Sm;\m{E})\to\m{E}
    \end{equation}
is an equivalence with inverse $\Gamma^*$.\footnote{In particular, a constant sheaf is homotopy invariant; this is a key step in the proof.} If $\m{E}$ is presentably symmetric monoidal, then the aforementioned equivalence is in fact one of presentably symmetric monoidal $\infty$-categories. Moreover, the aforementioned equivalence is natural in $\m{E}$.
\end{thm}

\begin{defin}
    Let $\m{F} \in \GysP_{\varphi}(\Sm)$. We say that $\m{F}$ is \textit{homotopy invariant} (or \textit{$\bbR$-invariant}) if, for every $(X,d_{X}) \in \TrCor_{\varphi}^{\ex}(\Sm)$ the map 
    \begin{equation}
    \m{F}(X,d_{X})\to\m{F}(X\times\bbR,\pr_{1}^{\ast}d_{X})
    \end{equation}
    induced by the obvious projection $(X,d_{X})\boxtimes(\bbR,0) \simeq (X\times\bbR,\pr_{1}^{\ast}d_{X}) \to (X,d_{X})$ is an equivalence. We denote by 
    \begin{equation}
    \GysP_{\varphi}^{\bbR}(\Sm;\m{E})\subset \GysP_{\varphi}(\Sm;\m{E})
    \end{equation}
    the full subcategory of such objects. Analogously, we say $\m{F}\in\Gys_{\varphi}(\Sm;\m{E})$ is \emph{homotopy invariant} if its underlying $\varphi$-Gysin presheaf is; we denote by 
    \begin{equation}
    \Gys_{\varphi}^{\bbR}(\Sm;\m{E})\subset\Gys_{\varphi}(\Sm;\m{E})
    \end{equation}
    the full subcategory of such objects.
\end{defin}

\begin{rem}
    It is clear that restriction along $\Sm^{\op} \rightarrow \TrCor_{\varphi}^{\ex}(\Sm)$ sends homotopy invariant (pre)sheaves to homotopy invariant (pre)sheaves. 
    In particular, we have a functor
    \begin{equation}
        \Gys_{\varphi}^{\bbR}(\Sm;\m{E}) \rightarrow \Shv_{\bbR}(\Sm;\m{E}). 
    \end{equation}
\end{rem}

\begin{construction}[Morel-Suslin-Voevodsky]\label{construction: MSV construction}
    Given $n \geq 0$, we let $\Deltasm^{n}$ denote the hyperplane in $\bbR^{n+1}$ defined by 
    \begin{equation}
    \Deltasm^{n} =  \Bigg\{(x_{0},\dots,x_{n}) \in \bbR^{n+1} \mid \sum_{i=0}^{n} x_{i} = 1\Bigg\}. 
    \end{equation}
    By the usual argument, this determines a cosimplicial smooth manifold $\Deltasm^{\bullet} : \Delta \rightarrow \Sm$, and thus a simplicial object
    \begin{equation}
    (\Deltasm^{\bullet},0) : \Delta^{\op} \rightarrow\Sm^{\op} \rightarrow \TrCor_{\varphi}^{\ex}(\Sm). 
    \end{equation}
    Given $\m{F} \in \GysP_{\varphi}(\Sm;\m{E})$, we define $\rmH\m{F} \in \GysP_{\varphi}(\Sm)$ to be the geometric realization of the simplicial spatial$\varphi$-Gysin presheaf  
    \begin{equation}
    \m{F}\big(-\boxtimes(\Deltasm^{\bullet},0)\big) : \Delta^{\op} \rightarrow \GysP_{\varphi}(\Sm). 
    \end{equation}
    On objects, this is given by the formula 
    \begin{equation}
    \rmH\m{F}(X,d_{X}) \simeq\big\lvert \m{F}\big((X,d_{X})\boxtimes(\Deltasm^{\bullet},0)\big)\big\rvert \simeq \big\lvert\m{F}(X\times\Deltasm^{\bullet},\pr_{1}^{\ast}d_{X})\big\rvert. 
    \end{equation}
    Because $\bbR^0$ is terminal in $\Sm$, there is a natural transformation $\Delta \times [1] \rightarrow \Sm$ from $\Deltasm^{\bullet}$ to the constant cosimplicial object on $\bbR^{0}$. 
    Therefore, we have a functor
    \begin{equation}
    \rmH : \GysP_{\varphi}(\Sm;\m{E}) \rightarrow \GysP_{\varphi}(\Sm;\m{E})
    \end{equation}
    together with a natural transformation
    \begin{equation}
    \eta:\mathrm{id}_{\GysP_{\varphi}} \rightarrow \rmH. 
    \end{equation}
\end{construction}

The following lemma describes the basic features of $\rmH$ and $\eta$ constructed above. 

\begin{lem}\label{lemma: properties of MSV construction}
    Let $\m{E}$ be a presentable \category. 
    \begin{enumerate}
        \item For all $\m{F} \in \GysP_{\varphi}(\Sm;\m{E})$, $\rmH\m{F}$ is a homotopy invariant $\varphi$-Gysin presheaf. 
        \item If $E$ is a homotopy invariant $\varphi$-Gysin presheaf, then $\eta_{\m{F}} : \m{F} \rightarrow \rmH\m{F}$ is an equivalence. 
        \item For any $\m{F} \in \GysP_{\varphi}(\Sm;\m{E})$, the map $\rmH\eta_{\m{F}} : \rmH\m{F} \rightarrow \rmH(\rmH\m{F})$ is an equivalence. 
    \end{enumerate}
    Hence, $\rmH$ determines a Bousfield localization of $\GysP_{\varphi}(\Sm;\m{E})$ which we denote by 
    \begin{equation}
    L_{\bbR} : \GysP_{\varphi}(\Sm;\m{E}) \rightarrow \GysP_{\varphi}^{\bbR}(\Sm;\m{E}). 
    \end{equation}
\end{lem}

\begin{proof}
    The proof is essentially identical to the arguments in \cite[I.5]{diff-coh-book}. 
\end{proof}

\begin{prop}\label{prop: homotopy invar gysin sym mon}
    Let $\m{E}$ be a presentably symmetric monoidal \category. 
    The functor $L_{\bbR} : \GysP_{\varphi}(\Sm;\m{E}) \rightarrow \GysP^{\bbR}_{\varphi}(\Sm;\m{E})$ is compatible with the Day convolution product in the sense of \cite[2.2.1.6]{lurieHigherAlgebra2017}: that is, for all $\m{F},\m{G} \in \GysP_{\varphi}(\Sm;\m{E})$, the induced map 
    \begin{equation}
        \m{F} \otimes \m{G} \rightarrow (\rmH\m{F})\otimes \m{G}
    \end{equation}
    is an $H$-local equivalence. Consequently, $\GysP_{\varphi}^{\bbR}(\Sm;\m{E})$ inherits the structure of a symmetric monoidal \category with respect to which $L_{\bbR}$ is a symmetric monoidal functor.    
\end{prop}

\begin{proof}
    As in the proof of Proposition~\ref{prop: Gysin sheaves monoidal}, we may reduce to the universal case $\m{E} = \Spaces$, and, as $L_{\bbR}$ preserves small colimits, we may assume $\m{G} = \Tr_{\varphi}^{(Y,d_{Y})}$. 
    It is enough to see that
    \begin{equation}
        \Map_{\GysP_{\varphi}}(\m{F}\otimes\Tr_{\varphi}^{(Y,d_{Y})},\m{G}) \rightarrow \Map_{\GysP_{\varphi}}(\rmH\m{F}\otimes\Tr_{\varphi}^{(Y,d_{Y})},\m{G}) 
    \end{equation}
    is an equivalence for all $\m{G} \in \GysP_{\varphi}^{\bbR}(\Sm)$. 
    Using the tensor-Hom adjunction, it is enough to check 
    \begin{equation}
        \Hom(\Tr_{\varphi}^{(Y,d_{Y})},\m{G}) \simeq \m{G}\big(-\boxtimes(Y,d_{Y})\big)
    \end{equation}
    is homotopy invariant. 
    However, this is obvious as we assumed $\m{G}$ was homotopy invariant. 
\end{proof}

We now prove a useful lemma for recognizing when a $\varphi$-Gysin sheaf is homotopy invariant. 

\begin{lem}\label{lem: homotopy invar detection}
    Let $\m{E}$ be an $\infty$-category and $\m{F} \in \Gys_{\varphi}(\Sm;\m{E})$. 
    We have that $\m{F}$ is homotopy invariant if and only if, for each $d \in \bbZ$, the sheaf $i_{d}^{\ast}\m{F} \in \Shv(\Sm;\m{E})$ is homotopy invariant. 
\end{lem}

\begin{proof}
    Let $(X,d_{X}) \in \TrCor_{\varphi}^{\ex}(\Sm)$ and decompose $X$ into path components: $X = \sqcup_{j \in J}, X_{j}$; of course, this implies $d_{X}|_{X_j} = d_{j} \in \bbZ$.
    Because $\m{F}$ is a $\varphi$-Gysin sheaf, we obtain a commutative square
    \begin{equation}
    \begin{tikzcd}
	{\m{F}(X,d_{X})} & {\m{F}(X\times\bbR,\pr_{1}^{\ast}d_{X})} \\
	{\prod_{j\in J}\m{F}(X_{j},d_{j})} & {\prod_{j \in J}\m{F}(X_{j}\times\bbR,d_{j})};
	\arrow[from=1-1, to=1-2]
	\arrow["\sim"', from=1-1, to=2-1]
	\arrow["\sim", from=1-2, to=2-2]
	\arrow[from=2-1, to=2-2]
    \end{tikzcd}
    \end{equation}
    in particular, the top horizontal map is an equivalence if and only if the bottom horizontal map is, proving the claim. 
\end{proof}

\begin{cor}
    If $\m{F}\in\GysP^\bbR_\varphi(\Sm;\m{E})$, then $L_{\sh}\m{F}$ is homotopy invariant. 
\end{cor}

\begin{proof}
    The result follows by combining Lemmas~\ref{lem: homotopy invar detection}, \ref{lem: sheafification i_d}, and \cite[4.4.11]{diff-coh-book}.
\end{proof}

\begin{lem}
    The inclusion $\Gys_{\varphi}^{\bbR}(\Sm;\m{E}) \subset \Gys_{\varphi}(\Sm;\m{E})$ admits a left adjoint
    \begin{equation}
        L_{\rm hi} : \Gys_{\varphi}(\Sm;\m{E}) \rightarrow \Gys_{\varphi}^{\bbR}(\Sm;\m{E})
    \end{equation}
    which is equivalent to the composition
    \begin{equation}
        \Gys_{\varphi}(\Sm;\m{E})  \subset \GysP_{\varphi}(\Sm;\m{E}) \xrightarrow{L_{\bbR}} \GysP_{\varphi}^{\bbR}(\Sm;\m{E}) \xrightarrow{L_{\sh}}\Gys_{\varphi}^{\bbR}(\Sm;\m{E}).
    \end{equation} 
\end{lem}

\begin{proof}
    It is not hard to see that the inclusion $\Gys_{\varphi}^{\bbR}(\Sm;\m{E}) \subset \Gys_{\varphi}(\Sm;\m{E})$ preserves small limits and is accessible, so $L_{\rm hi}$ exists by the adjoint functor theorem.
    Now, the unit for this adjunction canonically factors as 
    \begin{equation}
    \m{F} \rightarrow L_{\sh}L_{\bbR}\m{F} \rightarrow L_{\rm hi}\m{F}. 
    \end{equation}
    By Corollary~\ref{cor: i_d conservative family} and Lemma~\ref{lem: sheafification i_d}, the second map is an equivalence as soon as 
    \begin{equation}
    L_{\sh}L_{\bbR}i_{d}^{\ast}\m{F} \rightarrow L_{\rm hi}i_{d}^{\ast}\m{F}
    \end{equation}
    is an equivalence of sheaves on $\Sm$ for all $d \in \bbZ$; the latter claim follows from from \cite[4.4.12]{diff-coh-book}. 
\end{proof}

\begin{prop}
    Let $\m{E}$ be a presentably symmetric monoidal \category. 
    The functor $L_{\rm hi} : \Gys_{\varphi}(\Sm;\m{E}) \rightarrow \Gys^{\bbR}_{\varphi}(\Sm;\m{E})$ is compatible with Day convolution product in the sense of \cite[2.2.1.6]{lurieHigherAlgebra2017}: that is, for all $\m{F},\m{G} \in \Gys_{\varphi}(\Sm;\m{E})$, the induced map 
    \begin{equation}
        \m{F} \otimes \m{G} \rightarrow (L_{\rm hi}\m{F})\otimes \m{G}
    \end{equation}
    is an $L_{\rm hi}$-local equivalence. 
    Consequently, $\Gys^{\bbR}_{\varphi}(\Sm;\m{E})$ inherits the structure of a symmetric monoidal \category with respect to which $L_{\rm hi}$ is a symmetric monoidal functor.    
\end{prop}

\begin{proof}
    As in the proof of Proposition~\ref{prop: Gysin sheaves monoidal}, we may reduce to the universal case $\m{E} = \Spaces$, and, as both $L_{\rm hi}$ and $\otimes$ preserve small colimits of sheaves, we may assume $\m{G} = \Tr_{\varphi}^{(Y,d_{Y})}$. 
    It is enough to see that
    \begin{equation}
        \Map_{\Gys_{\varphi}}(\m{F}\otimes\Tr_{\varphi}^{(Y,d_{Y})},\m{G}) \rightarrow \Map_{\Gys_{\varphi}}(L_{\rm hi}\m{F}\otimes\Tr_{\varphi}^{(Y,d_{Y})},\m{G}) 
    \end{equation}
    is an equivalence for all $\m{G} \in \Gys_{\varphi}^{\bbR}(\Sm)$, but as $\m{G}$ is a $\varphi$-Gysin sheaf, this is equivalent to showing
    \begin{equation}
        \Map_{\GysP_{\varphi}}(\m{F}\otimes\Tr_{\varphi}^{(Y,d_{Y})},\m{G}) \rightarrow \Map_{\GysP_{\varphi}}(L_{\sh}L_{\bbR}\m{F}\otimes\Tr_{\varphi}^{(Y,d_{Y})},\m{G}) 
    \end{equation}
    is an equivalence. 
    By the tensor-Hom adjunction, we only need to check that 
    \begin{equation}
        \Hom(\Tr_{\varphi}^{(Y,d_{Y})},\m{G}) \simeq \m{G}\big(-\boxtimes (Y,d_{Y})\big)
    \end{equation}
    is a homotopy invariant $\varphi$-Gysin sheaf. 
    However, this follows by our work in the proofs of Propositions~\ref{prop: Gysin sheaves monoidal} and \ref{prop: homotopy invar gysin sym mon}, whence the claim.
\end{proof}

\begin{notation}
    Let $\m{E}$ be a presentably symmetric monoidal \category. 
    We write 
    \begin{equation}
    \unit^{\m{E}}_{\varphi} \equiv L_{\hi}\Tr_{\varphi}
    \end{equation}
    for the unit of $\Gys_{\varphi}^{\bbR}(\Sm;\m{E})$. 
\end{notation}

\begin{prop}\label{prop: eval at R^0 lax sm}
    Let $\m{E}$ be a presentably symmetric monoidal \category. 
    The functor 
    \begin{equation}
    i_{0}^{\ast} : \Gys_{\varphi}^{\bbR}(\Sm;\m{E}) \rightarrow \Shv_{\bbR}(\Sm;\m{E}) \simeq \m{E}
    \end{equation}
    is canonically lax symmetric monoidal. 
\end{prop}

\begin{proof}
    This follows from the fact that $i_{0}^{\ast} : \GysP_{\varphi}(\Sm;\m{E}) \rightarrow \PShv(\Sm;\m{E})$ is canonically lax symmetric monoidal together with compatibility of $i_{0}^{\ast}$, $L_{\bbR}$, and $L_{\sh}$.
\end{proof}

\subsection{Thom-Gysin transfers}\label{subsec:shiftoperators}
In this subsection, we will treat $\TrCor^{\rm ex}_\varphi(\Sm)$ as if we had defined it to contain objects which are allowed to be smooth manifolds with boundary. This is a harmless assumption since (1) any smooth manifold with boundary is an example of a derived smooth manifold and (2) any smooth manifold with boundary is smoothly homotopic to a smooth manifold with empty boundary, namely, its interior. In particular, the value of a homotopy invariant sheaf on a smooth manifold with boundary is completely determined by its value on the interior of that smooth manifold with boundary.

The following construction is essentially due to Dold, cf. \cite[Example 3.11]{doldGeometricCobordismFixed1978}.

\begin{construction}[Thom-Gysin transfer]
Let $\pi:E\to X\in\Sm$ be a smooth vector bundle of rank $d\in\bbZ_{\geq0}$; we will consider the disk resp. sphere bundle of $E$, denoted $D(E)$ resp. $S(E)$. Of course, $D(E)$ is a smooth manifold with boundary $S(E)$. Suppose $[E]\in\TKO(X)$ admits a $\varphi$-structure $\alpha$, then the zero section embedding,
    \begin{equation}
    i:X\hookrightarrow D(E),
    \end{equation}
admits the $\varphi$-structure $-\alpha$ since
    \begin{equation}
    \tau_i\simeq\tau_X-i^*\tau_{D(E)}\simeq\tau_X-\big([E]+\tau_X)\simeq-[E].
    \end{equation}
In particular, we obtain the morphism
    \begin{equation}
    (X,0) \xleftarrow{\rm id} X \xrightarrow{(i,-\alpha,\theta)} \big(D(E),d\big)\in\TrCor^{\rm ex}_\varphi(\Sm)
    \end{equation}
which we denote by $i$ for brevity. In fact, if we consider the morphism $\big(S(D),d\big)\to\big(D(E),d\big)\in\TrCor^{\rm ex}_\varphi(\Sm)$ induced by the inclusion $S(D)\hookrightarrow D(E)$, then we obtain the following pullback square in $\TrCor^{\rm ex}_\varphi(\Sm)$:
    \begin{equation}
    \begin{tikzcd}
    (\varnothing,0)\arrow[r]\arrow[d]\arrow[dr, phantom, "\lrcorner", very near start] & \big(S(E),d\big)\arrow[d] \\
    (X,0)\arrow[r,"i"] &\big(D(E),d\big).
    \end{tikzcd}
    \end{equation}
Finally, given $\m{F}\in\Gys^\bbR_\varphi(\Sm;\Sp)$, we obtain a commutative diagram
    \begin{equation}
    \begin{tikzcd}
    \m{F}(X,0)\arrow[r,"\sim"]\arrow[d,"\exists!"] & \m{F}(X,0)\arrow[r]\arrow[d,"\m{F}i"] & \m{F}(\varnothing,0)\arrow[d] \\
    \m{F}\big(D(E),S(E),d\big)\arrow[r] & \m{F}\big(D(E),d\big)\arrow[r] & \m{F}\big(S(E),d\big),
    \end{tikzcd}
    \end{equation}
where the rows are fiber sequences. We refer to 
    \begin{equation}
    \iota_!:\m{F}(X,0)\to\m{F}\big(D(E),S(E),d\big)
    \end{equation}
as the \emph{Thom-Gysin transfer}.
\end{construction}

\begin{lem}
Let $\pi:E\to X\in\Sm$ be a smooth vector bundle of rank $d$ such that $[E]\in\TKO(X)$ admits a $\varphi$-structure $\alpha$, then, for any $\m{F}\in\Gys^\bbR_\varphi(\Sm;\Sp)$, the associated Thom-Gysin transfer is a homotopy equivalence.
\end{lem}

\begin{proof}
We may produce an explicit inverse as follows. Observe, $\pi$ induces a morphism 
    \begin{equation}
    \big(D(E),d\big)\xleftarrow{\id} D(E)\xrightarrow{(\pi,\alpha,\theta)}(X,0)\in\TrCor^{\rm ex}_\varphi(\Sm);
    \end{equation}
hence, we obtain a morphism
    \begin{equation}
    \pi_!:\m{F}\big(D(E),S(E),d\big)\to\m{F}(X,0).
    \end{equation}
Now, the composition 
    \begin{equation}
    \begin{tikzcd}
    & X\arrow[dl,"\id",swap]\arrow[dr,"i"] & & D(E)\arrow[dl,"\id",swap]\arrow[dr,"\pi"] & \\
    (X,0) & & \big(D(E),d\big) & & (X,0)
    \end{tikzcd}
    \end{equation}
is clearly the identity, therefore,  $\pi_!\iota_! \simeq \id$. Meanwhile, since $i\circ\pi$ is (smoothly) homotopic to the identity (recall, $X$ is a deformation retract of $D(E)$), the composition 
    \begin{equation}
    \begin{tikzcd}
    & D(E)\arrow[dl,"\id",swap]\arrow[dr,"\pi"] & & X\arrow[dl,"\id",swap]\arrow[dr,"i"] & \\
    \big(D(E),d\big) & & (X,0) & & \big(D(E),d\big)
    \end{tikzcd}
    \end{equation}
is equivalent to the identity, therefore, $\iota_!\pi_! \simeq\id$.
\end{proof}

\begin{notation}
Given $\m{F}\in\Gys^\bbR_\varphi(\Sm;\Sp)$ and $d\in\bbZ$, we denote by 
    \begin{equation}
    \m{F}^d\in\Gys^\bbR_\varphi(\Sm;\Sp)
    \end{equation}
the homotopy invariant $\varphi$-Gysin sheaf given by 
    \begin{equation}
    \m{F}^d(-)\equiv\m{F}\left(-\boxtimes(\bbR^{0},d)\right).
    \end{equation}
Explicitly, this sends $(X,d_{X})$ to $\m{F}(X,d_{X}+d)$ with the obvious effect on morphisms.
\end{notation}

\begin{cor}\label{cor: shifting = suspension}
For any $\m{F}\in\Gys^\bbR_\varphi(\Sm;\Sp)$, there is a canonical homotopy equivalence which is natural in $X\in\Sm$:
    \begin{equation}
    \m{F}(X,0) \xrightarrow{\sim}\Sigma^{-1}\m{F}(X,1)\in\Sp,
    \end{equation}
i.e., there is a canonical equivalence of underlying sheaves on $\Sm$:
    \begin{equation}
    \Sigma\m{F}\xrightarrow{\sim}\m{F}^1\in\Shv_\bbR(\Sm;\Sp).
    \end{equation}
Consequently, for any $d\in\bbZ$, it follows that there is a canonical equivalence of underlying sheaves on $\Sm$:
    \begin{equation}
    \Sigma^d\m{F} \xrightarrow{\sim} \m{F}^{d}\in\Shv_\bbR(\Sm;\Sp).
    \end{equation}
\end{cor}

\begin{proof}
Let 
    \begin{equation}
    \pi:\underline{\bbR}\equiv X\times\bbR\to X\in\Sm
    \end{equation}
be the trivial bundle of rank 1. In this case, we have that 
    \begin{equation}
    D(\underline{\bbR})=X\times[0,1]\;\;\textrm{and}\;\; S(\underline{\bbR})=X\times S^0=X\sqcup X.
    \end{equation}
In particular, since the inclusion of a component of $S(\underline{\bbR})$ into $D(\underline{\bbR})$ followed by $\pi$ is clearly the identity on $X$, we have the fiber sequence
    \begin{equation}
    \m{F}\big(D(\underline{\bbR}),S(\underline{\bbR}),1\big)\to\m{F}\big(D(\underline{\bbR}),1\big)\simeq\m{F}(X,1)\xrightarrow{\Delta}\m{F}(X,1)\oplus\m{F}(X,1)\simeq\m{F}\big(S(\underline{\bbR}),1\big).\footnote{Recall, $\m{F}$ sends sums to products since it is a sheaf.}
    \end{equation}
But this implies 
    \begin{equation}
    \m{F}(X,0)\simeq\m{F}\big(D(\underline{\bbR}),S(\underline{\bbR}),1\big)\simeq\Sigma^{-1}\m{F}(X,1).
    \end{equation}
Clearly, this argument can be made to be functorial in $X\in\Sm$, whence the claim.
\end{proof}

\subsection{Homotopy invariant Gysin sheaves and cobordism}
The purpose of this subsection is to elucidate the connection between $\varphi$-structured cobordism and homotopy invariant spatial/spectral $\varphi$-Gysin sheaves.

\begin{lem}
    For each $\m{F} \in \Gys_{\varphi}(\Sm;\Spaces)$, the counit map $\m{F} \rightarrow L_{\hi}\m{F}$ induces an equivalence of $\varphi$-Gysin sheaves
    \begin{equation}
    \rmH\m{F} \xrightarrow{\sim} L_{\hi}\m{F}. 
    \end{equation}
    Consequently, 
        \begin{equation}
        L_{\bbR} : \GysP_{\varphi}(\Sm;\Spaces) \rightarrow \GysP_{\varphi}^{\bbR}(\Sm;\Spaces)
        \end{equation}
    preserves sheaves and $L_{\bbR} \simeq L_{\hi}$. 
    The same result is true for $\varphi$-Gysin sheaves valued in $\Spaces_\ast$, $\CMon$, or $\CMongp$. 
\end{lem}

\begin{proof}
    By Lemma~\ref{lem: sheaf detection} and Corollary~\ref{cor: i_d conservative family}, it is enough to show that $\rmH\m{F}$ is a $\varphi$-Gysin sheaf and $i_{d}^{\ast}\rmH \m{F} \simeq i_{d}^{\ast}L_{\hi}\m{F}$ for all $d \in \bbZ$. 
    By \cite[Theorem 1.2]{berwick-evansClassifyingSpacesInfinitysheaves2024} combined with Lemma~\ref{lem: sheafification i_d}, we can conclude that $i_{d}^{\ast}\rmH\m{F}$ is a sheaf and that $i_{d}^{\ast}\rmH\m{F} \simeq i_{d}^{\ast}L_{\hi}\m{F}$ for all $d \in \bbZ$. 
    Therefore, it will be enough to verify condition (1) from Lemma~\ref{lem: sheaf detection}. 
    To this end, let $\big\{(X_{j},d_{j})\big\}_{j\in J}$ be a countable collection with $d_{j} \in \bbZ$ for each $j \in J$. 
    Because $\m{F}$ is a $\varphi$-Gysin sheaf, the claim follows as soon as we can verify that the limit interchange map 
    \begin{equation}
    \colim_{\Delta^{\op}}\Bigg(\prod_{j \in J}\m{F}(X_{j}\times\Deltasm^{\bullet},d_j)\Bigg) \rightarrow \prod_{j\in J} \Bigg(\colim_{\Delta^{\op}}\m{F}(X_{j}\times\Deltasm^{\bullet},d_{j})\Bigg)
    \end{equation}
    is an equivalence of spaces.  
    However, this follows \textit{mutatis mutandis} from the elegant proof of \cite[Lemma 5.13]{berwick-evansClassifyingSpacesInfinitysheaves2024}. 
\end{proof}

\begin{rem}\label{rem: cobordism relation}
Given $(X,d_{X}) \in \TrCor_{\varphi}^{\ex}(\Sm)$, we have an equivalence 
\begin{equation}
\unit^{\Spaces}_{\varphi}(X,d_{X}) \simeq \colim_{\Delta^{\op}} \Tr_{\varphi}(X\times\Deltasm^{\bullet},d_{X}). 
\end{equation}
Therefore, we can identify $\pi_{0}\unit^{\Spaces}_{\varphi}(X,d_{X})$ with the set of proper extended $\varphi$-structured maps 
\begin{equation}
(g,\alpha,\theta): W \rightarrow (Y,d_{X})
\end{equation} 
up to the relation of $\varphi$-cobordism; by the usual argument, this is indeed a group, cf. Proposition \ref{prop:spatialgysin1}.
\end{rem}

\begin{recollection}\label{recollection: semiadditive Day convolution}
Recall, by \cite[Theorem 5.1]{gepnerUniversalityMultiplicativeInfinite2015}, the symmetric monoidal functor $\Sigma_{+}^{\infty} : \Spaces \rightarrow \Sp$ admits a symmetric monoidal factorization into
\begin{equation}\label{eqn: free sm functors}
\Spaces \xrightarrow{(-)_+} \Spaces_{\ast} \rightarrow \CMon \xrightarrow{(-)^{\rm gp}} \CMongp \xrightarrow{\rmB^{\infty}} \Sp. 
\end{equation}
Here, we stress the presentably symmetric monoidal structure on $\CMon$ is the one uniquely determined by the requirement that $\Free_{\einf} : \Spaces \rightarrow \CMon$ is symmetric monoidal. 
Furthermore, by the results of \cite[Section 6]{nikolausStable$infty$OperadsMultiplicative2016}, if $\m{C}$ is a semiadditive symmetric monoidal \category, the forgetful functor $\CMon \rightarrow \Spaces$ induces a symmetric monoidal equivalence of \categories 
\begin{equation}
\Fun^{\times}(\m{C},\CMon) \xrightarrow{\sim} \Fun^{\times}(\m{C},\Spaces),
\end{equation}
where the source and target both carry the Day convolution product and $\Fun^{\times}$ denotes the \category of product-preserving functors. 
Furthermore, if $\unit_{\m{C}}$ is the unit of $\m{C}$, then the unit of $\Fun^{\times}(\m{C},\CMon)$ is given by 
\begin{equation}
\Map_{\m{C}}(\unit_{\m{C}},-) : \m{C} \rightarrow \CMon,
\end{equation}
where the semiadditivity of $\m{C}$ provides a lift of the mapping spaces through $\CMon$. 
Finally, we note that there is a canonical equivalence of symmetric monoidal \categories 
\begin{equation}
\Fun^{\times}(\m{C},\CMon) \simeq \CMon\big(\Fun^{\times}(\m{C},\Spaces)\big). 
\end{equation} 
\end{recollection}

\begin{prop}\label{prop: gysin sheaves are group-complete}
The symmetric monoidal functors in \eqref{eqn: free sm functors} induce equivalences of symmetric monoidal \categories 
\begin{equation}
\Gys_{\varphi}(\Sm;\Spaces) \xrightarrow{\sim}\Gys_{\varphi}(\Sm;\CMon) 
\end{equation}
and
\begin{equation}
\Gys_{\varphi}^{\bbR}(\Sm;\Spaces)\xrightarrow{\sim}\Gys^{\bbR}_{\varphi}(\Sm;\CMongp).
\end{equation}
As a consequence, the $\einf$-algebra $\unit_{\varphi}^{\Spaces}=\rmH\Tr_{\varphi}\in \Gys_{\varphi}^{\bbR}(\Sm;\Spaces)$ admits a canonical lift to $\Gys^{\bbR}_{\varphi}(\Sm;\CMongp)$ such that $\rmH\Tr_{\varphi} \simeq \unit_{\varphi}^{\CMongp}$.
\end{prop}

\begin{proof}
    First, we show $\Gys_{\varphi}(\Sm;\Spaces)$ is semiadditive. 
    Due to the sheaf condition, it is clear $\Gys_{\varphi}(\Sm;\Spaces)$ is a full subcategory of 
    \begin{equation}
    \Fun^{\times}(\TrCor_{\varphi}^{\ex}(\Sm),\Spaces),
    \end{equation}
    which, by \cite[Corollary 2.4]{gepnerUniversalityMultiplicativeInfinite2015}, is a semiadditive \category. 
    It is easy to see that the inclusion
    \begin{equation}
    \Gys_{\varphi}(\Sm;\Spaces) \subset \Fun^{\times}(\TrCor_{\varphi}^{\ex}(\Sm),\Spaces)
    \end{equation}
    preserves products so that $\Gys_{\varphi}(\Sm;\Spaces)$ is semiadditive. 
    Following \cite[Proposition 6.5]{barwick} and \cite[Lemmas 3.5, 3.7]{barwick-glasman-shahii}, $\Fun^{\times}(\TrCor_{\varphi}^{\ex}(\Sm),\Spaces)$ is a symmetric monoidal localization of $\GysP_{\varphi}(\Sm;\Spaces)$, and one can show that $\Gys_{\varphi}(\Sm;\Spaces)$ is a further symmetric monoidal localization of $\Fun^{\times}(\TrCor_{\varphi}^{\ex}(\Sm),\Spaces)$. 
    Therefore, by Recollection~\ref{recollection: semiadditive Day convolution}, we obtain an equivalence of semiadditive symmetric monoidal \categories 
    \begin{equation}
    \Fun^{\times}\big(\TrCor_{\varphi}^{\ex}(\Sm),\Spaces\big) \simeq \Fun^{\times}\big(\TrCor_{\varphi}^{\ex}(\Sm),\CMon\big).
    \end{equation}
    This identification induces the desired symmetric monoidal equivalence between \categories of $\varphi$-Gysin sheaves. 
    Furthermore, the unit of $\Gys_{\varphi}(\Sm;\CMon)$ is given by $\Tr_{\varphi}$. 

    Because $\Gys_{\varphi}^{\bbR}(\Sm;\Spaces)$ is a symmetric monoidal localization of $\Gys_{\varphi}(\Sm;\Spaces)$, the free commutative monoid functor 
    \begin{equation}
    \Gys_{\varphi}^{\bbR}(\Sm;\Spaces) \rightarrow \Gys_{\varphi}^{\bbR}(\Sm;\CMon)
    \end{equation}
    induces an equivalence of symmetric monoidal \categories.  
    To complete the proof, it will be enough to show that $\rmH\Tr_{\varphi}(X,d_{X})$ is group-complete for all $(X,d_{X})$ as group-completion is a smashing localization.
    However, this was demonstrated in Remark~\ref{rem: cobordism relation}, which completes the proof. 
\end{proof}

\begin{rem}
    The argument above also shows that $\unit^{\Spaces_{\ast}}_{\varphi} \simeq \rmH\Tr_{\varphi}$. 
\end{rem}

\begin{cor}\label{cor:(de)loop}
    The canonical maps
    \begin{equation}
    \unit^{\Sp}_{\varphi} \rightarrow \Sigma_{+}^{\infty}\unit^{\Spaces}_{\varphi}\;\;{\rm and}\;\;
    \unit^{\Spaces}_{\varphi} \rightarrow \Omega^{\infty}\unit^{\Sp}_{\varphi}
    \end{equation}
    are equivalences of $\einf$-algebras. 
\end{cor}

\begin{proof}
    The first equivalence is purely formal. 
    For the second equivalence, note that the composition 
    \begin{equation}
    \CMongp \xrightarrow{\rmB^{\infty}} \Sp \xrightarrow{\Omega^{\infty}} \Spaces
    \end{equation}
    is equivalent to the forgetful functor and that $\rmB^{\infty}$ is symmetric monoidal and colimit-preserving. 
    These induce functors 
    \begin{equation}
    \Gys_{\varphi}^{\bbR}(\Sm;\CMongp) \rightarrow \Gys_{\varphi}^{\bbR}(\Sm;\Sp) \rightarrow \Gys_{\varphi}^{\bbR}(\Sm;\Spaces)
    \end{equation}
    whose composite, by Proposition~\ref{prop: gysin sheaves are group-complete}, is an equivalence of symmetric monoidal \categories. 
    Therefore, we have canonical equivalences 
    \begin{equation}
    \unit^{\Spaces}_{\varphi} \xrightarrow{\sim}\Omega^{\infty}\rmB^{\infty}\rmH\Tr_{\varphi} \simeq \Omega^{\infty}\unit^{\Sp}_{\varphi}
    \end{equation}
    as claimed.
\end{proof}

\begin{defin}\label{defin:kuranishitosmooth}
We say 
    \begin{equation}
    (\bbR^0,0)\xleftarrow{(f,0)} W \xrightarrow{(g,\alpha,\theta)}(Y,d_Y)\in\TrCor^{\rm ex}_\varphi(\Sm)
    \end{equation}
is \emph{$\varphi$-cobordant} to 
    \begin{equation}
    (\bbR^0,0)\xleftarrow{(f',0)} W' \xrightarrow{(g',\alpha',\theta')}(Y,d_Y)\in\TrCor^{\rm ex}_\varphi(\Sm)
    \end{equation}
if there exists an extended correspondence of the form 
    \begin{equation}
    (\bbR^0,0)\xleftarrow{(f,0)} W_{\rm bord} \xrightarrow{(\widetilde{g},\widetilde{\alpha},\widetilde{\theta})}(Y\times\bbR,d_Y)\in\TrCor^{\rm ex}_\varphi(\Sm)
    \end{equation}
together with the following property. There exists a commutative diagram in $\DSm$ of the form
    \begin{equation}
    \begin{tikzcd}
    W\arrow[r]\arrow[d,"g"]\arrow[dr, phantom, "\lrcorner", very near start] & W_{\rm bord}\arrow[d,"\widetilde{g}"] & W'\arrow[l]\arrow[d,"{g'}"]\arrow[dl, phantom, "\llcorner", very near start] \\
    Y\times\bbR^0\arrow[r] & Y\times\bbR & Y\times\bbR^0\arrow[l],
    \end{tikzcd}
    \end{equation}
where each square is a pullback and the bottom horizontal morphisms are the inclusion over 0 and 1, such that $\alpha$ resp. $\alpha'$ and $\theta$ resp. $\theta'$ are obtained via base change from $\widetilde{\alpha}$ resp. $\widetilde{\theta}$.
\end{defin}

The following result is the well-known fact that, in the language of the present article, (non-orbifold) global Kuranishi cobordism is equivalent to smooth cobordism.

\begin{lem}\label{lem:kuranishitosmooth}
The extended correspondence 
    \begin{equation}
    (\bbR^0,0)\xleftarrow{(f,0)} W \xrightarrow{(g,\alpha,\theta)}(Y,d_Y)\in\TrCor^{\rm ex}_\varphi(\Sm),
    \end{equation}
is $\varphi$-cobordant to some 
    \begin{equation}
    (\bbR^0,0)\xleftarrow{(f,0)} W' \xrightarrow{(g',\alpha',\theta')}(Y,d_Y)\in\TrCor^{\rm ex}_\varphi(\Sm)
    \end{equation}
such that $g':W'\to Y\in\Sm$.
\end{lem}

\begin{proof}
Recall, $g$ is locally of standard presentation. In particular, since all of our derived smooth manifolds are paracompact Hausdorff, we may use a standard partition of unity argument to show that $g$ is of the following form: 
    \begin{equation}
    \begin{tikzcd}
    W\arrow[r,"e"]\arrow[d]\arrow[dr, phantom, "\lrcorner", very near start] & \calT\arrow[r,"\pi"]\arrow[d,"s"] & Y \\
    \calT\arrow[r] & \calE, &
    \end{tikzcd}
    \end{equation}
where $\pi:\calT\to Y\in\Sm$, $\calE\to\calT$ is a smooth vector bundle, $s$ is a smooth section, and the bottom horizontal morphism is the 0-section embedding. Now, by perturbing $s$ to be transverse to the 0-section, we have constructed the desired commutative diagram in $\DSm$: 
    \begin{equation}
    \begin{tikzcd}
    W\arrow[r]\arrow[d,"g"]\arrow[dr, phantom, "\lrcorner", very near start] & W_{\rm bord}\arrow[d,"\widetilde{g}"] & W'\arrow[l]\arrow[d,"{g'}"]\arrow[dl, phantom, "\llcorner", very near start]\\
    Y\times\bbR^0\arrow[r] & Y\times\bbR & Y\times\bbR^0\arrow[l].
    \end{tikzcd}
    \end{equation}

It remains to show the claims about the $\varphi$-structure and transfer class; this is straightforward. First, $e$ is $\varphi$-structured by assumption, and this induces a $\varphi$-structure on $\pi$. In particular, when perturbing $s$, we induce a $\varphi$-structure on $W_{\rm bord}\to Y\times\bbR$ with the desired property. Moreover, since $\theta$ is assumed to be an extended transfer class, we see that $\theta$ is in fact equivalent to the adjoint of the composite given by the interchange map associated to $e$ followed by (the pullback of the adjoint of the) classical transfer associated to $\pi$. In particular, again when perturbing $s$, we induce an extended transfer class on $W_{\rm bord}\to Y\times\bbR$ with the desired property. This completes the proof.
\end{proof}

Recall, 
    \begin{equation}
    \Tr_\varphi(X\times\Delta^0_{\rm sm},d_X)=\TrCor_\varphi^\ex\big((\bbR^0,0),(X\times\Delta^0_{\rm sm},d_X)\big)
    \end{equation}
is naturally a pointed space with basepoint
    \begin{equation}
    (\bbR^0,0)\leftarrow\varnothing\rightarrow (X\times\Delta^0_{\rm sm},d_X);
    \end{equation}
in fact, every space we will encounter in the present subsection will naturally be a pointed space with basepoint the empty correspondence, so we will elide this fact from now on. Moreover, recall that we have shown $\unit^{\Spaces}_\varphi$ is the unit of a presentably symmetric monoidal structure on $\Gys^\bbR_\varphi(\Sm;\Spaces)$; in particular, $\pi_*\unit^{\Spaces}_\varphi(\bbR^0,0)$ is a graded commutative ring.

\begin{prop}\label{prop:spatialgysin1}
We have the following isomorphism of graded commutative rings: 
    \begin{equation}
    \pi_*\unit^{\Spaces}_\varphi(\bbR^0,0)\cong\Omega^{-*}_\varphi(*).
    \end{equation}
Moreover, let $X\in\Sm$ such that $X=\sqcup_i C_i$, where $C_i$ are the path components of $X$. We have the following isomorphism of abelian groups:
    \begin{equation}
    \pi_n\unit^{\Spaces}_\varphi(X,d_X)\cong\prod_\alpha\Omega^{-n+d_{C_{i}}}_\varphi(C_{i}),\;\;d_{C_{i}}\equiv d_X\big(\Sing(C_{i})\big)\in\bbZ,\;\;n\geq0.
    \end{equation}
In fact, the previous ismorphism is in fact an isomorphism of graded modules.\footnote{I.e., the action of $\pi_*\unit^{\Spaces}_\varphi(\bbR^0,0)\cong\Omega^{-\ast}_\varphi(*)$ on the graded abelian group $\pi_*\unit^{\Spaces}_\varphi(X,d_X)\cong \prod_{i}\Omega^{-*+d_{C_{i}}}_\varphi(C_{i})$, where $*$ must be at least 0, is identical; one can think of the latter graded abelian group as the ``truncated'' $\varphi$-structured geometric cobordism group of $C_{i}$.}
\end{prop}

\begin{proof}
It suffices to consider the case where $X$ is connected; the general case is a straightforward extension. Recall, we have the pointed simplicial space 
    \begin{equation}
    \mathcal{S}^\varphi_\bullet\equiv\Tr_\varphi(X\times\Delta^0_{\rm sm},d_X)\substack{\leftarrow \\ \leftarrow}\Tr(X\times\Delta^1_{\rm sm},d_X)\substack{\leftarrow \\[-1em] \leftarrow \\[-1em]\leftarrow}\cdots
    \end{equation}
which satisfies $\rlz{\mathcal{S}^\varphi_\bullet}=\unit^{\Spaces}_\varphi(X,d_X)$. We will compute $\pi_n\rlz{\mathcal{S}^\varphi_\bullet}$ via \cite[Proposition 8]{lurienotes}.\footnote{Technically, to apply \cite[Proposition 8]{lurienotes}, we must check that $\mathcal{S}^\varphi_\bullet$ satisfies the Kan condition; this is relatively straightforward as the Kan condition, in this case, amounts to the fact that cobordisms can be (1) reversed and (2) glued.} As $\mathcal{S}^\varphi_\bullet$ is a (pointed) simplicial space, it determines a functor from ${\rm sSet}$ to $\Spaces$ via $\mathcal{S}^\varphi_\bullet(K)\equiv\Map_{{\rm s}\Spaces}(K,X)$, where $K$ is given the discrete topology levelwise.

We see that $\mathcal{S}^\varphi_\bullet(\partial\Delta^n)=\Tr_\varphi(X\times\partial\Delta^n_{\rm sm},d_X)$, where $\partial\Delta^n_{\rm sm}$ denotes the smoothing of $\bbR^n$ minus the standard $n$-simplex. Let $\mathcal{S}^\varphi_{n,\partial}$ be the homotopy fiber of the map $\mathcal{S}^\varphi_n\to\mathcal{S}^\varphi_\bullet(\partial\Delta^n)$ lying over the empty correspondence; this space is the subspace of $\Tr_\varphi(X\times\Delta^n_{\rm sm},d_X)$ consisting of correspondences 
    \begin{equation}
    (\bbR^0,0)\xleftarrow{(f,0)}W\xrightarrow{(g,\alpha,\theta)}(X\times\Delta^n_{\rm sm},d_X)
    \end{equation}
such that $g$ is of constant rank $-g^*d_X$ and the fibers of $g$ are empty over $X\times\partial\Delta^n_{\rm sm}$.

Meanwhile, let $\Lambda\subset\partial\Delta^{n+1}$ be the subset obtained by removing the interiors of two faces. We see that $\mathcal{S}^\varphi_\bullet(\Lambda)=\Tr_\varphi(X\times\Lambda_{\rm sm},d_X)$, where $\Lambda_{\rm sm}$ denotes the smoothing of $\bbR^n$ minus $\Lambda$. Let $\mathcal{S}^\varphi_{\partial,\Lambda}$ be the homotopy fiber of the map $\mathcal{S}^\varphi_\bullet(\partial\Delta^{n+1})\to\mathcal{S}^\varphi_\bullet(\Lambda)$ lying over the empty correspondence; this space is the subspace of $\Tr_\varphi(X\times\partial\Delta^n_{\rm sm},d_X)$ consisting of correspondences 
    \begin{equation}
    (\bbR^0,0)\xleftarrow{(f,0)}W\xrightarrow{(g,\alpha,\theta)}(X\times\partial\Delta^n_{\rm sm},d_X)
    \end{equation}
such that $g$ is of constant rank $-g^*d_X$ and the fibers of $g$ are empty over $X\times\Lambda_{\rm sm}$. 

Observe, we have a canonical equivalence:
    \begin{equation}
    \mathcal{S}^\varphi_{\partial,\Lambda}\xrightarrow{\sim}\mathcal{S}^\varphi_{n,\partial}\times\mathcal{S}^\varphi_{n,\partial}.
    \end{equation}
Now, \cite[Proposition 8]{lurienotes} says $\pi_n\rlz{\mathcal{S}^\varphi_\bullet}$ is precisely $\pi_0\mathcal{S}^\varphi_{n,\partial}$ modulo the following equivalence relation: two elements $\omega,\omega'\in\pi_0\mathcal{S}^\varphi_{n,\partial}$ are identified if the element $\overline{\omega}\in\pi_0\mathcal{S}^\varphi_{\partial,\Lambda}$, corresponding to $(\omega,\omega')$, can be lifted to $\pi_0\mathcal{S}^\varphi_{n+1,\Lambda}$, where the latter space is the homotopy fiber of the map $\mathcal{S}^\varphi_{n+1}\to\mathcal{S}^\varphi_\bullet(\Lambda)$ lying over the empty correspondence. More concretely, $\pi_n\rlz{\mathcal{S}^\varphi_\bullet}$ consists of correspondences 
    \begin{equation}
    (\bbR^0,0)\xleftarrow{(f,0)}W\xrightarrow{(g,\alpha,\theta)}(X\times\Delta^n_{\rm sm},d_X),
    \end{equation}
such that $g$ is of constant rank $-g^*d_X$ and the fibers of $g$ are empty over $X\times\partial\Delta^n_{\rm sm}$, and we identify two such correspondences, denoted $W$ and $W'$, if there exists a correspondence 
    \begin{equation}
    (\bbR^0,0)\xleftarrow{(f,0)}\overline{W}\xrightarrow{(\overline{f},\overline{\alpha},\overline{\theta})}(X\times\Delta^{n+1}_{\rm sm},d_X),
    \end{equation}
such that the fibers of $g$ are empty over $X\times\Lambda_{\rm sm}$, which restricts to $W$ and $W'$ over the appropriate subset of $X\times\partial\Delta^{n+1}_{\rm sm}$. Utilizing an appropriate modification of Definition \ref{defin:kuranishitosmooth} and the subsequent Lemma \ref{lem:kuranishitosmooth}, we have shown that we have a bijection of sets
    \begin{equation}
    \pi_n\unit^{\Spaces}_\varphi(X,d_X)\cong\Omega^{-n+d_X}_\varphi(X).
    \end{equation}

It remains to identify all of the claimed algebraic structure. We begin with the additive structure. Now, the group structure on $\pi_n\unit^{\Spaces}_\varphi(X,d_X)$ is explicitly induced by the semiadditive structure on $\TrCor^{\rm ex}_\varphi(\Sm)$. Consider two
    \begin{align}
    (\bbR^0,0)\xleftarrow{(f,0)}W\xrightarrow{(g,\alpha,\theta)}(X\times\Delta^n_{\rm sm},d_X),&\\
    (\bbR^0,0)\xleftarrow{(f,0)}W'\xrightarrow{(g',\alpha',\theta')}(X\times\Delta^n_{\rm sm},d_X),&
    \end{align}
then their sum is the correspondence obtained via the composition
    \begin{equation}
    \begin{tikzcd}[column sep=tiny]
    W\sqcup W'\arrow[d,"{(f\sqcup f',0)}",swap]\arrow[dr,"{(g\sqcup g',\alpha\sqcup\alpha',\theta\sqcup\theta')}"] & & \big((X\times\Delta^n_{\rm sm})\sqcup(X\times\Delta^n_{\rm sm}),d_X\big)\arrow[dl,equals]\arrow[d] \\
    (\bbR^0,0) & \big((X\times\Delta^n_{\rm sm})\sqcup(X\times\Delta^n_{\rm sm}),d_X\big) & (X\times\Delta^n_{\rm sm},d_X),
    \end{tikzcd}
    \end{equation}
where the map $(X\times\Delta^n_{\rm sm})\sqcup(X\times\Delta^n_{\rm sm})\to X\times\Delta^n_{\rm sm}$ is the obvious "fold" map. But this is precisely the group structure on $\Omega^{-n+d_X}_\varphi(X)$ given by disjoint union.

Next, we continue with the ring structure in the case $(X,d_X)=(\bbR^0,0)$. Now, the ring structure on $\pi_*\unit^{\Spaces}_\varphi(\bbR^0,0)$ is explicitly induced by the symmetric monoidal structure on $\TrCor^{\rm ex}_\varphi(\Sm)$. Consider two
    \begin{align}
    (\bbR^0,0)\xleftarrow{(f,0)}W\xrightarrow{(g,\alpha,\theta)}(\bbR^0\times\Delta^n_{\rm sm},),&\\
    (\bbR^0,0)\xleftarrow{(f,0)}W'\xrightarrow{(g',\alpha',\theta')}(\bbR^0\times\Delta^n_{\rm sm},0),&
    \end{align}
then their product is the correspondence obtained via the composition
    \begin{equation}
    \begin{tikzcd}
    W\times W'\arrow[d,"{(f\times f',0)}",swap]\arrow[dr,"{(g\times g',\alpha\times\alpha',\theta\times\theta')}"] & & (\bbR^0\times\Delta^{n+m}_{\rm sm},0)\arrow[dl,"\Delta\times\id",swap]\arrow[d,equals] \\
    (\bbR^0,0) & \big((\bbR^0\times\Delta^n_{\rm sm})\times(\bbR^0\times\Delta^m_{\rm sm}),0\big) & (\bbR^0\times\Delta^{n+m}_{\rm sm},0).
    \end{tikzcd}
    \end{equation}
But this is precisely the ring structure on $\Omega^{-*}_\varphi(*)$ given by Cartesian product.

Finally, we end with the module structure. Consider two
    \begin{align}
    (\bbR^0,0)\xleftarrow{(f,0)}W\xrightarrow{(g,\alpha,\theta)}(\bbR^0\times\Delta^n_{\rm sm},0),& \\(\bbR^0,0)\xleftarrow{(f',0)}W'\xrightarrow{(g',\alpha',\theta')}(X\times\Delta^m_{\rm sm},d_X);&
    \end{align}
we abuse notation and denote by 
    \begin{equation}
    (\bbR^0,0)\xleftarrow{(f,0)}W\xrightarrow{(g,\alpha,\theta)}(X\times\Delta^n_{\rm sm},0)
    \end{equation}
the correspondence obtained via the composition 
    \begin{equation}
    \begin{tikzcd}[column sep=large]
    W\arrow[d,"{(f,0)}",swap]\arrow[dr,"{(g,\alpha,\theta)}"] & & (X\times\Delta^n_{\rm sm},0)\arrow[dl,"*\times\id",swap]\arrow[d,equals] \\
    (\bbR^0,0) & (\bbR^0\times\Delta^n_{\rm sm},0) & (X\times\Delta^n_{\rm sm},0).
    \end{tikzcd}
    \end{equation}
The module action is the correspondence obtained via the composition
    \begin{equation}
    \begin{tikzcd}
    W\times W'\arrow[d,"{(f\times f',0)}",swap]\arrow[dr,"{(g\times g',\alpha\times\alpha',\theta\times\theta')}"] & & (X\times\Delta^{n+m}_{\rm sm},d_X)\arrow[dl,"\Delta\times\id",swap]\arrow[d,equals] \\
    (\bbR^0,0) & \big((X\times\Delta^n_{\rm sm})\times(X\times\Delta^m_{\rm sm}),d_X\big) & (X\times\Delta^{n+m}_{\rm sm},d_X).
    \end{tikzcd}
    \end{equation}
But this is precisely the action of $\Omega^{-\ast}_\varphi(\bbR^{0})$ on $\Omega^{-*+d_X}_\varphi(X)$, $*$ at least 0, given by Cartesian product. It is straightforward to check this behaves well with the additive structure; the proposition follows.
\end{proof}

Fortunately, although $\unit^{\Spaces}_\varphi(X,d)$, for an arbitrary $d\in\bbZ$, does not see all of the $\varphi$-structured geometric cobordism groups of $X$, we may make a careful choice of shift so that we can.

\begin{cor}
Let $X\in\Sm$ be path connected. We have the following isomorphism of graded modules:
    \begin{equation}
    \pi_*\unit^{\Spaces}_\varphi(X,\dim X)\cong\Omega^{-*+\dim X}_\varphi(X).
    \end{equation}
\end{cor}

Meanwhile, with $\unit^{\Sp}_\varphi$, we can see more.

\begin{prop}
Let $X\in\Sm$ be connected. We have the following isomorphism of graded modules: 
    \begin{equation}
    \pi_*\unit^{\Sp}_\varphi(X,d)\cong\Omega^{-*+d}_\varphi(X),\;\;d\in\bbZ.
    \end{equation}
In fact, if $(X,d)=(\bbR^0,0)$, then we have the following isomorphism of graded commutative rings:
    \begin{equation}
    \pi_*\unit^{\Sp}_\varphi(\bbR^0,0)\cong\Omega^{-*}_\varphi(*).
    \end{equation}
\end{prop}

\begin{proof}
First, we have the following commutative diagram: 
    \begin{equation}
    \begin{tikzcd}[row sep=small,column sep=small]
    \Gys^\bbR_\varphi(\Sm;\Spaces)\arrow[d]\arrow[r,"\sim"] & \Gys^\bbR_\varphi(\Sm;\CMongp)\arrow[d]\arrow[r] & \Gys^\bbR_\varphi(\Sm;\Sp^{\geq0})\arrow[d]\arrow[r] & \Gys^\bbR_\varphi(\Sm;\Sp)\arrow[d] \\
    \Shv_\bbR(\Sm;\Spaces)\arrow[d,"\sim"]\arrow[r] & \Shv_\bbR(\Sm;\CMongp)\arrow[d,"\sim"]\arrow[r] & \Shv_\bbR(\Sm;\Sp^{\geq0})\arrow[d,"\sim"]\arrow[r] & \Shv_\bbR(\Sm;\Sp)\arrow[d,"\sim"] \\
    \Spaces\arrow[r] & \CMongp\arrow[r] & \Sp^{\geq0}\arrow[r] & \Sp,
    \end{tikzcd}
    \end{equation}
where the equivalence between the second row and the third row is Dugger's theorem. By Proposition \ref{prop: gysin sheaves are group-complete}, for any $d\in\bbZ$, we have that the delooping of
    \begin{equation}
    \unit^{\Spaces}_\varphi(X,d)\simeq\Map_{\Spaces}\big(\Sing(X),\unit^{\Spaces}_\varphi(\bbR^0,0)\big)
    \end{equation}
is given by 
    \begin{equation}
    \unit^{\Sp}_\varphi(X,d)\simeq F\big(\Sigma^\infty_+\Sing(X),\unit^{\Sp}_\varphi(\bbR^0,0)\big);
    \end{equation}
i.e., 
    \begin{equation}
    \unit^{\Spaces}_\varphi(X,d)\simeq\Omega^\infty\unit^{\Sp}_\varphi(X,d).
    \end{equation}
Now, we have that
    \begin{equation}
    \pi_n\unit^{\Sp}_\varphi(X,d)\cong\pi_n\unit^{\Spaces}_\varphi(X,d),\;\;n\geq0;
    \end{equation}
meanwhile, using Subsection \ref{subsec:shiftoperators}, we obtain the negative homotopy groups: 
    \begin{equation}
    \pi_{-n}\unit^{\Sp}_\varphi(X,d)\cong\pi_0\unit^{\Sp}_\varphi(X,d+n)\cong\pi_0\unit^{\Spaces}_\varphi(X,d+n),\;\;n\geq0.
    \end{equation}
The result follows by Proposition \ref{prop:spatialgysin1}.
\end{proof}

\section{Homotopy coherent Pontryagin-Thom}\label{sec: coherent PT}
\subsection{Quillen cohomology theories}
The following definition essentially appears in \cite{quillenElementaryProofsResults1971}; hence, we find it appropriate that such objects carry Quillen's namesake. We denote by ${\rm GrAb}$ the category of graded abelian groups.

\begin{defin}\label{defin:qct}
Let $h^*:\Sm^\op\to{\rm GrAb}$ be a functor and, given $f:X\to Y\in\Sm$, denote by $f^*:h^*(Y)\to h^*(X)\in{\rm GrAb}$ the induced map. A \emph{$\varphi$-structured Quillen cohomology theory} is a functor $h^*$ of the aforementioned form together with, for any proper $\varphi$-structured map $g:W\to Y\in\Sm$, a map
    \begin{equation}
    g_*:h^{*+\rk g}(W)\to h^*(Y),
    \end{equation}
referred to as a \emph{Gysin} map, satisfying the following properties. 
\begin{enumerate}
\item If $f_0,f_1:X\to Y$ are homotopic, then $f_0^*=f_1^*$.
\item If $g:W\to W'$ and $g':W'\to Y$ are both proper $\varphi$-structured maps and if $g'\circ g$ has the induced $\varphi$-structure, then $(g'\circ g)_*=g'_*\circ g_*$.
\item If 
    \begin{equation}
    \begin{tikzcd}
    \overline{W}\equiv W\times_Y\overline{Y}\arrow[d,"\overline{f}",swap]\arrow[r,"\overline{g}"]\arrow[dr, phantom, "\lrcorner", very near start] & \overline{Y}\arrow[d,"f"] \\
    W\arrow[r,"g",swap] & Y
    \end{tikzcd}
    \end{equation}
is a transversal pullback in $\Sm$, where $g$ is a proper $\varphi$-structured map such that $\overline{g}$ has the induced $\varphi$-structure, then $g^*f_*=\overline{f}_*\overline{g}^*$.
\end{enumerate}
\end{defin}

\begin{rem}
We may define a morphism of $\varphi$-structured Quillen cohomology theories in the obvious way: it is a natural transformation of the underlying functors compatible with the Gysin maps.
\end{rem}

\begin{example}
The following are quick examples of $\varphi$-structured Quillen cohomology theories.
\begin{enumerate}
\item The functor $\Omega^*_\varphi:\Sm^\op\to{\rm GrAb}$ which sends a smooth manifold to its usual $\varphi$-structured geometric cobordism ring.
\item The functor $\Omega^{*,{\rm ho}}_\varphi\equiv\pi_{-*}F\big(\Sigma^\infty_+\Sing(-),\rmM\varphi\big):\Sm^\op\to{\rm GrAb}$ given by the usual $\varphi$-structured homotopical cobordism ring.
\end{enumerate}
\end{example}

\begin{rem}
As already mentioned, the usual Pontryagin-Thom isomorphism says we have the following equivalence of Quillen cohomology theories: 
    \begin{equation}
    \Omega^*_\varphi\xrightarrow{\sim}\Omega^{*,{\rm ho}}_\varphi;
    \end{equation}
this equivalence is given by the celebrated Pontryagin-Thom construction.
\end{rem}

The following result shows $\varphi$-structured homotopical cobordism is initial among $\varphi$-structured Quillen cohomology theories.

\begin{thm}[Proposition 1.10 in \cite{quillenElementaryProofsResults1971}]
Let $h^*$ be a $\varphi$-structured Quillen cohomology theory and choose $a\in h^0(*)$. There is a unique map of $\varphi$-structured Quillen cohomology theories $\theta:\Omega^{*,{\rm ho}}_\varphi\to h^*$ such that $\theta(1)=a$, where $1\in\Omega^{0,{\rm ho}}_\varphi(*)$ is the cobordism class of the identity map $*\to *$ with the $\varphi$-structure corresponding to the positive orientation.
\end{thm}

\begin{rem}
Of course, \cite[Proposition 1.10]{quillenElementaryProofsResults1971} is only written to pertain to complex cobordism, but the proof straightforwardly generalizes.
\end{rem}

We will now explain how, given a spectral Gysin sheaf, we may construct a Quillen cohomology theory.

\begin{construction}\label{construction:preqct}
Let $\m{F}\in\Gys^\bbR_\varphi(\Sm;\Sp)$; we may construct a Quillen cohomology theory, denoted $\mathscr{Q}^*_{\m{F}}$, as follows. First, given $X\in\Sm$, we define
    \begin{equation}
    \mathscr{Q}^*_{\m{F}}(X)\equiv\pi_{-*}\m{F}(X,0).
    \end{equation}
Second, given $f:X\to Y\in\Sm$, we define 
    \begin{equation}
    f^*:\mathscr{Q}^*_{\m{F}}(Y)\to\mathscr{Q}^*_{\m{F}}(X)
    \end{equation}
via
    \begin{equation}
    \pi_{-*}\m{F}\big((Y,0)\xleftarrow{(f,0)}X\xrightarrow{\id}(X,0)\big):\pi_{-*}\m{F}(Y,0)\to\pi_{-*}\m{F}(X,0).
    \end{equation}
Finally, given $g:X\to Y\in\Sm$ with $\varphi$-structure $\alpha$, we define 
    \begin{equation}
    g_{!}:\mathscr{Q}^{*+\rk g}_{\m{F}}(X)\to\mathscr{Q}^*_{\m{F}}(Y)
    \end{equation}
via 
     \begin{multline}
    \pi_{-*}\m{F}\big((X,-\rk g)\xleftarrow{\id}X\xrightarrow{(g,\alpha,\theta_\alpha)}(Y,0)\big):\pi_{-*+\rk g}\m{F}(X,0)\cong \\
    \pi_{-*}\m{F}(X,-\rk g)\to\pi_{-*}\m{F}(Y,0);
    \end{multline}
here, we use Subsection \ref{subsec:shiftoperators}. Moreover, it is straightforward to see that 
    \begin{equation}
    \theta:\m{F}\to\m{G}\in\Gys^\bbR_\varphi(\Sm;\Sp)\rightsquigarrow\mathscr{Q}^*_\theta:\mathscr{Q}^*_\m{F}\to\mathscr{Q}^*_\m{G}.
    \end{equation}
\end{construction}

\subsection{Global sections and the categorified Thom isomorphism}
By definition, the realization functor $\DSm \rightarrow \Top$ necessarily factors through $\LCH \subset \Top$. 
This induces a functor $\Cor(\DSm;\calP) \rightarrow \Cor(\LCH;\calP)$ and therefore a 6-functor formalism 
\begin{equation}
\Shv(-;R) : \Cor(\DSm;\calP) \rightarrow \bigcat, \ \ \ X \mapsto \Shv(X;R).
\end{equation}
Unstraightening, we obtain a symmetric monoidal functor 
\begin{equation}
\Un\big(\Shv(-;R)\big) \equiv \int^{\rm cc}_{\Cor(\DSm;\calP)} \Shv(-;R) \rightarrow \Cor(\DSm;\calP). 
\end{equation}
For each $X \in \DSm$, the formation of global sections determines a functor $\Gamma(X;-) :\Shv(X;R) \rightarrow \Mod_{R}$ which is lax symmetric monoidal.
Given $f: X \rightarrow Y\in\DSm$, there is a natural transformation 
\begin{equation}
\Gamma(Y;-) \rightarrow \Gamma\big(X;f^{\ast}(-)\big).
\end{equation}
In the case where $f$ is proper, there is an additional natural equivalence 
\begin{equation}
\Gamma(X;-) \simeq \Gamma\big(Y;f_{\ast}(-)\big) \simeq \Gamma\big(Y;f_{!}(-)\big). 
\end{equation}
As we will see, these two kinds of natural transformations can be packaged into a coherent global sections functor. 

\begin{prop}\label{prop: global sections (unstructured)}
    Let $R$ be an $\einf$-ring spectrum. 
    There is a lax symmetric monoidal functor 
    \begin{equation}
    \Gamma : \Un\big(\Shv(-;R)\big) \rightarrow \Mod_{R}
    \end{equation}
    given by sending a pair $(X,\m{F})$ to $\Gamma(X;\m{F})$. 
\end{prop}

\begin{proof}
    It suffices to produce a lax symmetric monoidal functor 
    \begin{equation}
    \Un(\Shv(-;R)) = \int^{\rm cc}_{\Cor(\DSm;\calP)} \Shv(-;R) \rightarrow \Mod_{R}, 
    \end{equation}
    which is tantamount to producing a lax symmetric monoidal functor 
    \begin{equation}
    \Un(\Shv(-;R)) \rightarrow \Cor(\DSm;\calP) \times \Mod_{R},
    \end{equation}
    where the target is the coCartesian unstraightening of the constant functor on $\Cor(\DSm;\calP)$ with value $\Mod_{R}$; such a functor is the same as a lax natural transformation
    \begin{equation}
    \Shv(-;R) \rightarrow \const(\Mod_{R})
    \end{equation}
    of functors $\Cor(\DSm;\calP) \rightarrow \bigcat$ which is additionally lax symmetric monoidal. 
    To produce such a lax natural transformation we appeal to \cite[Theorem 6.13]{cnossenUniversalitySpan2categories2026}; we now explain how this result applies to our situation. 
    Following the notation of \textit{loc cit.}, we have $\m{C} = \DSm$, $I = \mathrm{iso}$, $P = \calP$, and $E = \calP$; because $I = \mathrm{iso}$ it is clear these choices satisfy \cite[Convention 3.1]{cnossenUniversalitySpan2categories2026}.
    By \cite[Proposition 2.32]{volpe}, the functor
    \begin{equation}
    \Shv(-;R) : \DSm^{\op} \rightarrow \PrSt
    \end{equation}
    is symmetric monoidal, and, as $\DSm^{\op}$ has an initial object, there is a canonical symmetric monoidal natural transformation of functors $\DSm^{\op} \rightarrow \PrSt$,
    \begin{equation}
    \const(\Mod_{R}) \rightarrow \Shv(-;R),
    \end{equation}
    given by sending an $R$-module $M$ to the constant sheaf $p_{X}^{\ast}M \in \Shv(X;R)$. 
    By forming the fiberwise right adjoint of the transformation above,
    we obtain a lax natural transformation of functors from $\DSm^{\op}$ to $\bbCat$,
    \begin{equation}
    \Gamma: \Shv(-;R) \rightarrow \const(\Mod_{R}),
    \end{equation}
    which is lax symmetric monoidal. 
    We are now in the situation of \cite[Theorem 6.13]{cnossenUniversalitySpan2categories2026}. 
    Because $I = \mathrm{iso}$, we only need to check that, for any proper map $p : X \rightarrow Y$, the square 
    \begin{equation}\begin{tikzcd}
	{\Shv(Y;R)} & {\Shv(X;R)} \\
	{\Mod_{R}} & {\Mod_{R}}
	\arrow["{p^{\ast}}", from=1-1, to=1-2]
	\arrow["{\Gamma(Y;-)}"', from=1-1, to=2-1]
	\arrow["{\Gamma(X;-)}", from=1-2, to=2-2]
	\arrow[Rightarrow, from=2-1, to=1-2]
	\arrow["{\mathrm{id}}"', from=2-1, to=2-2]
    \end{tikzcd}\end{equation}
    is horizontally right adjointable; this follows from the equivalence 
    \begin{equation}
    \Gamma(X;-) \simeq \Gamma\big(Y;p_{\ast}(-)\big) \simeq \Gamma\big(Y;p_{!}(-)\big)
    \end{equation}
    above. 
    Therefore, $\Gamma$ extends to a lax symmetric monoidal lax natural transformation of $(\infty,2)$-categories from $\bCor(\DSm;\calP)$ to $\bbCat$:
    \begin{equation}
    \Shv(-;R) \rightarrow \const(\Mod_{R}).
    \end{equation}
    Restricting to the invertible $2$-cells in $\bCor(\DSm;\calP)$, we obtain a lax symmetric monoidal lax natural transformation of functors from $\Cor(\DSm;\calP) \rightarrow \bigcat$. 
    Unstraightening this transformation, we obtain the desired lax symmetric monoidal functor 
    \begin{equation}
    \Un\big(\Shv(-;R)\big) \rightarrow \Cor(\DSm;\calP) \times \Mod_{R}. 
    \end{equation} 
\end{proof}

\begin{rem}
Note that $\Gamma|_{\Sm^{\op}} \simeq F\big(\Sigma_{+}^{\infty}\Sing(-),\rmM\varphi\big)$.
\end{rem}

We use global sections to produce a lax symmetric monoidal functor 
\begin{equation}
\Gamma_{\varphi}: \TrCor_{\varphi}^{\ex}(\Sm) \rightarrow \Sp.
\end{equation}
To accomplish this, we categorify the Thom isomorphism.
\begin{construction}\label{construction: categorified Thom isomorphism}
The $\einf$-$\rmM\varphi$-trivialization from Remark~\ref{remark: universal phi orientation} induces the following canonical commutative diagram in $\CAlg\big(\Fun(\DSm^{\op},\bigcat)\big)$:
\begin{equation}
\begin{tikzcd}
	\frakB & {\Shv(-;\rmM\varphi)} \\
	\ast & {\Shv(-;\rmM\varphi)},
	\arrow["{\Th_{\varphi}}", from=1-1, to=1-2]
	\arrow[from=1-1, to=2-1]
	\arrow["{\mathrm{id}}", from=1-2, to=2-2]
	\arrow["{\mathrm{can}_{\varphi}}", between={0.1}{0.9}, Rightarrow, 2tail reversed, from=2-1, to=1-2]
	\arrow[from=2-1, to=2-2]
\end{tikzcd}
\end{equation}
where we have written $\ast$ for the trivial structure functor. 
This induces a commutative diagram in $\Fun\big(\DSm^{\op},\CAlg\big(\Mod(\bigcat)\big)\big)$:
\begin{equation}
\begin{tikzcd}
	{(\frakB,\Shv(-;\rmM\varphi))} & {(\ast,\Shv(-;\rmM\varphi))} \\
	{(\frakB,\ast)} & {(\ast,\ast)}.
	\arrow[from=1-1, to=1-2]
	\arrow[from=1-1, to=2-1]
	\arrow[from=1-2, to=2-2]
	\arrow[from=2-1, to=2-2]
\end{tikzcd}
\end{equation}
After applying $-\quot -$ and forming the coCartesian unstraightening, this yields a commutative square of symmetric monoidal \categories:
\begin{equation}\label{equation: Thom iso}
\begin{tikzcd}
    \int^{\rm cc}_{\DSm^{\op}} \Shv(-;\rmM\varphi)\quot \frakB \arrow[r] \arrow[d] & \int^{\rm cc}_{\DSm^{\op}} \Shv(-;\rmM\varphi) \arrow[d] \\
    \DSm^{\op}_{\frakB} \arrow[r,"\pi_{\frakB}"'] & \DSm^{\op}.  
\end{tikzcd}
\end{equation}
Passing to the fiber over $X \in \DSm^{\op}$, we have a commutative square
\begin{equation}
\begin{tikzcd}
	{\Shv(X;\rmM\varphi)\quot\frakB(X)} & {\Shv(X;\rmM\varphi)} \\
	{\ast\quot\frakB(X)} & {\{X\}}.
	\arrow[from=1-1, to=1-2]
	\arrow[from=1-1, to=2-1]
	\arrow[from=1-2, to=2-2]
	\arrow[from=2-1, to=2-2]
\end{tikzcd}
\end{equation}
As $\frakB(X)$ acts trivially on $\Shv(X;\rmM\varphi)$, the induced map 
\begin{equation}
    \Shv(X;\rmM\varphi)\quot\frakB(X) \rightarrow \ast\quot\frakB(X)\times \Shv(X;\rmM\varphi)
\end{equation}
is an equivalence, which, combined with \cite[Corollary 2.4.4.4]{lurieHigherToposTheory2009}, shows that \eqref{equation: Thom iso} is a pullback square. 
Therefore, by straightening, we have an equivalence of lax symmetric monoidal functors 
\begin{equation}
    \mathrm{can}_{\varphi}: \Shv(-;\rmM\varphi)_{\frakB} \simeq \pi_{\frakB}^{\ast}\Shv(-;\rmM\varphi).
\end{equation}
\end{construction}

\begin{prop}[Categorified Thom isomorphism]
    Let $\pi_{\varphi} : \Cor_{\varphi}(\DSm) \rightarrow \Cor(\DSm;\calP)$ denote the forgetful functor. 
    There is a canonical equivalence of 6-functor formalisms 
    \begin{equation}
    \m{D}_{\varphi} \xrightarrow{\sim} \pi_{\varphi}^{\ast} \Shv(-;\rmM\varphi);
    \end{equation}
    equivalently, there is a pullback square of symmetric monoidal \categories:
    \begin{equation}
    \begin{tikzcd}
	{\int^{\rm cc}_{\Cor_{\varphi}(\DSm)} \m{D}_{\varphi}}\arrow[dr, phantom, "\lrcorner", very near start] & {\int_{\Cor(\DSm;\calP)}^{\rm cc}\Shv(-;\rmM\varphi)} \\
	{\Cor_{\varphi}(\DSm)} & {\Cor(\DSm;\calP)}.
	\arrow[from=1-1, to=1-2]
	\arrow[from=1-1, to=2-1]
	\arrow[from=1-2, to=2-2]
	\arrow["{\pi_{\varphi}}"', from=2-1, to=2-2]
    \end{tikzcd}
    \end{equation}
\end{prop}

\begin{proof}
    By the construction of the 6-functor formalism in Proposition~\ref{prop:twisted3ff}, the equivalence 
    \begin{equation}
    \mathrm{can}_{\varphi}: \Shv(-;\rmM\varphi)_{\frakB} \simeq \pi_{\frakB}^{\ast}\Shv(-;\rmM\varphi)
    \end{equation}
    prolongs to an equivalence of 6-functor formalisms 
    \begin{equation}
    \Shv_{\varphi} \simeq (\pi^{\rm cor}_{\frakB})^{\ast}\Shv(-;\rmM\varphi),
    \end{equation}
    where $\pi_{\frakB}^{\rm cor} : \Cor(\DSm_{\frakB};\calP_{\frakB}) \rightarrow \Cor(\DSm;\calP)$. 
    Restricting this equivalence along $\Cor_{\varphi}(\DSm) \subset \Cor(\DSm_{\frakB};\calP_{\frakB})$ yields the claim.
\end{proof}

\begin{defin}
    We define \textit{$\varphi$-structured homotopical cobordism} to be the lax symmetric monoidal functor 
    \begin{equation}
    \Gamma_{\varphi}: \TrCor_{\varphi}^{\rm ex}(\Sm) \subset \Un(\m{D}_{\varphi}) \rightarrow \Un\big(\Shv(-;\rmM\varphi)\big) \xrightarrow{\Gamma} \Sp.
    \end{equation}
\end{defin}

\begin{rem}
    By construction, for each $(X,d_{X})\in\TrCor^{\rm ex}_\varphi(\Sm)$, we have an equivalence 
    \begin{equation}
    \Gamma_{\varphi}(X,d_{X}) = \Gamma(X;\rmM\varphi_{X}\{d_{X}\}) \simeq \prod_{i \in I} F\big(\Sigma_{+}^{\infty}\Sing(X_{i}),\Sigma^{d_{i}}\rmM\varphi\big),
    \end{equation}  
    where $\sqcup_{i} X_{i} = X$ is the decomposition of $X$ into its path components.
\end{rem}

\begin{thm}
    $\varphi$-structured homotopical cobordism admits the structure of an $\einf$-algebra in the category of spectral $\varphi$-Gysin sheaves. 
    Therefore, there is an essentially unique map of $\einf$-algebras:
    \begin{equation}
    \mathrm{PT}_{\varphi} : \unit_{\varphi} \rightarrow \Gamma_{\varphi}\in\Gys_{\varphi}^{\bbR}(\Sm;\Sp).
    \end{equation}
\end{thm}

\begin{proof}
    The functor $\Gamma_{\varphi}$ is a sheaf essentially by construction, and the fact it is homotopy invariant follows from the equivalence in the preceding remark. 
    As the localization 
    \begin{equation}
    L_{\rm hi}L_{\bbR} : \GysP_{\varphi}(\Sm;\Sp) \rightarrow \Gys_{\varphi}^{\bbR}(\Sm;\Sp)  
    \end{equation}
    is lax symmetric monoidal, $\Gamma_{\varphi}$ admits the structure of an $\einf$-algebra in $\Gys_{\varphi}^{\bbR}(\Sm;\Sp)$; this determines the essentially unique map $\mathrm{PT}_{\varphi}$, as claimed. 
\end{proof}

\subsection{Identifying universal Thom spectra}

We are now ready to prove Theorem \ref{thm:main}; it essentially amounts to a proof of the following result.

\begin{thm}\label{thm: PT is an equivalence}
    The map of $\einf$-algebras,
    \begin{equation}
        \mathrm{PT}_{\varphi} : \unit_{\varphi} \rightarrow \Gamma_{\varphi},
    \end{equation}
    is an equivalence of $\varphi$-Gysin sheaves.
\end{thm}

\begin{proof}
It suffices to show the induced map of $\bbE_\infty$-rings
    \begin{equation}
        \mathrm{PT}_{\varphi}(\bbR^0,0) : \unit_{\varphi}(\bbR^0,0) \rightarrow \Gamma_{\varphi}(\bbR^0,0)=\rmM\varphi
    \end{equation}
is an equivalence, i.e., is an isomorphism on homotopy groups. Let us consider the induced map on $\varphi$-structured Quillen cohomology theories induced by $\mathrm{PT}_{\varphi}$ from Construction \ref{construction:preqct}:
    \begin{equation}\label{eqn:identifyingaux1}
    \mathscr{Q}^*_{\mathrm{PT}_{\varphi}}:\mathscr{Q}^*_{\unit^\Sp_\varphi}\to\mathscr{Q}^*_{\Gamma_\varphi}=\Omega^{*,{\rm ho}}_\varphi.
    \end{equation}
Now, Proposition \ref{prop:spatialgysin1} implies
    \begin{equation}
    \mathscr{Q}^*_{\unit^\Sp_\varphi}=\Omega^*_\varphi.
    \end{equation}
In particular, \cite[Proposition 1.10]{quillenElementaryProofsResults1971} shows $\mathscr{Q}^*_{\mathrm{PT}_{\varphi}}$ is precisely the Pontryagin-Thom construction via the uniqueness statement. We conclude by observing 
    \begin{equation}
    \mathscr{Q}^*_{\mathrm{PT}_{\varphi}}=\mathrm{PT}_{\varphi}(\bbR^0,0)_*.
    \end{equation}
\end{proof}

Finally, we prove Theorem \ref{thm:main2} and Corollary \ref{cor:main}; we restate them for convenience. 

\begin{thm}\label{thm: gysin sheaves monadic}
The lax symmetric monoidal functor 
    \begin{equation}
    \Hom(\unit^{\Sp}_{\varphi},-):\Gys_{\varphi}^{\bbR}(\Sm;\Sp)\to\Sp
    \end{equation}
induces an equivalence of symmetric monoidal $\infty$-categories
    \begin{equation}
    \Hom(\unit_{\varphi},-):\Gys_{\varphi}^{\bbR}(\Sm;\Sp)\to\Mod_{\operatorname{End}(\unit_{\varphi})}(\Sp). 
    \end{equation}
Consequently, $\mathrm{PT}_{\varphi}$ induces a symmetric monoidal equivalence 
    \begin{equation}
    \Gys_{\varphi}^{\bbR}(\Sm;\Sp)\simeq\Mod_{\rmM\varphi}(\Sp). 
    \end{equation}
\end{thm}

\begin{proof}
    For the first claim, by \cite[Proposition 7.1.2.7]{lurieHigherAlgebra2017}, it is enough to show $\unit_{\varphi}$ is compact and the functor $\Hom(\unit_{\varphi},-)$ is conservative. 
    To begin, note there is an equivalence  
    \begin{equation}
    \Hom(\unit_{\varphi},\m{F}) \simeq \m{F}(\bbR^0,0)
    \end{equation}
    which is natural in $\m{F} \in \Gys_{\varphi}^{\bbR}(\Sm;\Sp)$. 
    Therefore, by Dugger's theorem, the functor in question may be identified with restriction along $\Sm^{\op} \rightarrow \TrCor_{\varphi}^{\ex}(\Sm)$:
    \begin{equation}
    (-)|_{\Sm^{\op}} : \Gys_{\varphi}^{\bbR}(\Sm;\Sp) \rightarrow \Shv_{\bbR}(\Sm;\Sp). 
    \end{equation}
    As $(-)|_{\Sm^{\op}}$ preserves small limits and colimits, it follows that $\unit_{\varphi}$ is compact. 
    Now, let $\m{F}$ be a homotopy invariant $\varphi$-Gysin sheaf and assume $\m{F}(\bbR^{0},0) \simeq 0$; we aim to show $\m{F} \simeq 0$. 
    To proceed, let $(X,d_{X}) \in \TrCor_{\varphi}^{\ex}(\Sm)$ and consider the decomposition of $X$ into path components: $X = \sqcup_{i\in I} X_{i}$; note, $d_X\vert_{X_i}\equiv d_i\in\bbZ$.
    As the family of maps
    \begin{equation}
        \big\{(X_{i},d_{i}) \rightarrow (X,d_{X})\big\}_{i\in I}
    \end{equation}
    determines a covering family in $\TrCor_{\varphi}^{\ex}(\Sm)$ and $\m{F}$ is a sheaf, the canonical map 
    \begin{equation}
    \m{F}(X,d_{X}) \rightarrow \prod_{i\in I}\m{F}(X_{i},d_{i})
    \end{equation}
    is an equivalence. 
    Therefore, by Corollary~\ref{cor: shifting = suspension}, we have an equivalence
    \begin{equation}
        \m{F}(X,d_{X}) \simeq \prod_{i \in I}\Sigma^{d_{i}}\m{F}(X_{i},0), 
    \end{equation}
    and the conservativity of $\Hom(\unit_{\varphi},-)$ follows by another application of Dugger's theorem.    
    The equivalence with $\rmM\varphi$-modules is now immediate from Theorem~\ref{thm: PT is an equivalence}.
\end{proof}

\appendix
\section{Grothendieck topologies, new from old}\label{appendix: sites}
This appendix addresses a few elementary results which allow us to construct Grothendieck topologies in some cases of interest.
As a point of order, by \cite[Remark 6.2.2.3]{lurieHigherToposTheory2009}, a Grothendieck topology on an \category $\m{C}$ reduces to the notion of a Grothendieck topology on its homotopy category $\mathrm{h}\m{C}$; we will make repeated use of this observation in our arguments below.
For more details, we direct the reader to \cite[Chapter III]{maclaneSheavesGeometryLogic1994} and \cite[Chapter 6.2]{lurieHigherToposTheory2009}.

\begin{defin}\label{defin: basis for a topology}
    Let $\m{C}$ be an \category. 
    A \textit{basis for a Grothendieck topology} on $\m{C}$ is a function $\m{B}$ which assigns to each object $X \in \m{C}$ a collection of subsets of $\mathrm{Ob}(\mathrm{h}\m{C}_{/X})$ such that the following holds.
    \begin{enumerate}
        \item If $f :X \rightarrow X'$ is an equivalence, then $\{f:X' \rightarrow X\} \in \m{B}(X)$.
        \item If $\{V_{j} \rightarrow Y\}_{j\in J} \in \m{B}(Y)$, then, for any morphism $f:X \rightarrow Y$, there exists a covering family $\{U_{i} \rightarrow X\}_{i\in I} \in \m{B}(X)$ with the property that the composite $U_{i} \rightarrow X \rightarrow Y$ factors through some $Y_{j} \rightarrow Y$. 
        \item If $\{U_{i} \rightarrow X\}_{i\in I} \in \m{B}(X)$ and if for each $i \in I$ we have $\{U_{ij} \rightarrow U_{i}\}_{j\in J_{i}} \in K(U_{i})$, then $\{U_{ij} \rightarrow X\}_{(i,j) \in \cup_{i \in I}J_{i}} \in \m{B}(X)$.
    \end{enumerate}
    The basis $\m{B}$ induces a topology by declaring a sieve $S \subset \m{C}_{/X}$ to be covering if and only if it is generated by a covering family $\mathcal{U} \in \m{B}(X)$.     
\end{defin}

\begin{rem}
    Notice that $\{X\xrightarrow{\sim}Y\}$ in $\m{C}$ generates the sieve $\m{C}_{/X}$, as does the family $\{\id_{X} :X \rightarrow X\}$. 
    Therefore, condition (1) can be replaced with the condition that $\{\id_{X}:X \rightarrow X\} \in \m{B}(X)$.
\end{rem}

\begin{rem}
    Any topological site has a Grothendieck topology determined by open embeddings.
\end{rem}

\begin{defin}
    Let $p : \m{D} \rightarrow \m{C}$ be a functor of \categories where $(\m{C},\tau)$ is a site with basis $\m{B}$. 
    We say that $p$ is \textit{Cartesian over $\m{B}$} provided that: for every $f:U \rightarrow X$ belonging to a covering family $\mathcal{U} \in \m{B}(X)$ and every object $\widetilde{X}$ with $p(\widetilde{X})\simeq X$, there exists a $p$-Cartesian morphism $\tilde{f} : \widetilde{U} \rightarrow \widetilde{X}$ lifting $f$. 
\end{defin}

\begin{lem}\label{lem: Cartesian lift of site}
    Let $(\m{C},\tau)$ be a Grothendieck site with basis $\m{B}$ and let $p:\m{D} \rightarrow \m{C}$ be a functor which is Cartesian over $\m{B}$. 
    Given $D \in \m{D}$ and a covering family $\m{U}$ of $X = p(D)$, let $\mathcal{U}_{D}$ denote collection of $p$-Cartesian lifts of morphisms in $\m{U}$ with target $D$. 
    For each $D \in \m{D}$, let $p^{\ast}\m{B}(D)$ denote the set of families $\mathcal{U}_{D}$, where $\mathcal{U} \in \m{B}\big(p(D)\big)$. 
    The function $D \mapsto p^{\ast}\m{B}(D)$ is a basis for a Grothendieck topology, and the functor $p: \m{D} \rightarrow \m{C}$ is continuous. 
\end{lem}

\begin{proof}
    The continuity claim follows immediately once we verify $p^{\ast}\m{B}$ is a basis for a Grothendieck topology. 
    Furthermore, we may work with $\mathrm{h}\m{D}$ throughout in place of $\m{D}$ while checking the axioms for a basis above. 
    \begin{enumerate}
        \item Equivalences in $\m{D}$ are $p$-Cartesian lifts of equivalences in $\m{C}$, so any equivalence $D' \rightarrow D$ determines a covering family in $p^{\ast}\m{B}(D)$.
        \item Let $\phi: D \rightarrow D'$ in $\m{D}$ with $f = p(\phi):X \rightarrow X'$ and let $\mathcal{U}_{D'} \in p^{\ast}\m{B}(D')$.
        As $\m{B}$ is a basis, we may choose a covering family $f^{\ast}\m{U} \in \m{B}(X)$, from which we may form $(f^{\ast}\mathcal{U})_{D}$.
        The Cartesian lifting property easily implies $(f^{\ast}\mathcal{U})_{D}$ satisfies (2) from Definition~\ref{defin: basis for a topology}. 
        \item Let $\{D_{i} \rightarrow D\}_{i \in I} \in p^{\ast}\m{B}(D)$ and, for each $i \in I$, let $\{D_{ij} \rightarrow D_{i}\}_{j\in J_{i}} \in p^{\ast}\m{B}(D_{i})$. 
        The fact that $p$-Cartesian arrows are closed under composition implies $\{D_{ij} \rightarrow D\}_{(i,j)\in \cup_{i\in I}J_{i}} \in p^{\ast}\m{B}(D)$. 
    \end{enumerate}

\end{proof}

    

\begin{rem}
    For any structure functor $F$ on $\DSm$, the right fibration $\pi_{F} : \DSm_{F} \rightarrow \DSm$ determines a Grothendieck topology on $\DSm_{F}$. 
    In fact, $\DSm_{F}$ is a topological site by \cite[2.8.3.14]{pardon-DSm}.
\end{rem}

\begin{defin}\label{defin: fine site}
    Let $\m{C}$ be a site generated by a basis $\m{B}$. 
    We say the basis $\m{B}$ is \textit{fine} provided that, for any morphism $f : X \rightarrow Y$ in $\m{C}$ and for any covering family $\mathcal{U}$ of $X$, there exists a covering family $\mathcal{V}$ of $Y$ together with a refinement $f^{\ast}\mathcal{V} \rightarrow \mathcal{U}$.   
    By abuse of notation, we say the Grothendieck site $\m{C}$ is fine. 
\end{defin}

\begin{example}
    $\Sm$ and $\DSm$ are fine Grothendieck sites.
    In fact, any topological $\infty$-site determines a fine Grothendieck site.
    Both the Nisnevich and \'{e}tale topologies are fine sites by \cite[Lemma 6.16]{motivic-coh-lectures}. 
\end{example}

\begin{lem}\label{lem: Cor Grothendieck topology}
    Let $\m{C}$ be an \category which admits pullbacks, $(\m{C};\m{C}_{L},\m{C}_{R})$ an adequate triple, and assume $\m{C}_{R}$ is a fine Grothendieck site with basis $\m{B}$, such that, for each covering family $\{U_{i} \rightarrow X\}_{i \in I} \in \m{B}(X)$ and for each morphism $W \rightarrow X$ in either $\m{C}_{L}$ or $\m{C}_{R}$, we have $\{U_{i}\times_{X}W \rightarrow W\}_{i \in I} \in \m{B}(W)$. 
    For each $\mathcal{U} \in \m{B}(X)$, let 
    \begin{equation}
        \mathcal{U}_{\Cor} \equiv \{U_{i} \xleftarrow{\id} U_{i} \rightarrow X\}_{i \in I}
    \end{equation}
    and let $\m{B}_{\Cor}(X)$ denote the collection of families of the form above. 
    We have that $\m{B}_{\Cor}$ is a basis for a Grothendieck topology on $\Cor(\m{C};\m{C}_{L},\m{C}_{R})$, where a sieve is covering if and only if it is generated by a family of the form $\mathcal{U}_{\Cor}$. 
\end{lem}

\begin{proof}
    We check the conditions from \ref{defin: basis for a topology}. 
    \begin{enumerate}
    \item It is clear that, for any $X\in\m{C}$, the identity correspondence is of the form $\mathcal{U}_{\Cor}$.
    \item Let $\mathcal{V} \in \m{B}(Y)$ and denote by $A$ the correspondence
    \begin{equation}
        Y\xleftarrow{f} W \xrightarrow{g} X.
    \end{equation}
    It will be enough to prove that there exists $\mathcal{U} \in \m{B}(X)$ with the property that, for each $U_{i} \rightarrow X \in \mathcal{U}$, there exists a commutative diagram 
    \begin{equation}
\begin{tikzcd}
	{U_{i}} & {W\times_{Y}U_{i}} & {V_{j}} \\
	X & W & Y
	\arrow[from=1-1, to=2-1]
	\arrow[from=1-2, to=1-1]
	\arrow[from=1-2, to=1-3]
	\arrow[from=1-3, to=2-3]
	\arrow[from=2-2, to=2-1]
	\arrow[from=2-2, to=2-3]
\end{tikzcd}
    \end{equation}
    with $V_{j} \rightarrow Y \in \mathcal{V}$. 
    By pulling back $\mathcal{V}$ along $f$, the fineness of $\m{B}$ allows us to refine $g^{\ast}\mathcal{V}$ by a covering $f^{\ast}\mathcal{U}$ for some $\mathcal{U} \in \m{B}(X)$. 
    \item This claim is clear from the definitions, as the left legs in $\m{U}_{\Cor}$ are identity morphisms.     
    \end{enumerate}
    Therefore, $\m{B}_{\Cor}$ is a basis for a Grothendieck topology whose covering sieves are described as claimed.
\end{proof}

\begin{cor}\label{cor: sheaves on Cor(C;L,R)}
Let $(\m{C};\m{C}_{L},\m{C}_{R})$ be an adequate triple, $\m{B}$ a basis for a Grothendieck topology on $\m{C}_{R}$ satisfying the hypotheses of Lemma~\ref{lem: Cor Grothendieck topology}, and $\m{E}$ be a presentable \category. 
    \begin{enumerate}
        \item The functor $i_{R} : \m{C}_{R} \rightarrow \Cor(\m{C};\m{C}_{L},\m{C}_{R})$ is continuous. 
        \item An $\m{E}$-valued presheaf $\m{F} : \Cor(\m{C};\m{C}_{L},\m{C}_{R})^{\op} \rightarrow \m{E}$ is a sheaf if and only if $i_{R}^{\ast}\m{F}$ is a sheaf. 
    \end{enumerate}
\end{cor}

\begin{proof}
    (1) is clear from the definitions. 
    The claim (2) follows from a cofinality argument.
\end{proof}



\bibliography{References}{}

@article{CJS95,
	author = {Cohen, Ralph and Jones, John D.S. and Segal, Graeme B.},
	journal = {The Floer Memorial Volume, Progress in Mathematics},
	pages = {297-325},
	title = {{Floer's infinite dimensional Morse theory and homotopy theory}},
	volume = {133},
	year = {1995}}

@article{Tho52,
	author = {Thom, Ren{\'e}},
	journal = {Annales scientifiques de l'\'Ecole Normale Sup\'erieure},
    number = {3},
	pages = {109-182},
	title = {{Espaces fibr\'es en sph\`eres et carr\'es de Steenrod}},
    volume = {69}, 
	year = {1952}}

@article{quillenElementaryProofsResults1971,
  title = {Elementary Proofs of Some Results of Cobordism Theory Using {{Steenrod}} Operations},
  author = {Quillen, Daniel},
  year = {1971},
  journal = {Advances in Mathematics},
  volume = {7},
  number = {1},
  pages = {29--56},
  issn = {0001-8708},
  doi = {10.1016/0001-8708(71)90041-7},
  urldate = {2025-08-29}
}

@unpublished{diff-coh-book,
  title = {Differential {{cohomology}}: {{categories}}, {{characteristic classes}}, and {{connections}}},
  author = {Amabel, Araminta and Debray, Arun and Haine, Peter J.},
  year = {2023},
  note = {To appear in Cambridge Studies in Advanced Mathematics, currently available at \url{https://peterjhaine.github.io/files/diffcoh.pdf}}
}

@unpublished{AB24,
	author = {Abouzaid, Mohammed and Blumberg, Andrew J.},
	note = {arXiv:2404.03193},
	title = {{Foundation of Floer homotopy theory I: flow categories}},
	year = {2024}}

@unpublished{BP26,
	author = {Blakey, Kenneth and Porcelli, Noah},
	note = {arXiv:2605.06620},
	title = {{Bulk-deformations, Floer complex bordism, and Grothendieck-Riemann-Roch}},
	year = {2026}}

@unpublished{HO26,
	author = {Hedenlund, Alice and Oldervoll, Trygve P.},
	note = {arXiv:2603.29576},
	title = {{Structured flow categories and twisted presheaves}},
	year = {2026}}

@article{barwick-glasman-shahii,
    AUTHOR = {Barwick, Clark and Glasman, Saul and Shah, Jay},
     TITLE = {Spectral {M}ackey functors and equivariant algebraic
              {$K$}-theory, {II}},
   JOURNAL = {Tunisian Journal of Mathematics},
  FJOURNAL = {Tunisian Journal of Mathematics},
    VOLUME = {2},
      YEAR = {2020},
    NUMBER = {1},
     PAGES = {97--146},
      ISSN = {2576-7658,2576-7666},
   MRCLASS = {19D99 (55P91)},
  MRNUMBER = {3933393},
MRREVIEWER = {Marco\ Varisco},
       DOI = {10.2140/tunis.2020.2.97},
       URL = {https://doi-org.revproxy.brown.edu/10.2140/tunis.2020.2.97},
}

@book{lurieHigherToposTheory2009,
    AUTHOR = {Lurie, Jacob},
     TITLE = {Higher topos theory},
    SERIES = {Annals of Mathematics Studies},
    VOLUME = {170},
 PUBLISHER = {Princeton University Press, Princeton, NJ},
      YEAR = {2009},
     PAGES = {xviii+925},
      ISBN = {978-0-691-14049-0; 0-691-14049-9},
   MRCLASS = {18-02 (18B25 18E35 18G30 18G55 55U40)},
  MRNUMBER = {2522659},
MRREVIEWER = {Mark\ Hovey},
       DOI = {10.1515/9781400830558},
       URL = {https://doi-org.revproxy.brown.edu/10.1515/9781400830558},
}

@unpublished{lurieHigherAlgebra2017,
    author = {Lurie, Jacob},
    title = {Higher {{algebra}}},
    note = {Unpublished manuscript, available at \url{https://www.math.ias.edu/~lurie/}},
    year = {2017}
}

@unpublished{lurienotes,
  title = {Lecture 7: Simplicial Spaces},
  author = {Lurie, Jacob},
  year = {2014},
  note = {Notes from Math 281: Manifold Topology and Algebraic K-theory, Harvard University Fall 2014, available at \url{https://www.math.ias.edu/~lurie/}},
}

@article{barwick,
	author = {Barwick, Clark},
	journal = {Advances in Mathematics},
    pages = {646-727},
	title = {{Spectral Mackey functors and equivariant algebraic $K$-theory (I)}},
	volume = {304},
	year = {2017}}

@article{volpe,
	author = {Volpe, Marco},
	journal = {Journal of Topology},
    pages = {e70050},
	title = {{The six operations in topology}},
	volume = {18},
    number ={4},
	year = {2025}}

@unpublished{6FF,
	author = {Heyer, Claudius and Mann, Lucas},
	note = {arXiv:2410.13038},
	title = {{6-functor formalisms and smooth representations}},
	year = {2024}}

@article{dugger,
    AUTHOR = {Dugger, Daniel},
     TITLE = {Universal homotopy theories},
   JOURNAL = {Advances in Mathematics},
  FJOURNAL = {Advances in Mathematics},
    VOLUME = {164},
      YEAR = {2001},
    NUMBER = {1},
     PAGES = {144--176},
      ISSN = {0001-8708,1090-2082},
   MRCLASS = {18G55 (18E20 55U10)},
  MRNUMBER = {1870515},
MRREVIEWER = {Manuel\ Bullejos Lorenzo},
       DOI = {10.1006/aima.2001.2014},
       URL = {https://doi-org.revproxy.brown.edu/10.1006/aima.2001.2014},
}

@unpublished{scholze6FF,
      title={Six-Functor Formalisms}, 
      author={Peter Scholze},
      year={2026},
      note ={arXiv:2510.26269} 
}

@book{motivic-coh-lectures,
    AUTHOR = {Mazza, Carlo and Voevodsky, Vladimir and Weibel, Charles},
     TITLE = {Lecture notes on motivic cohomology},
    SERIES = {Clay Mathematics Monographs},
    VOLUME = {2},
 PUBLISHER = {American Mathematical Society, Providence, RI; Clay
              Mathematics Institute, Cambridge, MA},
      YEAR = {2006},
     PAGES = {xiv+216},
      ISBN = {978-0-8218-3847-1; 0-8218-3847-4},
   MRCLASS = {14F42 (19E15)},
  MRNUMBER = {2242284},
MRREVIEWER = {Thomas\ Geisser},
}

@unpublished{pardon-DSm,
      title={Derived moduli spaces of pseudo-holmorphic curves}, 
      author={John Pardon},
      year={2025},
      note={Unpublished manuscript, available at \url{https://johnpardon.com/holomorphiccurves-2025-04.pdf}}
}

@unpublished{keenan-peroux,
	author = {Keenan, Liam and P\'eroux, Maximilien},
	note = {arXiv:2510.18961},
	title = {{On products of skeleta}},
	year = {2025}
}

@unpublished{saunier-weight-heart,
      title={{A theorem of the heart for K-theory of endomorphisms}}, 
      author={Victor Saunier},
      year={2023},
      note={arXiv:2311.13836}
}

@unpublished{lurieSAG,
  title = {Spectral Algebraic Geometry},
  author = {Lurie, Jacob},
  year = {2018},
  note = {Unpublished manuscript, available at \url{https://www.math.ias.edu/~lurie/}}
}

@article{blumbergUniversalCharacterizationHigher2013,
  title = {A Universal Characterization of Higher Algebraic {{K-theory}}},
  author = {Blumberg, Andrew J. and Gepner, David and Tabuada, Gon{\c c}alo},
  year = 2013,
  journal = {Geometry \& Topology},
  volume = {17},
  number = {2},
  pages = {733--838},
  publisher = {Mathematical Sciences Publishers}
}

@article{antieau-gepner-brauer,
    AUTHOR = {Antieau, Benjamin and Gepner, David},
     TITLE = {Brauer groups and \'etale cohomology in derived algebraic
              geometry},
   JOURNAL = {Geometry \& Topology},
  FJOURNAL = {Geometry \& Topology},
    VOLUME = {18},
      YEAR = {2014},
    NUMBER = {2},
     PAGES = {1149--1244},
      ISSN = {1465-3060,1364-0380},
   MRCLASS = {14F22 (13D09 14D20 18G55 55P43)},
  MRNUMBER = {3190610},
MRREVIEWER = {Timothy\ J.\ Ford},
       DOI = {10.2140/gt.2014.18.1149},
       URL = {https://doi-org.brown.idm.oclc.org/10.2140/gt.2014.18.1149},
}

@unpublished{cnossenUniversalitySpan2categories2026,
  title = {{Universality of span 2-categories and the construction of 6-functor formalisms}},
  author = {Cnossen, Bastiaan and Lenz, Tobias and Linskens, Sil},
  year = {2026},
  note = {arXiv:2505.19191}
}

@article{antolin-camarenaSimpleUniversalProperty2019,
  title = {A Simple Universal Property of {{Thom}} Ring Spectra},
  author = {{Antol{\'i}n-Camarena}, Omar and Barthel, Tobias},
  year = 2019,
  journal = {Journal of Topology},
  volume = {12},
  number = {1},
  pages = {56--78},
  issn = {1753-8424},
  doi = {10.1112/topo.12084},
  urldate = {2026-04-02},
  copyright = {\copyright{} 2018 London Mathematical Society},
  langid = {english}
}

@article{gepnerUniversalityMultiplicativeInfinite2015,
  title = {Universality of Multiplicative Infinite Loop Space Machines},
  author = {Gepner, David and Groth, Moritz and Nikolaus, Thomas},
  year = 2015,
  journal = {Algebraic \& Geometric Topology},
  volume = {15},
  number = {6},
  pages = {3107--3153},
  publisher = {MSP},
  issn = {1472-2747, 1472-2739},
  doi = {10.2140/agt.2015.15.3107},
  urldate = {2026-04-16},
  langid = {english}
}

@article{berwick-evansClassifyingSpacesInfinitysheaves2024,
  title = {Classifying Spaces of Infinity-Sheaves},
  author = {{Berwick-Evans}, Daniel and {Boavida de Brito}, Pedro and Pavlov, Dmitri},
  year = 2024,
  journal = {Algebraic \& Geometric Topology},
  volume = {24},
  number = {9},
  pages = {4891--4937},
  publisher = {Mathematical Sciences Publishers},
  issn = {1472-2739},
  doi = {10.2140/agt.2024.24.4891},
  urldate = {2026-04-15},
  langid = {english}
}

@inproceedings{doldGeometricCobordismFixed1978,
  title = {Geometric Cobordism and the Fixed Point Transfer},
  booktitle = {Algebraic {{Topology}}},
  author = {Dold, Albrecht},
  editor = {Hoffman, Peter and Piccinini, Renzo A. and Sjerve, Denis},
  year = 1978,
  pages = {32--87},
  publisher = {Springer},
  address = {Berlin, Heidelberg},
  doi = {10.1007/BFb0064688},
  isbn = {978-3-540-35737-7},
  langid = {english}
}

@misc{kerodon,
    title        = {Kerodon},
    author       = {Jacob Lurie},
    howpublished = {\url{https://kerodon.net}},
    year         = {2026},
  }

@inproceedings{quinn,
  title = {Assembly maps in bordism-type theories},
  booktitle = {Novikov conjectures, index theorems and rigidity, Volume 1},
  author = {Quinn, Frank},
  editor = {},
  pages = {201-271},
  year = 1995,
  publisher = {Cambridge University Press},
}

@phdthesis{quinn-thesis,
    author = {Quinn, Frank},
    title = {A geometric formulation of surgery},
    school = {Princeton University},
    year = {1970}
}

@article{thomQuelquesProprietesGlobales1954,
  title = {{Quelques propri\'et\'es globales des vari\'et\'es diff\'erentiables}},
  author = {Thom, Ren{\'e}},
  year = 1954,
  journal = {Commentarii Mathematici Helvetici},
  volume = {28},
  number = {1},
  pages = {17--86},
  issn = {1420-8946},
  doi = {10.1007/BF02566923},
  urldate = {2025-11-17},
  langid = {french}
}

@article{quillenFormalGroupLaws1969,
  title = {On the Formal Group Laws of Unoriented and Complex Cobordism Theory},
  author = {Quillen, Daniel},
  year = 1969,
  journal = {Bulletin of the American Mathematical Society},
  volume = {75},
  number = {6},
  pages = {1293--1298},
  publisher = {American Mathematical Society},
  issn = {0002-9904, 1936-881X},
  urldate = {2025-08-29}
}

@article{novikov-Thom,
    AUTHOR = {Novikov, Sergei P.},
     TITLE = {Homotopy properties of {T}hom complexes},
   JOURNAL = {Matematicheski\u i\ Sbornik. Novaya Seriya},
  FJOURNAL = {Matematicheski\u i\ Sbornik. Novaya Seriya},
    VOLUME = {57(99)},
      YEAR = {1962},
     PAGES = {407--442},
      ISSN = {0368-8666},
   MRCLASS = {55.50 (55.60)},
  MRNUMBER = {157381},
MRREVIEWER = {J.\ W.\ Jaworowski},
}

@article{novikov-cobordism,
    AUTHOR = {Novikov, Sergei P.},
     TITLE = {Methods of algebraic topology from the point of view of
              cobordism theory},
   JOURNAL = {Izvestiya Akademii Nauk SSSR. Seriya Matematicheskaya},
  FJOURNAL = {Izvestiya Akademii Nauk SSSR. Seriya Matematicheskaya},
    VOLUME = {31},
      YEAR = {1967},
     PAGES = {855--951},
      ISSN = {0373-2436},
   MRCLASS = {55.52 (57.00)},
  MRNUMBER = {221509},
MRREVIEWER = {A.\ Liulevicius},
}

@article{ABGHR,
    AUTHOR = {Ando, Matthew and Blumberg, Andrew J. and Gepner, David and
              Hopkins, Michael J. and Rezk, Charles},
     TITLE = {An {$\infty$}-categorical approach to {$R$}-line bundles, {$R$}-module {T}hom spectra, and twisted {$R$}-homology},
   JOURNAL = {Journal of Topology},
  FJOURNAL = {Journal of Topology},
    VOLUME = {7},
      YEAR = {2014},
    NUMBER = {3},
     PAGES = {869--893},
      ISSN = {1753-8416,1753-8424},
   MRCLASS = {55P43 (55N20 55U40)},
  MRNUMBER = {3252967},
MRREVIEWER = {Tyler\ D.\ Lawson},
       DOI = {10.1112/jtopol/jtt035},
       URL = {https://doi-org.brown.idm.oclc.org/10.1112/jtopol/jtt035},
}

@article{haugsengTwovariableFibrationsFactorisation2023,
  title = {Two-Variable Fibrations, Factorisation Systems and $\infty$-Categories of Spans},
  author = {Haugseng, Rune and Hebestreit, Fabian and Linskens, Sil and Nuiten, Joost},
  year = 2023,
  journal = {Forum of Mathematics, Sigma},
  volume = {11},
  pages = {e111},
  issn = {2050-5094},
  doi = {10.1017/fms.2023.107},
  urldate = {2026-03-15},
  langid = {english}
}

@article{gepnerLaxColimitsFree2017,
  title = {Lax Colimits and Free Fibrations in {$\infty$}-Categories},
  author = {Gepner, David and Haugseng, Rune and Nikolaus, Thomas},
  year = 2017,
  journal = {Documenta Mathematica},
  volume = {22},
  pages = {1255--1266}
}

@book{maclaneSheavesGeometryLogic1994,
  title = {Sheaves in {{geometry}} and {{logic}}: {{a first introduction}} to {{topos theory}}},
  shorttitle = {Sheaves in {{geometry}} and {{logic}}},
  author = {Mac Lane, Saunders and Moerdijk, Ieke},
  year = 1994,
  series = {Universitext},
  publisher = {Springer},
  address = {New York, NY},
  doi = {10.1007/978-1-4612-0927-0},
  urldate = {2025-06-25},
  copyright = {https://www.springernature.com/gp/researchers/text-and-data-mining},
  isbn = {978-0-387-97710-2 978-1-4612-0927-0},
  langid = {english}
}

@article{linskensGlobalHomotopyTheory2025,
  title = {Global Homotopy Theory via Partially Lax Limits},
  author = {Linskens, Sil and Nardin, Denis and Pol, Luca},
  year = 2025,
  journal = {Geometry \& Topology},
  volume = {29},
  number = {3},
  pages = {1345--1440},
  publisher = {Mathematical Sciences Publishers},
  issn = {1364-0380},
  doi = {10.2140/gt.2025.29.1345},
  urldate = {2025-08-07},
  langid = {english}
}

@unpublished{nikolausStable$infty$OperadsMultiplicative2016,
    author = {Nikolaus, Thomas},
    note = {arXiv:1608.02901},    
    title = {{Stable $\infty$-operads and the multiplicative Yoneda lemma}},
    year = {2016}
}

@book{Bredon-sheaf-theory,
    AUTHOR = {Bredon, Glen E.},
     TITLE = {Sheaf theory},
    SERIES = {Graduate Texts in Mathematics},
    VOLUME = {170},
   EDITION = {Second},
 PUBLISHER = {Springer-Verlag, New York},
      YEAR = {1997},
     PAGES = {xii+502},
      ISBN = {0-387-94905-4},
   MRCLASS = {55N30 (18F20 54B40 55-02)},
  MRNUMBER = {1481706},
       DOI = {10.1007/978-1-4612-0647-7},
       URL = {https://doi-org.brown.idm.oclc.org/10.1007/978-1-4612-0647-7},
}

@book{Kashiwara-Schapira,
    AUTHOR = {Kashiwara, Masaki and Schapira, Pierre},
     TITLE = {Sheaves on manifolds},
    SERIES = {Grundlehren der mathematischen Wissenschaften [Fundamental
              Principles of Mathematical Sciences]},
    VOLUME = {292},
      NOTE = {With a chapter in French by Christian Houzel},
 PUBLISHER = {Springer-Verlag, Berlin},
      YEAR = {1990},
     PAGES = {x+512},
      ISBN = {3-540-51861-4},
   MRCLASS = {58G07 (18F20 32C38 35A27)},
  MRNUMBER = {1074006},
MRREVIEWER = {Michael\ M.\ Kapranov},
       DOI = {10.1007/978-3-662-02661-8},
       URL = {https://doi.org/10.1007/978-3-662-02661-8},
}

@book{weibel-k-book,
    AUTHOR = {Weibel, Charles A.},
     TITLE = {The {$K$}-book},
    SERIES = {Graduate Studies in Mathematics},
    VOLUME = {145},
      NOTE = {An introduction to algebraic $K$-theory},
 PUBLISHER = {American Mathematical Society, Providence, RI},
      YEAR = {2013},
     PAGES = {xii+618},
      ISBN = {978-0-8218-9132-2},
   MRCLASS = {19-01},
  MRNUMBER = {3076731},
MRREVIEWER = {L.\ N.\ Vaserstein},
       DOI = {10.1090/gsm/145},
       URL = {https://doi.org/10.1090/gsm/145},
}

@article{laures-mcclure-comm,
    AUTHOR = {Laures, Gerd and McClure, James E.},
     TITLE = {Commutativity properties of {Q}uinn spectra},
   JOURNAL = {Proceedings of the Royal Society of Edinburgh. Section A.
              Mathematics},
  FJOURNAL = {Proceedings of the Royal Society of Edinburgh. Section A.
              Mathematics},
    VOLUME = {154},
      YEAR = {2024},
    NUMBER = {6},
     PAGES = {1848--1936},
      ISSN = {0308-2105,1473-7124},
   MRCLASS = {55P43},
  MRNUMBER = {4883993},
MRREVIEWER = {Daniel\ A.\ Ramras},
       DOI = {10.1017/prm.2024.119},
       URL = {https://doi-org.brown.idm.oclc.org/10.1017/prm.2024.119},
}

@article{laures-mcclure-mult,
    AUTHOR = {Laures, Gerd and McClure, James E.},
     TITLE = {Multiplicative properties of {Q}uinn spectra},
   JOURNAL = {Forum Mathematicum},
  FJOURNAL = {Forum Mathematicum},
    VOLUME = {26},
      YEAR = {2014},
    NUMBER = {4},
     PAGES = {1117--1185},
      ISSN = {0933-7741,1435-5337},
   MRCLASS = {55P43 (57R67 57R90)},
  MRNUMBER = {3228927},
MRREVIEWER = {Laurence\ R.\ Taylor},
       DOI = {10.1515/forum-2011-0086},
       URL = {https://doi-org.brown.idm.oclc.org/10.1515/forum-2011-0086},
}

@article{GMTW,
    AUTHOR = {Galatius, S{\o}ren and Tillmann, Ulrike and Madsen, Ib and
              Weiss, Michael},
     TITLE = {The homotopy type of the cobordism category},
   JOURNAL = {Acta Mathematica},
  FJOURNAL = {Acta Mathematica},
    VOLUME = {202},
      YEAR = {2009},
    NUMBER = {2},
     PAGES = {195--239},
      ISSN = {0001-5962,1871-2509},
   MRCLASS = {55P47 (55P15 55P42 57R75)},
  MRNUMBER = {2506750},
MRREVIEWER = {Richard\ John\ Steiner},
       DOI = {10.1007/s11511-009-0036-9},
       URL = {https://doi-org.brown.idm.oclc.org/10.1007/s11511-009-0036-9},
}

@article{ramziMONOIDALGROTHENDIECKCONSTRUCTION2026,
  title = {A {{monoidal Grothendieck construction for}} {$\infty$}-{{categories}}},
  author = {Ramzi, Maxime},
  year = 2026,
  journal = {Nagoya Mathematical Journal},
  volume = {261},
  pages = {e8},
  issn = {0027-7630, 2152-6842},
  doi = {10.1017/nmj.2025.10086},
  urldate = {2026-04-28},
  langid = {english}
}

@unpublished{lurie2009classificationtopologicalfieldtheories,
      title={On the Classification of Topological Field Theories}, 
      author={Jacob Lurie},
      note={arXiv:0905.0465},
      year={2009}
}
\bibliographystyle{alpha.bst}
\end{document}